%% file: L_Collocation_ODE_mcom-l-template.tex
\documentclass{amsart}

\usepackage{graphicx}

\usepackage{hyperref}
\usepackage{glossaries}
\makeglossaries
\input{glossary_entries}
\usepackage[capitalise,noabbrev]{cleveref}
\usepackage{xcolor}
\definecolor{mod}{rgb}{0.0,0.0,0.0}     
\def\mod{\textcolor{mod}}

\usepackage{scalerel}
\usepackage{tikz}
\usetikzlibrary{svg.path}
\input{ORCID_logo.tex}

\copyrightinfo{}{\quad}

\newtheorem{theorem}{Theorem}[section]
\newtheorem{lemma}[theorem]{Lemma}
\newtheorem{corollary}[theorem]{Corollary} 
\newtheorem{proposition}[theorem]{Proposition} 

\theoremstyle{definition}
\newtheorem{definition}[theorem]{Definition}
\newtheorem{example}[theorem]{Example}

\newtheorem{assumption}[theorem]{Assumption}

\theoremstyle{remark}
\newtheorem{remark}[theorem]{Remark}

\numberwithin{equation}{section}

\begin{document}

\def\d{\mathrm{d}}
\def\D{\mathrm{D}}
\def\p{\partial}
\def\R{\mathbb{R}}
\def\N{\mathbb{N}}
\def\Z{\mathbb{Z}}
\def\rf{{{\mathrm{ref}}}}
\def\DEL{{{\mathrm{DEL}}}}
\def\EL{{{\mathrm{EL}}}}
\def\dist{{{\mathrm{dist}}}}

\title[Learning of Lagrangians as Gaussian fields]{Machine learning of continuous and discrete variational ODEs with convergence guarantee and uncertainty quantification}


\author{Christian Offen ${\protect \orcidicon{0000-0002-5940-8057}}$}
\address{Paderborn University, Department of Mathematics,Warburger Str. 100, 33098 Paderborn, Germany}
\email{christian.offen@uni-paderborn.de}
\thanks{The author acknowledges the Ministry for Culture and Science of the State of North Rhine-Westphalia (MKW) in Germany and computing time provided by the Paderborn Center for Parallel Computing (PC2).}



\date{}

\dedicatory{}

\begin{abstract}
	\input{abstract}
\end{abstract}

\maketitle

\input{Introduction_ODE}
\input{background_LagrangianDynamics}
\input{normalisation}
\input{method}

\section{Numerical experiments}\label{sec:Experiment}

\input{numerical_experiments_continuous.tex}

\input{ConvergenceAnalysis_ODE}

\input{convergence_rates}

\input{Summary}

\input{statements_ack_data}

\bibliographystyle{plainurl}
\bibliography{literature}

\appendix
\section*{Appendices}

\input{Appendix_Owhadi_Recap}

\input{Appendix_ODE}

\printglossary[title={List of notations}]\label{sec:glossary}

\end{document}

%% file: glossary_entries.tex
\newglossaryentry{L}
{
  name={\ensuremath{L}},
  description={Lagrangian function},
  sort=L
}

\newglossaryentry{LM}
{
  name={\ensuremath{L_{(M)}}},
  description={Inferred Lagrangian function based on all $M$ observed data points and regularisation conditions},
  sort=L
}

\newglossaryentry{LOmega0}
{
  name={\ensuremath{L_{\Omega_0}}},
  description={Inferred Lagrangian function based on observed data points contained in the finite set $\Omega_0$ and regularisation conditions},
  sort=L
}

\newglossaryentry{Lj}
{
  name={\ensuremath{L_{(j)}}},
  description={Inferred Lagrangian function based on the first $j$ observed data points in a sequence of observations and regularisation conditions},
  sort=L
}

\newglossaryentry{LRef}
{
  name={\ensuremath{L_{\mathrm{ref}}}},
  description={A true Lagrangian that defines a dynamical system. (There are many true Lagrangians for the same dynamical system.)},
  sort=L
}

\newglossaryentry{LInf}
{
  name={\ensuremath{L_{(\infty)}}},
  description={Inferred Lagrangian function based on an infinite sequence of observed data points and regularisation conditions},
  sort=L
}

\newglossaryentry{Ld}
{
  name={\ensuremath{L_d}},
  description={discrete Lagrangian function},
  sort=Ld
}

\newglossaryentry{LdM}
{
  name={\ensuremath{L_{d,(M)}}},
  description={Inferred discrete Lagrangian function based on all $M$ observed data points},
  sort=Ld
}

\newglossaryentry{Ldj}
{
  name={\ensuremath{L_{d,(j)}}},
  description={Inferred discrete Lagrangian function based on the first $j$ observed data points},
  sort=Ld
}

\newglossaryentry{EL}
{
  name={\ensuremath{\mathrm{EL}}},
  description={Euler--Lagrange operator, see \eqref{eq:ELContinuous}},
  sort=EL
}

\newglossaryentry{ELxhatj}
{
  name={\ensuremath{\mathrm{EL}_{\hat{x}^{(j)}}}},
  description={Euler--Lagrange operator composed with evaluation at a data point, see \eqref{eq:ELxhat}},
  sort=EL
}

\newglossaryentry{ELxhatj1}
{
  name={\ensuremath{\mathrm{EL}^1_{\hat{x}^{(j)}}}},
  description={Applies the operator $\mathrm{EL}_{\hat{x}^{(j)}}$ to the first input variable of a scalar valued binary function, see \eqref{eq:EL12}},
  sort=EL
}

\newglossaryentry{ELxhatj2}
{
  name={\ensuremath{\mathrm{EL}^2_{\hat{x}^{(j)}}}},
  description={Applies the operator $\mathrm{EL}_{\hat{x}^{(j)}}$ to the second input variable of a scalar valued binary function, see \eqref{eq:EL12}},
  sort=EL
}

\newglossaryentry{DEL}
{
  name={\ensuremath{\mathrm{DEL}}},
  description={Discrete Euler--Lagrange operator, see \eqref{eq:DELEQIntro}},
  sort=DEL
}

\newglossaryentry{DELxhatj}
{
  name={\ensuremath{\mathrm{DEL}_{\hat{x}^{(1)}}}},
  description={discrete Euler--Lagrange operator composed with evaluation at the data point $\hat{x}^{(1)}$},
  sort=DEL
}

\newglossaryentry{DELxhatj1}
{
  name={\ensuremath{\mathrm{DEL}^1_{\hat{x}^{(j)}}}},
  description={Applies the operator $\mathrm{DEL}_{\hat{x}^{(j)}}$ to the first input variable of a scalar valued binary function},
  sort=DEL
}

\newglossaryentry{DELxhatj2}
{
  name={\ensuremath{\mathrm{DEL}^2_{\hat{x}^{(j)}}}},
  description={Applies the operator $\mathrm{DEL}_{\hat{x}^{(j)}}$ to the second input variable of a scalar valued binary function},
  sort=DEL
}

\newglossaryentry{nabla1}
{
  name={\ensuremath{\nabla_1}},
  description={Gradient of a binary function with respect to its first (vector valued) input argument},
  sort=nabla
}

\newglossaryentry{nabla2}
{
  name={\ensuremath{\nabla_2}},
  description={Gradient of a binary function with respect to its second (vector valued) input argument},
  sort=nabla
}

\newglossaryentry{nabla12}
{
	name={\ensuremath{\nabla_{1,2}}},
	description={Gradient of a binary function with respect to its first (vector valued) and second input arguments, i.e.\ $\nabla_{1,2}f(a,b)=\frac{\p f}{\p a \p b}$},
	sort=nabla
}

\newglossaryentry{xhat}
{
  name={\ensuremath{\hat x}},
  description={Data point. In the continuous Lagrangian setting $\hat x = (x,\dot x, \ddot x)$. In the discrete Lagrangian setting $\hat x = (x_0,x_1, x_2)$ is a snapshot of a discrete trajectory of length 3},
  sort=x
}

\newglossaryentry{xbar}
{
  name={\ensuremath{\overline {x}}},
  description={Data point. In the continuous Lagrangian setting $\overline x = (x,\dot x)$. In the discrete Lagrangian setting $\overline x = (x_0,x_1)$ is a snapshot of a discrete trajectory of length 2},
  sort=x
}

\newglossaryentry{xb}
{
  name={\ensuremath{\overline {x}_b}},
  description={Point in $\Omega$ or $(\mathbb{R}^d)^2$ used in the formulation of regularisation conditions},
  sort=x
}

\newglossaryentry{pb}
{
  name={\ensuremath{p_b}},
  description={Vector used in the formulation of regularisation conditions. (Discrete) conjugate momentum at $\overline {x}_b$},
  sort=pb
}

\newglossaryentry{cb}
{
  name={\ensuremath{c_b}},
  description={Scalar used in the formulation of regularisation conditions. Value of (discrete) Lagrangian at $\overline {x}_b$},
  sort=cb
}

\newglossaryentry{var}
{
  name={\ensuremath{\mathrm{var}}},
  description={Variance of a random variable},
  sort=var
}

\newglossaryentry{Ham}
{
  name={\ensuremath{\mathrm{Ham}}},
  description={Assigns a Hamiltonian function to a continuous Lagrangian, see \eqref{eq:HamOperator}. The corresponding symplectic structure is obtained using the operator Sympl},
  sort=ham
}

\newglossaryentry{Sympl}
{
  name={\ensuremath{\mathrm{Sympl}}},
  description={Assigns a symplectic 2-form to a continuous or discrete Lagrangian, see \eqref{eq:SymplecticStructure} or \eqref{eq:SymplDiscrete}, respectively},
  sort=sympl
}

\newglossaryentry{Mm}
{
  name={\ensuremath{\mathrm{Mm}}},
  description={Assigns an expression for the conjugate momentum to a continuous Lagrangian, see \eqref{eq:MmDef}. When a lower index $\overline x$ is given, we consider composition with the evaluation functional at $\overline x$, see \eqref{eq:MmEvaluate}. When an upper index $1$ or $2$ is provided, the operator acts on the first or second input argument of a binary function, respectively},
  sort=Mm
}

\newglossaryentry{evxb}
{
  name={\ensuremath{\mathrm{ev}_{\overline{x}_b}}},
  description={Evaluation functional at point $\overline{x}_b$, see \eqref{eq:evalFunc}},
  sort=ev
}

\newglossaryentry{MmPlus}
{
  name={\ensuremath{\mathrm{Mm}^+}},
  description={Assigns the right discrete conjugate momentum expression to a discrete Lagrangian, see \eqref{eq:MmPlus}},
  sort=Mm
}

\newglossaryentry{MmMinus}
{
  name={\ensuremath{\mathrm{Mm}^-}},
  description={Assigns the left discrete conjugate momentum expression to a discrete Lagrangian, see \eqref{eq:MmMinus}. When a lower index $\overline x$ is given, we consider composition with the evaluation functional at $\overline x$. When an upper index $1$ or $2$ is provided, the operator acts on the first or second input argument of a binary function, respectively},
  sort=Mm
}

\newglossaryentry{Vol}
{
  name={\ensuremath{\mathrm{Vol}}},
  description={Assigns a volume form to a continuous or discrete Lagrangian, see \eqref{eq:VolDef} or \eqref{eq:VolDefLd}, respectively},
  sort=vol
}

\newglossaryentry{g}
{
  name={\ensuremath{g}},
  description={In the setting of continuous Lagrangians, $g(L)$ is the acceleration field induced by a non-degenerate Lagrangian $L$, i.e.\ $g_{\overline{x}}(L) = g(L)(\overline x)$ is the solution of $\mathrm{EL}(L)(x,\dot x,\ddot x)$ for $\ddot x$.
  In the setting of discrete Lagrangians, $g(L_d)$ is the 2nd order dynamical evolution rule induced by a non-degenerate, discrete Lagrangian $L_d$, i.e.\ $g_{\overline{x}}(L_d) = g(L_d)(\overline x)$ is the solution of $\mathrm{DEL}(L_d)(\overline x,x_2)$ for $x_2$},
  sort=g
}

\newglossaryentry{gref}
{
  name={\ensuremath{g_\mathrm{ref}}},
  description={The acceleration field $g(L_{\mathrm{ref}})$ (continuous setting) or the 2nd order dynamical evolution rule $g(L^{\mathrm{ref}}_d)$ (discrete setting) induced by a non-degenerate discrete Lagrangian $L^{\mathrm{ref}}_d$},
  sort=g
}

\newglossaryentry{grefbar}
{
  name={\ensuremath{\overline{g_\mathrm{ref}}}},
  description={The dynamical evolution rule induced by a non-degenerate discrete Lagrangian $L_{\mathrm{ref}}$, i.e.\ $\overline{g_\mathrm{ref}}(x_0,x_1)=(x_1,g_{\mathrm{ref}}(x_0,x_1))$},
  sort=g
}

\newglossaryentry{TRd}
{
  name={\ensuremath{T\mathbb{R}^d}},
  description={Tangent bundle over $\mathbb{R}^d$},
  sort=tangentbundle
}

\newglossaryentry{K}
{
  name={\ensuremath{K}},
  description={kernel function},
  sort=K
}

\newacronym{RKHS}{RKHS}{reproducing kernel Hilbert space}

\newglossaryentry{U}
{
  name={\ensuremath{U}},
  description={Reproducing kernel Hilbert space induced by the kernel $K$},
  sort=U
}

\newglossaryentry{V}
{
  name={\ensuremath{V}},
  description={Topological vector space with topological dual $V^\ast$},
  sort=V
}

\newglossaryentry{B}
{
  name={\ensuremath{B}},
  description={Separable Banach space with quadratic norm $\| \cdot \|_B$ and linear, positive symmetric bijection $Q \colon B^\ast \to B$ such that $\|u\|_B=(Q^{-1}u)u$},
  sort=B
}

\newglossaryentry{Q}
{
  name={\ensuremath{Q}},
  description={Linear, positive symmetric bijection $Q \colon B^\ast \to B$ such that $\|u\|_B=(Q^{-1}u)u$ for a separable Banach space $B$ with quadratic norm $\| \cdot \|_B$},
  sort=Q
}

\newglossaryentry{T}
{
  name={\ensuremath{T}},
  description={Positive symmetric linear operator. $T$ occurs as covariance operators of Gaussian fields},
  sort=T
}

\newglossaryentry{MathcalKappa}
{
  name={\ensuremath{\mathcal{K}}},
  description={Positive symmetric linear operator.
  	Quadratic identification of the RKHS $U^\ast$ and $U$ induced by the kernel $K$, see \eqref{eq:MathcalKappa}. It is used as a covariance operator of Gaussian fields defined on the RKHS $U$},
  sort=K
}

\newglossaryentry{xi}
{
  name={\ensuremath{\xi}},
  description={Gaussian random field},
  sort=xi
}

\newglossaryentry{xiM}
{
  name={\ensuremath{\xi_M}},
  description={Gaussian random field conditioned on all $M$ data points and regularisation conditions},
  sort=xi
}

\newglossaryentry{Normal}
{
  name={\ensuremath{\mathcal{N}}},
  description={Gaussian distribution or Gaussian random field distribution. $\xi \sim \mathcal{N}(u,T)$ denotes a Gaussian random field with mean $u$ and covariance operator $T$},
  sort=N
}

\newglossaryentry{phi}
{
  name={\ensuremath{\phi}},
  description={Element in the topological dual of a vector space (typically dual of $U$ or $B$)},
  sort=phi
}

\newglossaryentry{Phi}
{
  name={\ensuremath{\Phi}},
  description={Element in $(U^\ast)^n$ or $(B^\ast)^n$ for some finite $n \in \mathbb{N}$},
  sort=phi
}

\newglossaryentry{Phibm}
{
  name={\ensuremath{\Phi_b^{(M)}}},
  description={Functional consisting of the evaluation functionals of the (discrete) Euler--Lagrange operator at observation points, evaluation of the (discrete) conjugate momentum and the evaluation functional at reference point $\overline {x}_b$, see \eqref{eq:PhibM} or \eqref{eq:PhibmDiscrete} for the discrete case},
  sort=phibm
}

\newglossaryentry{PhiN}
{
  name={\ensuremath{\Phi_N}},
  description={Functional that evaluates the conjugate momentum and value of a (discrete) Lagrangian at a reference point $\overline {x}_b$},
  sort=phin
}

\newglossaryentry{ybm}
{
  name={\ensuremath{y_b^{(M)}}},
  description={Vector to serve as the right hand side when conditioning a Gaussian random field on evaluations of $\Phi_b^{(M)}$, see \eqref{eq:ybM} or \eqref{eq:PhibmDiscrete} for the continuous or discrete Lagrangian case, respectively},
  sort=ybm
}

\newglossaryentry{Theta}
{
  name={\ensuremath{\Theta}},
  description={Matrix required for the computation of posterior mean and covariance of a Gaussian random field conditioned on finitely many linear observations. It occurs in the treatment for continuous Lagrangians \eqref{eq:ThetaMat}, discrete Lagrangians \eqref{eq:ThetaMatLd}, and in general treatments of Gaussian random fields (\cref{app:CondExpandVar})},
  sort=Theta
}

\newglossaryentry{psi}
{
  name={\ensuremath{\psi}},
  description={Element in the topological dual of a vector space (typically dual of $U$ or $B$)},
  sort=psi
}

\newglossaryentry{Rd}
{
  name={\ensuremath{\mathbb{R}^d}},
  description={Eucledian space of dimension $d$},
  sort=psi
}

\newglossaryentry{S}
{
  name={\ensuremath{S}},
  description={Action functional},
  sort=S
}

\newglossaryentry{Sd}
{
  name={\ensuremath{S_d}},
  description={Discrete action functional},
  sort=Sd
}

\newglossaryentry{M}
{
  name={\ensuremath{M}},
  description={Number of observed data points},
  sort=M
}

\newglossaryentry{hFill}
{
  name={\ensuremath{h_{\Omega_0}}},
  description={Fill distance of a finite subset $\Omega_0 \subset \overline{\Omega}$, see \cref{def:hFill}},
  sort=h
}

\newglossaryentry{WSobolev}
{
  name={\ensuremath{W^r(\Omega)}},
  description={Sobolev space $W^r(\Omega) = W^{r,2}(\Omega)$, see \cref{rem:AssumptionTrueForK}},
  sort=W
}

\newglossaryentry{Cm}
{
  name={\ensuremath{\mathcal{C}^m}},
  description={Space of $m$-times continuously differentiable functions. For a compact set $\overline{\Omega}$, where $\Omega$  is an open, bounded, non-empty domain in a Eucledian space, $\mathcal{C}^m(\overline{\Omega},\R^k)$ and $\mathcal{C}^m(\overline{\Omega})=\mathcal{C}^m(\overline{\Omega},\R^1)$ denote Banach spaces described in \eqref{eq:CmOmegabar}},
  sort=Cm
}

\newglossaryentry{Omega}
{
  name={\ensuremath{\Omega}},
  description={Domain, where data is collected or domain of definition of a (discrete) Lagrangian. Typically assumed to be open, bounded, non-empty. In the continuous Lagrangian setting, $(x,\dot x) \in \Omega \subset T\R^d \cong \R^d \times \R^d$. In the discrete setting, $(x_0,x_1),(x_1,x_2)\in \Omega \subset \R^d \times \R^d$},
  sort=Omega
}

\newglossaryentry{OmegaBar}
{
  name={\ensuremath{\overline{\Omega}}},
  description={Topological closure of the set $\Omega$. By the assumed properties of $\Omega$, $\overline{\Omega}$ is compact},
  sort=Omega
}

\newglossaryentry{Omega0}
{
  name={\ensuremath{\Omega_0}},
  description={Set of points in $\Omega$, for which observation data is available. Can be finite or countably infinite},
  sort=Omega
}

\newglossaryentry{Omega0hat}
{
  name={\ensuremath{\hat{\Omega}_0}},
  description={Set of observation data. (Only used in the context of discrete systems.)},
  sort=Omega
}

%% file: ORCID_logo.tex
\definecolor{orcidlogocol}{HTML}{A6CE39}
\tikzset{
  orcidlogo/.pic={
    \fill[orcidlogocol] svg{M256,128c0,70.7-57.3,128-128,128C57.3,256,0,198.7,0,128C0,57.3,57.3,0,128,0C198.7,0,256,57.3,256,128z};
    \fill[white] svg{M86.3,186.2H70.9V79.1h15.4v48.4V186.2z}
                 svg{M108.9,79.1h41.6c39.6,0,57,28.3,57,53.6c0,27.5-21.5,53.6-56.8,53.6h-41.8V79.1z M124.3,172.4h24.5c34.9,0,42.9-26.5,42.9-39.7c0-21.5-13.7-39.7-43.7-39.7h-23.7V172.4z}
                 svg{M88.7,56.8c0,5.5-4.5,10.1-10.1,10.1c-5.6,0-10.1-4.6-10.1-10.1c0-5.6,4.5-10.1,10.1-10.1C84.2,46.7,88.7,51.3,88.7,56.8z};
  }
}

\newcommand\orcidicon[1]{\href{https://orcid.org/#1}{\mbox{\scalerel*{
\begin{tikzpicture}[yscale=-1,transform shape]
\pic{orcidlogo};
\end{tikzpicture}
}{|}}}}

%% file: abstract.tex
The article introduces a method to learn dynamical systems that are governed by Euler--Lagrange equations from data.
The method is based on Gaussian process regression and identifies continuous or discrete Lagrangians and is, therefore, structure preserving by design.
A rigorous proof of convergence as the distance between observation data points converges to zero \mod{and lower bounds for convergence rates are} provided.
Next to convergence guarantees, the method allows for quantification of model uncertainty, which can provide a basis of adaptive sampling techniques. We provide efficient uncertainty quantification of any observable that is linear in the Lagrangian, including of Hamiltonian functions (energy) and symplectic structures, which is of interest in the context of system identification.
The article overcomes major practical and theoretical difficulties related to the ill-posedness of the identification task of (discrete) Lagrangians through a careful design of geometric regularisation strategies and through an exploit of a relation to convex minimisation problems in reproducing kernel Hilbert spaces.

%% file: Introduction_ODE.tex
\section{Introduction}

The identification of models of dynamical systems from data is an important task in machine learning with applications in engineering, physics, and molecular biology. Data-driven models are required when explicit descriptions for the equations of motions of dynamical systems are either not known or analytic descriptions are too computationally complex for large scale simulations.
{This contribution focuses on structure-preserving machine learning of dynamical systems based on Gaussian process regression and Gaussian fields. The framework allows for a rigorous convergence analysis and numerically efficient uncertainty estimation. The proposed method is a Lagrangian-based data-driven model. Let us briefly contrast the approach to Hamiltonian data-driven models and other Lagrangian-based models.}


\subsection{Hamiltonian data-driven models}
Physics-based, data-driven modelling aims to exploit prior physical or geometric knowledge when developing data-driven surrogate models of dynamical systems. Recent activities have developed methods to learn Hamiltonian systems{, i.e.\ systems of the form}
\[
{\dot z = J^{-1}\nabla H(z),\;  J=\begin{pmatrix}
		0_{d \times d}&-1_{d \times d}\\1_{d \times d}&0_{d \times d}
\end{pmatrix} \quad H \colon \R^{2d}\to \R \, (\text{Hamiltonian})},
\]
or port-Hamiltonian systems from data by approximating the Hamiltonian, pseudo-{Hamiltonian structure}, or port-Hamiltonian structure by neural networks or Gaussian processes \cite{HNN, EIDNES2024112738,Bertalan2019, ortega2023learnability,symplecticShadowIntegrators,David2023}. Additionally, Lie group symmetries are identified in \cite{SymHNN}. Alternatively, the symplectic flow map of Hamiltonian systems can be approximated \cite{Rath2021,Tao2021,SympNets}.
The data-driven identification of interaction-based agent systems in \cite{Feng2003,Lu2019} or general Hamiltonian systems in \cite{DaiyingYin2024} employ similar statistical learning methods as in this article but in the context of Hamiltonian systems.
In contrast to the variational models considered in this article, Hamiltonian data-driven models mostly require prior knowledge of the symplectic phase space structure and observations of position and momenta, while the proposed Lagrangian-based methods only require observations of positions. \mod{However,} symplectic structures and Hamiltonians can be derived from a Lagrangian model in a post-processing step. Approaches based on identifying symplectic structures or canonical coordinates from data together with a Hamiltonian have been considered, for instance, in \cite{Bertalan2019,chen2023NeuralSymplecticForm}. However, these do not provide a systematic discussion of uncertainty quantification or regularisation of this ill-posed inverse problem \mod{or provide theoretical convergence guarantees.}

\subsection{Continuous Lagrangian data-driven models}
Similarly to Hamiltonian data-driven models, variational principles for dynamical systems have been identified from data by identifying a Lagrangian function of the system \cite{LNN,LagrangianShadowIntegrators,evangelisticdc2022,SymLNN}.
{We refer to \cite{Marsden1999,Arnold1989} for an introduction to Lagrangian mechanics.}
To recall briefly, a dynamical system is governed by a {\em variational principle} or a {\em least action principle}, if motions constitute critical points of an action functional. In case of an autonomous first-order time-dependent system, the action functional is of the form
\begin{equation}\label{eq:SContinuous}
	\gls{S}(x) = \int_{t_0}^{t_1} \gls{L}(x(t),\dot x(t)) \d t,
\end{equation}
where $x \colon [t_0,t_1] \to \gls{Rd}$ is a curve with derivative denoted by $\dot x$. The function $L$ is a {\em Lagrangian}.
A function $x \colon [t_0,t_1] \to \R^d$ is a solution or {\em motion} if the action $S$ is stationary at $x$ for all variations $\delta x \colon [t_0,t_1] \to \R^d$ that fix the endpoints $t_0,t_1$. Regularity assumptions on $L$ and $x$ provided, this is equivalent to the condition that $x$ fulfils the Euler-Lagrange equations
\begin{equation}\label{eq:ELODE}
\mathrm{EL}(L)(x(t),\dot x(t),\ddot x(t))=0, \qquad t \in (t_0,t_1)
\end{equation} with
\begin{equation}\label{eq:ELContinuous}
	\begin{split}
	\gls{EL}(L) = \frac{\d }{\d t} \left( \frac{\p L}{\p {\dot x}} \right) - \frac{\p L}{\p x}
	= {\frac{\p^2 L}{\p \dot x \p \dot x} \ddot{x} 
	+ \frac{\p^2 L}{\p \dot x \p x} \dot{x} }
	- \frac{\p L}{\p x}.
	\end{split}
\end{equation}
Here, {$\frac{\p^2 L}{\p \dot x \p \dot x} = \left( \frac{\p^2 L}{\p \dot x^k \p \dot x^l} \right)_{k,l=1}^d$, $\frac{\p^2 L}{\p \dot x \p  x} = \left( \frac{\p^2 L}{\p \dot x^k \p x^l} \right)_{k,l=1}^d$} refer to $d \times d$-dimensional blocks of the Hessian of $L$ and $\frac{\p L}{\p x}$ denotes the gradient.
Details may be found in \cite{gelfand2000calculus,RoubicekCalculusofVariations}, for instance.

In the data-driven context, $L$ is sought as a function of $\gls{xbar}=(x,\dot x)$ such that \eqref{eq:ELODE} is fulfilled at observed data points
$\left\{(x,\dot x, \ddot x)\right\}_{j=1}^{\gls{M}}$.
Once $L$ is known, \eqref{eq:ELODE} can be solved with a numerical method such as a variational integrator \cite{MarsdenWestVariationalIntegrators}.

\subsection{Discrete Lagrangian data-driven models}\label{Intro:DiscreteLagrangianModel}
Instead of learning continuous variational principles, in \cite{Qin2020} Qin proposes to learn discrete Lagrangian theories by approximating discrete Lagrangians. 
In discrete Lagrangian theories, motions $x(t)$ are described at discrete, equidistant times $t^0 < t^1 < \ldots < t^N$ by a sequence of snapshots ${\boldsymbol x} = (x_k)_{k =0}^N \subset \R^d$. The motions constitute stationary points of a discrete action functional
\[
\gls{Sd}({\boldsymbol x}) = \sum_{k=1}^N \gls{Ld}(x_{k-1},x_k)
\]
with respect to discrete variations of the interior points $x_1,\ldots,x_{N-1}$. In other words, ${\boldsymbol x}$ is a solution of the discrete field theory if $\frac{\p S_d}{\p x_k}({\boldsymbol x}) =0$ for all $1\le k <N$. This is equivalent to the discrete Euler--Lagrange equation \begin{equation}\label{eq:DELEQIntro}
	\mathrm{DEL}(L_d)(x_{k-1},x_k,x_{k+1}) =0, \quad 1 \le k <N
\end{equation} with
\begin{equation}\label{eq:DELIntro}
\gls{DEL}(L_d)(x_{k-1},x_k,x_{k+1})
= \gls{nabla2} L_d(x_{k-1},x_k) + \gls{nabla1} L_d (x_k,x_{k+1}).
\end{equation}
Here $\nabla_1 L_d$ and $\nabla_2 L_d$ denote the partial derivatives with respect to the first or second component of $L_d$, respectively. Details on discrete mechanics can be found in \cite{MarsdenWestVariationalIntegrators}.

For the identification of discrete Lagrangians from data, training data $\{x(t^{k})\}_k$ consists of snapshots of motions of the dynamical system at discrete time-steps $t^k$.
This needs to be contrasted to training of continuous Lagrangians which requires observations of first and second order derivatives of solutions, i.e.\ data of the form $\gls{xhat} = (x,\dot x, \ddot x)$.

The class of discrete Lagrangian systems is expressive enough to describe motions of continuous Lagrangian systems on bounded open subsets of $\R^d$ at the snapshot times $(t^k)_k$ exactly, i.e.\ without discretisation error, provided the step-size $\Delta t = t^{k+1}-t^k$ is small enough, see \cite[§1.6]{MarsdenWestVariationalIntegrators}. Thus, identifying $L_d$ instead of $L$ is fully justified from a modelling viewpoint.
In case a continuous Lagrangian is required for system identification tasks or highly accurate predictions of velocity data, in the article \cite{LagrangianShadowIntegrators} the author provides a method based on Vermeeren's variational backward error analysis \cite{Vermeeren2017} to recover continuous Lagrangians from data-driven discrete Lagrangians as a power series in the step-size \mod{$\Delta t$} of the time-grid.

\subsection{Ambiguity of Lagrangians}
The data-driven identification of a continuous or discrete Lagrangian density is an ill-defined inverse problem as many different Lagrangian densities can yield equations of motions with the same set of solutions.
This \mod{constitutes} a challenge in a machine learning context and can lead to badly conditioned identified models that amplify errors \cite{LagrangianShadowIntegrators}.
In \cite{DLNNPDE,DLNNDensity} the author develops regularisation strategies that optimise numerical conditioning of the learnt theory, when the Lagrangian density is modelled as a neural network. The present article relates to Gaussian fields to allow for efficient uncertainty quantification and a theoretical convergence analysis.

\subsection{Novelty}\label{sec:novelty}
The article
\begin{enumerate}
	\item introduces a method to learn continuous and discrete Lagrangians from data based on Gaussian process regression with 
	\begin{itemize}
		\item a rigorous convergence analysis as the distance between data points converges to zero
		\item \mod{and lower bounds on the convergence rates, depending on the smoothness of the dynamics and kernels.}
	\end{itemize}

\item Moreover, the article systematically discusses the ambiguity or Lagrangians and {regularisation} strategies for kernel-based learning methods for Lagrangians.

\item Furthermore, the article provides a statistical framework that allows for {efficient} uncertainty quantification of any linear observable of the dynamical system, such as Hamiltonian functions (energy) or symplectic structure, for instance.
The uncertainty quantification does not require sampling but only to solve linear systems of equations.

\end{enumerate}

This needs to be contrasted to aforementioned methods of the literature for learning Lagrangians, for which convergence guarantees are not provided or which do not provide uncertainty quantification of linear observables.
\mod{We will prove convergence of our inferred (discrete) Lagrangians to a limit that constitutes a true (discrete) Lagrangian of the underlying dynamical system in a reproducing kernel Hilbert space and in $\mathcal{C}^2$ ($\mathcal{C}^1$ in the discrete case) as the maximal distance between observed data points $h$ (fill distance) tends to zero.}
\mod{Moreover, lower bounds for convergence rates are proved:}
\mod{in case of continuous Lagrangian models, when the acceleration field of the underlying dynamics is at least $r$ times continuously differentiable, $r>2+d$ for $d$ the dimension of position data, and the model's kernel is sufficiently regular, then the learned dynamics (the Euler--Lagrange equations or the acceleration field away from non-degeneracies) converges to the true dynamics at least as fast as $h^r$. Similar results are proved for the discrete Lagrangian case.}

In the literature discussions on removing ambiguity of Lagrangians in data-driven identification are mostly absent: its necessity is sometimes avoided by assuming that torques are observed \cite{evangelisticdc2022}, an explicit mechanical ansatz is used \cite{Aoshima2021}. In other works regularisation is done implicitly without discussion \cite{LNN}, ad hoc as in the author's prior work \cite{LagrangianShadowIntegrators}, or relates to neural networks \cite{SymLNN,DLNNPDE,DLNNDensity} only.

Methodologically, the method of the present article stands in the context of meshless collocation methods \cite{SchabackWendland2006} for solving linear partial differential equations since it solves \eqref{eq:ELContinuous} for $L$.
It overcomes the major technical difficulty to prove convergence even though the Lagrangian density is {\em not} unique even after regularisation.
For this, the article exploits a relation between posterior means of Gaussian processes and constraint optimisation problems in reproducing kernel Hilbert spaces that was presented in a game theory context by Owhadi and Scovel in \cite{OwhadiScovel2019} and was employed to solve well-posed partial differential equations using Gaussian Processes in \cite{OwhadiLearningPDEGP}.
\mod{Moreover, interpolation and smoothening theory \cite{Arcangeli2007,Narcowich2005,Wendland2005} is applied to prove the aforementioned lower bounds for convergence rates.}


\subsection{Outline}
The article proceeds as follows: \Cref{sec:background} continues the review of continuous and discrete variational principles that was started in the introduction. Moreover, it presents symplectic structure and Hamiltonians as linear observables of Lagrangian systems and it reviews the ambiguity of Lagrangians. \Cref{sec:Normalisation} introduces methods to regularise the inverse problem of finding Lagrangian densities given dynamical data. In \cref{sec:MLSetting} we {briefly} review reproducing kernel Hilbert spaces and {aspects of} Gaussian {fields}. {A more detailed discussion of the underlying theoretical concepts is provided in \cref{app:GaussianField}. The section proceeds with an introduction of} our method to learn continuous and discrete Lagrangians and to provide uncertainty quantifications for linear observables.
\Cref{sec:Experiment} contains numerical experiments including identification of a Lagrangian and Hamiltonian for the coupled harmonic oscillator and convergence tests.
\Cref{sec:ConvergenceAnalysis} provides a theoretical convergence analysis of the method including a proof of the method's convergence. \mod{Moreover, lower bounds for} convergence rates are derived \mod{in \cref{sec:ConvergenceRates}}.
The article concludes with a summary in \cref{sec:Summary}.
\mod{A list of notation with reoccurring symbols is supplied at the end of the article as a glossary. 
}

%% file: background_LagrangianDynamics.tex
\section{Background - Lagrangian dynamics}\label{sec:background}

\subsection{Continuous Lagrangian theories}
\subsubsection{Definition of associated Hamiltonian and symplectic structure}\label{sec:DefHamSymplMmVol}
Let us continue our review of Lagrangian dynamics to fix notations and to explain the ambiguity that is inherent in the inverse problem of identifying (discrete) Lagrangians to observed motions.
We postpone a provision of a more detailed functional analytic settings to the convergence analysis of \cref{sec:ConvergenceAnalysis} and refer to the literature on variational calculus \cite{gelfand2000calculus,RoubicekCalculusofVariations} for details.

We consider the Hamiltonian to a Lagrangian defined via
\begin{equation}\label{eq:HamOperator}
	\gls{Ham}(L)(x,\dot x) = \dot{x}^\top \frac{\p L}{\p \dot x}(x,\dot x) - L(x,\dot x).
\end{equation}
Here $\dot x^\top$ denotes the transpose of $\dot x \in \R^d$.
The Hamiltonian $\mathrm{Ham}(L)$ is conserved along solutions of \eqref{eq:ELODE}. Moreover, we consider the symplectic structure related to $L$ which is given as the closed differential 2-form
\begin{equation}\label{eq:SymplecticStructure}
	\gls{Sympl}(L) = \sum_{s=1}^d \d x^s \wedge \d \left(\frac{\p L}{\p \dot x^s}\right) 
	=\sum_{s,r=1}^d  \mod{\Bigg(}\frac{\p^2 L}{\p  x^r \p \dot x^s} \d x^s \wedge \d x^r 
	+   \frac{\p^2 L}{\p \dot x^r \p \dot x^s} \d  x^s \wedge \d \dot x^r\mod{\Bigg)}.
\end{equation}
When $\frac{\p^2 L}{\p \dot x \p \dot x}$ is invertible everywhere, then the differential form $\mathrm{Sympl}(L)$ is non-degenerate and, therefore, a symplectic form.\footnote{$\mathrm{Sympl}(L)$ is the pull-back of the canonical symplectic form $\sum_{s=1}^{d} \d q^s \wedge \d p_s$ under the Legendre transform $T\R^d \to T^\ast \R^d$, $(x,\dot x) \mapsto (q,p)=(x,\frac{\p L}{\p \dot x}(x,\dot x))$.}
As an aside, the motions \eqref{eq:ELODE} can be described as Hamiltonian motions to the Hamiltonian $\mathrm{Ham}(L)$ and symplectic structure $\mathrm{Sympl}(L)$.
Moreover, we consider the induced \mod{conjugate} momenta
\begin{equation}\label{eq:MmDef}
	\gls{Mm}(L)(x,\dot x) = \frac{\p L}{\p \dot x}(x,\dot x).
\end{equation}
Additionally, we consider the induced Liouville volume form given as the $d$th exterior power of $\mathrm{Sympl}(L)$
\begin{equation}\label{eq:VolDef}
	\gls{Vol}(L) = \frac{1}{d!} (\mathrm{Sympl}(L))^d  =  \det\left(\frac{\p^2 L}{\p \dot x^r \p \dot x^s} \right) \d x^1 \wedge \d \dot x^1 \wedge \ldots \wedge \d x^d \wedge \d \dot x^d.
\end{equation}

It will be of significance later that $\mathrm{EL}$, $\mathrm{Ham}$, $\mathrm{Sympl}$, $\mathrm{Mm}$ are linear in the Lagrangian $L$, while $\mathrm{Vol}$ is not.

{%
	\begin{example}\label{ex:MechanicalL}
	Consider a mechanical Lagrangian $L(x,\dot x) = \frac 12 \dot x^\top \mod{\Lambda} \dot x - V(x)$ for a continuously differentiable potential $V \colon \R^d \to \R$ and a symmetric, positive definite matrix $\mod{\Lambda}$ (mass matrix).
	The equations of motions are $0=\mathrm{EL}(L)(x,\dot x,\ddot x) = \mod{\Lambda}\ddot{x} + \nabla V(x)$, where $\nabla V = \frac{\p V}{\p x}$ denotes the gradient of $V$.
	The conjugate momentum is $p :=\mathrm{Mm}(L)(x,\dot x)=\mod{\Lambda} \dot x$. The Hamiltonian function is $H(x,p) = \mathrm{Ham}(L)(x,\mod{\Lambda}^{-1} p) = \frac 12 p^\top \mod{\Lambda}^{-1} p + V(x)$. The symplectic form is $\omega = \mathrm{Symp}(L) = \sum_{s=1}^d \d x^s \wedge \d p^s$. In the frame induced by the coordinates $(x,p)$ of the phase space the symplectic form is represented by the block matrix  \[J=\begin{pmatrix}
		0_{n \times n} & -1_{n \times n} \\
		1_{n \times n} & 0_{n \times n}
	\end{pmatrix}.\]
	Here $0_{n \times n}$ and $1_{n \times n}$ denote the zero and the identity matrix of size $n \times n$, respectively.
	In the coordinates $(x,p)$, the equations of motions are Hamilton's equations in their standard form
	\[
	\begin{pmatrix}
		\dot x \\ \dot p
	\end{pmatrix}	
	= J^{-1} \nabla H(x,p)
	= \begin{pmatrix}
		\mod{\Lambda}^{-1}p \\ - \nabla V(x)
	\end{pmatrix}	
	\]
	The volume form $\mathrm{Vol}(L) = \det(\mod{\Lambda}) \d x^1 \wedge \d \dot x^1 \wedge \ldots \wedge \d x^d \wedge \d \dot x^d
	= \d x^1 \wedge \d p^1 \wedge \ldots \wedge \d x^d \wedge \d p^d$ is the standard Euclidean volume form on the phase space.
	\end{example}%
}

\subsubsection{Ambiguity of Lagrangian densities}\label{sec:AmbiguityContinuous}

The ambiguity of Lagrangians in the description of variational dynamical systems has been the subject of various articles in theoretical physics including \cite{HENNEAUX198245,Marmo1987,MARMO1989389}.
Lagrangians can be ambiguous in two different ways: 
\begin{enumerate}
	
	\item\label{it:LEquivalent} Lagrangians $L$ and $\tilde L$ can yield the same Euler--Lagrange operator \eqref{eq:ELContinuous} up to rescaling, i.e.\ 
	\[
	\rho\mathrm{EL}(L) =    \mathrm{EL}(\tilde{L}), \quad \rho \in \R\setminus \{0\}
	\]
	and, therefore, the same Euler--Lagrange equations \eqref{eq:ELODE} up to rescaling. We call $L$ and $\tilde L$ {\em (gauge-) equivalent}.
	For equivalent Lagrangians $L$, $\tilde L$ there exists $\rho \in \R \setminus \{0\}$, $c \in \R$ such that $\tilde L - \rho  L -c$ is a total derivative
	\[
	\tilde L - \rho L -c = \d_t F
	\]
	for a continuously differentiable function $F\colon \R^d \to \R$, where
	\begin{equation}
		\d_t F(x,\dot x)
		= \dot x^\top \nabla F(x)
		= \sum_{s=1}^d \dot x^s \frac{\p F}{\p x^s}(x)
	\end{equation}
	(See, e.g.\ \cite{gelfand2000calculus}.)
	We have restricted ourselves to autonomous Lagrangians.
	
	
	\item More generally, two Lagrangians $L$ and $\tilde L$ can yield the same set of solutions $x$, i.e.
	\[
	\mathrm{EL}(L)(x(t),\dot x(t)),\ddot{x}(t)) =0
	\iff \mathrm{EL}(\tilde L)(x(t),\dot x(t)),\ddot{x}(t)) =0
	\]
	for all regular curves $x\colon [t_0,t_1] \to \R^d$ even when they are {\em not} equivalent in the sense of \cref{it:LEquivalent}. In such a case, $\tilde{L}$ is called an {\em alternative Lagrangian} to $L$.
	\begin{example}[Affine linear motions]\label{ex:LinearMotions}
		For any twice differentiable $g \colon \R^d \to \R$ with nowhere degenerate Hessian matrix $\mathrm{Hess}(g)$, the Lagrangian $L(x,\dot x) = g(\dot x)$ describes affine linear motions in $\R^d$:
		\[
		0 = \mathrm{EL}(L) = \mathrm{Hess}(g)(\dot x)\ddot x.
		\]
		\mod{For instance, for $d=2$, $g(\dot x) = \frac 12((\dot x^1)^2+(\dot x^2)^2)$ and $g(\dot x) = \frac 12((\dot x^1)^2-(\dot x^2)^2)$ yield non-equivalent alternative Lagrangians.}
	\end{example}

\end{enumerate}

In general, the existence of alternative Lagrangian densities is related to additional geometric structure and conserved quantities of the system \cite{HENNEAUX198245,Marmo1987,MARMO1989389,Carinena1983}. This article mainly considers ambiguities by equivalence, which are exhibited by all variational systems.


\begin{lemma}\label{lem:TrafoGeomQuantities}
	Let $L$ be a Lagrangian depending on $(x,\dot x)$. Consider a continuously differentiable $F \colon \R^d \to \R$, $\rho \in \R$, $c \in \R$, and $\tilde L = \rho L + \d_t F + c$. We have
	\begin{align*}
		\mathrm{EL}(\tilde L) &= \rho \mathrm{EL}(L)\\
		\mathrm{Mm}(\tilde L) &= \rho \mathrm{Mm}(L) + \nabla F\\
		\mathrm{Sympl}(\tilde L) &= \rho \mathrm{Sympl}(L)\\
		\mathrm{Vol}(\tilde L) &= \rho^d \mathrm{Vol}(L)\\
		\mathrm{Ham}(\tilde L) &= \rho \mathrm{Ham}(L)-c
	\end{align*}
	Here $\nabla F$ denotes the gradient of $F$. Moreover, if $\rho \not =0$ then
	\begin{equation}\label{eq:NonDegeneratePtsInvariant}
		\left\{(x,\dot x) : \det \left( \frac{\p^2 L}{\p \dot x \p \dot x} \right)(x,\dot x) \not =0 \right\} = 
		\left\{(x,\dot x) : \det \left( \frac{\p^2 \tilde L}{\p \dot x \p \dot x} \right)(x,\dot x) \not =0\right\}.
	\end{equation}
\end{lemma}

\begin{proof}
	The transformation rules of EL, Mm, Sympl, Vol, and Ham are obtained by a direct computation. 
	The assertion \eqref{eq:NonDegeneratePtsInvariant} follows from the transformation rule for $\mathrm{Vol}$ or directly by observing that $\frac{\p^2 \tilde L}{\p \dot x \p \dot x} = \rho \frac{\p^2 L}{\p \dot x \p \dot x}$.
\end{proof}

The following \namecref{cor:NonDegInvariant} is a restatement of \eqref{eq:NonDegeneratePtsInvariant}.

\begin{corollary}\label{cor:NonDegInvariant}
	The set where a Lagrangian $L$ is non-degenerate, i.e.\ where $\frac{\p^2 L}{\p \dot x \p \dot x}$ is invertible, is invariant under equivalence. 
\end{corollary}

\subsection{Discrete Lagrangian systems}

\subsubsection{Associated symplectic structure}

In analogy to the continuous case (\cref{sec:DefHamSymplMmVol}) we define associated data to a discrete Lagrangian density $L_d \colon \R^d \times \R^d \to \R$ following definitions in discrete variational calculus \cite{MarsdenWestVariationalIntegrators}.
The quantities
\begin{align}\label{eq:MmMinus}
	\gls{MmMinus}(L_d)(x_j,x_{j+1}) &= -\nabla_1 L_d (x_j,x_{j+1})\\ \label{eq:MmPlus}
	\gls{MmPlus}(L_d)(x_{j-1},x_j) &= \nabla_2 L_d (x_{j-1},x_j)
\end{align}
relate to discrete conjugate momenta at time $t_j$. On motions ${\boldsymbol x} =(x_k)_{k=0}^N$ that fulfil \eqref{eq:DELEQIntro}, $\mathrm{Mm}^-(x_k,x_{k+1})$ and $\mathrm{Mm}^+(x_{k-1},x_k)$ coincide for all $1\le k<N$.
Moreover, denoting the coordinate of the domain of definition $\R^d \times \R^d$ of $L_d$ by $(x_0,x_1)$ we define the 2-form
\begin{equation}\label{eq:SymplDiscrete}
	\begin{split}
		\gls{Sympl}(L_d)	&= \sum_{r,s=1}^{d}
		\frac{\p^2 L_d}{\p x_1^s \p x_0^r} \d x_1^s \wedge \d x_0^r
	\end{split}
\end{equation}
and its $d$th exterior power normalised by $\frac 1 {d!}$
\begin{equation}\label{eq:VolDefLd}
	\gls{Vol}(L_d) = \det \left(\frac{\p^2 L_d}{\p x_1 \p x_0}\right) \d x_1^1 \wedge \d x_0^1 \wedge \ldots \wedge \d x_1^d \wedge \d x_0^d.
\end{equation}

When $\frac{\p^2 L_d}{\p x_1 \p x_0}$ is non-degenerate everywhere, then $\mathrm{Sympl}(L_d)$ is a symplectic form
and $\mathrm{Vol}(L_d)$ its induced volume form on the discrete phase space $\R^d \times \R^d$.
$\mathrm{Sympl}(L_d)$ is called {\em discrete Lagrangian symplectic form} in \cite[§1.3.2]{MarsdenWestVariationalIntegrators}.
(For consistency with the continuous theory \cref{sec:DefHamSymplMmVol} our sign convention differs from \cite[§1.3.2]{MarsdenWestVariationalIntegrators}. A derivation can be found in \cref{app:SymplVolLd}.)

\subsection{Ambiguity of discrete Lagrangians}

In analogy to \cref{sec:AmbiguityContinuous},
if $L_d$ is a discrete Lagrangian and $\tilde L_d(x_0,x_1) = \rho L_d(x_0,x_1) + F(x_1)-F(x_0) + c$ for $c \in \R$, $\rho \in \R \setminus \{0\}$, and continuously differentiable $F$, then
\[
\rho \mathrm{DEL}(L_d) =  \mathrm{DEL}(\tilde L_d)
\]
and $L_d$ and $\tilde L_d$ are called {\em (gauge-) equivalent}. 
Non-equivalent discrete Lagrangians such that the discrete Euler--Lagrange equations \eqref{eq:DELEQIntro} have the same solutions are called {\em alternative {discrete} Lagrangians}.

The analogy of \cref{lem:TrafoGeomQuantities} for discrete Lagrangians is as follows.

\begin{lemma}\label{lem:TrafoGeomQuantitiesDiscrete}
	Let $L_d$ be a {discrete} Lagrangian depending on $(x_0,x_1)$. Consider a continuously differentiable $F \colon \R^d \to \R$, $\rho \in \R $, $c \in \R$, and $\tilde L_d = \rho L_d + \Delta_t F + c$ with $\Delta_t F(x_0,x_1) = F(x_1)-F(x_0)$. We have
	\begin{align*}
		\mathrm{DEL}(\tilde L_d) &= \rho \mathrm{DEL}(L_d)\\
		\mathrm{Mm}^-(\tilde L_d)(x_0,x_1) &= \rho \mathrm{Mm}^-(L_d)(x_0,x_1) + \nabla F(x_0)\\
		\mathrm{Mm}^+(\tilde L_d)(x_0,x_1) &= \rho \mathrm{Mm}^+(L_d)(x_0,x_1) + \nabla F(x_1)\\
		\mathrm{Sympl}(\tilde L_d) &= \rho \mathrm{Sympl}(L_d)\\
		\mathrm{Vol}(\tilde L_d) &= \rho^d \mathrm{Vol}(L_d)
	\end{align*}
	Here $\nabla F$ denotes the gradient of $F$. Moreover, if $\rho \not =0$ then
	\[
	\left\{(x_0,x_1) : \det \left( \frac{\p^2 L_d}{\p x_0 \p x_1} \right)(x_0,x_1) \not =0 \right\}
	= 
	\left\{(x_0,x_1) : \det \left( \frac{\p^2 \tilde L_d}{\p x_0 \p x_1} \right)(x_0,x_1) \not =0 \right\}.
	\]
\end{lemma}

\begin{proof}
	The transformation rules of EL, Mm$^\pm$, Sympl, Vol are obtained by a direct computation. 
	The assertion about invariance of non-degenerate points follows from the transformation rule of Vol.
\end{proof}




%% file: normalisation.tex
\section{{Regularisation}}\label{sec:Normalisation}

In the machine learning framework that we will introduce in \cref{sec:MLSetting},
we will employ {regularisation} conditions to safeguard us from finding degenerate solutions to the inverse problem of identifying a Lagrangian to given motions. Extreme instances of degenerate solutions are Null-Lagrangians, for which $\mathrm{EL}(L) \equiv 0$. These are consistent with any dynamics but cannot discriminate curves that are not motions.

The following section serves two goals:

\begin{itemize}
	\item We justify that the employed {regularisation} conditions are covered by the ambiguities presented in \cref{sec:background}. {Therefore, imposing these on $L$ does not restrict the generality of the ansatz. We will also refer to these as {\em normalisation conditions} as we will impose that these are fulfilled exactly by the data-driven model.} 
	
	\item {The normalisation conditions (together with the system's motions) do {\em not} determine the Lagrangian uniquely. However, they guarantee that the sought Lagrangian is non-degenerate, provided that there are no true degenerate Lagrangians.
	Furthermore, we show that the normalisation conditions determine the symplectic structure $\mathrm{Sym}(L)$, the Hamiltonian $\mathrm{Ham}(L)$, and the Euler--Lagrange operator $\mathrm{EL}(L)$ of the system uniquely, provided that no true alternative Lagrangians exist. In the context of uncertainty quantification, this implies that any ambiguity in the representation of the model $L$ does not contribute to uncertainty in the Hamiltonian, the symplectic structure, or the equations of motions. This justifies the approach towards uncertainty quantification in the article.}
	
\end{itemize}

A reader mostly interested in the machine learning setting can skip ahead to \cref{sec:MLSetting}.



\subsection{{Preparation of the regularisation strategy}}\label{sec:PrepareRegularisationStrategy}

\begin{proposition}\label{prop:NormaliseL}
	Let $\gls{xb} = (x_b,\dot x_b) \in T\R^d \cong \R^d \times \R^d$, 
	{$\mathring L$ a Lagrangian, and } $\hat x_\tau = (x_\tau,\dot x_\tau, \ddot x_\tau) \in (\R^d)^3$ with $\mathrm{EL}(\mathring{L})(\hat x_\tau) \not =0$.\footnote{{This means that $\hat x_\tau = (x_\tau,\dot x_\tau, \ddot x_\tau)$ is any point that does not correspond to a motion of the dynamical system described by $\mathring{L}$. For instance, when $(x_\tau,\dot x_\tau)$ is an equilibrium point of the dynamics then we can chose any $\ddot x_\tau \not =0$. The assumption excludes trivial Lagrangians such as $\mathring{L}\equiv 0$.}}
	Let $c_b \in \R$, $p_b \in \R^d$, $c_\tau \not =0$.
Then there exists a Lagrangian $L$ such that $L$ is equivalent to $\mathring{L}$ and
\begin{equation}\label{eq:normLCond}
L(\overline{x}_b)=c_b,
\quad 
\mathrm{Mm}(L)(\overline{x}_b) = \frac{\p L}{\p \dot x}(\overline{x}_b) = p_b,
\quad
(\mathrm{EL}(L)(\hat{x}_\tau))_k = c_\tau,
\end{equation}
where $1\le k\le d$ is any index for which the $k$th component of  $\mathrm{EL}(\mathring{L})(\hat x_\tau)$ is not zero.	
\end{proposition}

\begin{proof}
Let
$\mathring{c_b} = \mathring L(\overline{x}_b)$,
$\mathring{p_b}=\mathrm{Mm}(\mathring L)(\overline{x}_b)$,
$\mathring{c_\tau} = (\mathrm{EL}(\mathring L)(\hat{x}_\tau))_k$ ($k$ th component).
We set
\[
\rho = \frac{c_\tau}{\mathring{c_\tau}},
\quad
F(x) = x^\top (p_b- \rho \mathring{p_b}),
\quad 
c = c_b - \dot{x}_b^\top (p_b- \rho \mathring{p_b})-\rho \mathring{c_b}.
\]
Now the Lagrangian
\begin{align*}
	L\mod{(x,\dot x)} &= \rho \mathring{L}\mod{(x,\dot x)} + \d_t F\mod{(x,\dot x)} + c\\
	&= \mod{\rho \mathring{L}(x,\dot x)
	+(\dot x - \dot x_b)^\top (p_b-\rho \mathring{p}_b)+c_b - \rho \mathring{c}_b}
\end{align*}
is equivalent to $\mathring{L}$ and fulfils \eqref{eq:normLCond}.
\end{proof}

While the equivalent Lagrangian $L$ constructed in \cref{prop:NormaliseL} is always non-degenerate if $\mathring{L}$ is non-degenerate (by \cref{lem:TrafoGeomQuantities}), this is not necessarily true for all Lagrangians governing the motions even when restricting to those that fulfil \eqref{eq:normLCond}: indeed, in \cref{ex:LinearMotions} of affine linear motions governed by $\mathring L(x,\dot x) = \dot x^2$, we can choose $g$ such that $L(x,\dot x) = g(\dot x)$ has degenerate points at any points. However, when we exclude systems with alternative Lagrangians, then we have the following \namecref{prop:NOmegaImpliesNonDeg}.

\begin{proposition}\label{prop:NOmegaImpliesNonDeg}
	Let $\mathring{L}$ be a Lagrangian that is non-degenerate on some non-empty, connected set $\mathcal{O} \subset T\R^d \cong \R^d \times \R^d$.
	When no alternative Lagrangian to $\mathring{L}$ exists, then any Lagrangian $L$ with the property
	\[
	\mathrm{EL}(\mathring L)(x(t),\dot{x}(t),\ddot{x}(t)) =0  \implies \mathrm{EL}(L)(x(t),\dot{x}(t),\ddot{x}(t)) =0
	\]
	on $\mathcal{O}\times \R^d$ is either a null-Lagrangian (i.e.\ $\mathrm{EL}(L) \equiv 0$) or is non-degenerate on $\mathcal{O}$.
\end{proposition}

\begin{proof}
	As no alternative Lagrangian exists, there must be $\rho,c \in \R$ and $F\colon \R^d \to \R$ such that on $\mathcal{O}$
	\[
	L = \rho \mathring L + \d_t F + c.
	\]
	If $L$ is not a null-Lagrangian on $\mathcal{O}$, there must be $\hat{x} \in \mathcal{O} \times \R^d$ with $\mathrm{EL}(L)(\hat x) \not =0$. Let $1\le k\le d$ such that $(\mathrm{EL}(L)(\hat x))_k \not =0$.
	By \cref{lem:TrafoGeomQuantities}
	\[
	0 \not = (\mathrm{EL}(L)(\hat x))_k
	= \rho (\mathrm{EL}(\mathring L)(\hat x))_k.
	\]
	Thus $\rho \not =0$. Non-degeneracy on $\mathcal O$ follows from $\mathrm{Vol}(L)=\rho^d\mathrm{Vol}(\mathring L)$.
\end{proof}

\begin{remark}
	Under genericity assumptions on the dynamics with $d \ge 2$, no alternative Lagrangians exist \cite{HENNEAUX198245}. If a generic dynamical system is governed by a non-degenerate Lagrangian, then any Lagrangian $L$ with $\mathrm{EL}(L)=0$ on all motions that is non-degenerate anywhere, is non-degenerate everywhere.
\end{remark}

Refer to \cref{prop:NormaliseLNonlinear} of \cref{app:SymplVolNormalise} for an alternative normalisation strategy for Lagrangians based on normalising symplectic volume. 
It is comparable to techniques developed in \cite{DLNNPDE} for neural network models of Lagrangians.

The following \namecref{prop:HSymWelldefined} implies that {the Euler--Lagrange operator (and thus the representation of the equation of motions) and the} Hamiltonian and symplectic structure are uniquely determined when the normalisation condition \eqref{eq:normLCond} is fulfilled, provided that no alternative Lagrangians exist.

\begin{proposition}\label{prop:HSymWelldefined}
Let $\mathring L$ be a Lagrangian on $\gls{TRd}$ with \eqref{eq:normLCond} for some $\gls{xb}=(x_b,\dot x_b) \in T\R^d$, $1\le k \le d$, $\gls{cb}\in \R$, $\gls{pb} \in \R^d$, $c_\tau \in \R \setminus \{0\}$.
Then for any Lagrangian $L$ {with \eqref{eq:normLCond} that is} equivalent to $\mathring L$ we have
\[
{\mathrm{EL}(L) = \mathrm{EL}(\mathring{L})}, \quad
\mathrm{Ham}(L) = \mathrm{Ham}(\mathring{L}), \quad
\mathrm{Sym\mod{pl}}(L) = \mathrm{Sym\mod{pl}}(\mathring L).
\]
\end{proposition}

\begin{proof}
$L$ is of the form $L=\rho {\mathring L} + \d_t F + c$. The last condition of \eqref{eq:normLCond} implies $\rho=1$. Thus {$\mathrm{EL}(L) = \mathrm{EL}(\mathring{L})$ and} $\mathrm{Sym\mod{pl}}(L) = \mathrm{Sym\mod{pl}}(\mathring L)$ by \cref{lem:TrafoGeomQuantities}. With $\rho=1$ and the first two conditions \eqref{eq:normLCond} we have
\[
\mathrm{Ham}(L)(\overline{x}_b) = \dot x_b^\top p_b - c_b = \mathrm{Ham}(\mathring{L})(\overline{x}_b).
\]
Then $\mathrm{Ham}(L)= \mathrm{Ham}(\mathring{L})$ follows by \cref{lem:TrafoGeomQuantities}.
\end{proof}

For discrete Lagrangians, we have the following analogy to \cref{prop:NormaliseL}.

\begin{proposition}\label{prop:NormaliseLd}
	Let $\gls{xb} = (x_{0b}, x_{1b}) \in ( \R^d)^2$, $\hat{x}_\tau = (x_{0\tau}, x_{1\tau},x_{2\tau}) \in ( \R^d)^3$ and $\mathring L_d$ a discrete Lagrangian with $\mathrm{DEL}(L_d)(\hat{x}_b) \not =0$.
	Let $\gls{cb} \in \R$, $\gls{pb} \in \R^d$, $c_\tau \in \R\setminus\{0\}$. There exists a discrete Lagrangian $L_d$ such that $L_d$ is equivalent to $\mathring{L}_d$ and
\begin{equation}\label{eq:normLdCond}
	L_d(\overline{x}_b)=c_b,
	\quad 
	\mathrm{Mm}^+(L_d)(\overline{x}_b) = p_b,
	\quad
	(\mathrm{DEL}(L_d)(\hat{x}_\tau))_k = c_\tau,
\end{equation}
where $1\le k\le d$ can be chosen as any index for which the component of $\mathrm{DEL}(\hat{x}_b)$ is not zero.
\end{proposition}

\begin{proof}
	Let $\mathring{c_b} = \mathring L_d(\overline{x}_b)$, $\mathring{p_b}=\mathrm{Mm}^+(\mathring L_d)(\overline{x}_b)$, $\mathring{c_\tau} = (\mathrm{DEL}(\mathring L_d)(\hat{x}_b))_k$.
	We set
\[
\rho = \frac{c_\tau}{\mathring{c_\tau}},
\quad
	F(x) = x^\top (p_b- \rho \mathring{p_b}),
	\quad 
	c = c_b -\rho \mathring{c_b}  - (x_{1b}-x_{0b})^\top (p_b- \rho \mathring{p_b}).
	\]
Now the Lagrangian $L_d = \rho \mathring{L}_d + \Delta_t F + c$ is equivalent to $L_d$ and fulfils \eqref{eq:normLdCond}.
\end{proof}

\begin{remark}
A statement similar to \cref{prop:NormaliseLd} holds true with $\mathrm{Mm}^-$ replacing $\mathrm{Mm}^+$. Moreover, a statement in analogy to \cref{prop:NOmegaImpliesNonDeg} can be obtained with discrete quantities replacing their continuous counterparts. The details shall not be spelled out in this context.
Moreover, an alternative normalisation strategy based on regularising the discrete symplectic volume is provided in \cref{prop:NormaliseLdNonlinear} in \cref{app:SymplVolNormalise}, where it is also compared to regularisation strategies in the neural network context of \cite{DLNNPDE}.
\end{remark}

\subsection{{Utilisation in a data-driven context}}\label{sec:UtiliseNormalisation}

In the following section, we will consider the inverse problem of inferring a Lagrangian or discrete Lagrangian from motion data. For this, we will augment the inverse problem by normalisation conditions \eqref{eq:normLCond} or \eqref{eq:normLdCond}, respectively, for values of $c_b\in \R$, $p_b \in \R^d$, and $c_\tau \in \R\setminus\{0\}$. \Cref{prop:NormaliseL} or \cref{prop:NormaliseLd} show that this augmentation does not restrict the generality of the ansatz.

Although the conditions together with the true dynamics do not determine the (discrete) Lagrangian uniquely, they do determine the Euler--Lagrange operator $\mathrm{EL}(L)$ as well as the Hamiltonian and symplectic structure, provided that the true dynamical system does not admit alternative Lagrangians (\mod{\cref{prop:HSymWelldefined}}).
When only limited data is observed, there is some uncertainty in the equations of motions $\mathrm{EL}(L)=0$, the Hamiltonian, symplectic structure, or any linear observable in $L$ that we want to quantify. The normalisation conditions eliminate any artificial uncertainty stemming from an ambiguous representation of the model.

{Moreover, when all true Lagrangians are non-degenerate, so is the sought Lagrangian in the augmented inverse problem (\cref{prop:NOmegaImpliesNonDeg}). Thus, the normalisation conditions safeguard us from inferring degenerate Lagrangians that are consistent with the observed motion data but fail to discriminate non-motions.}

%% file: method.tex
\section{Data-driven method}\label{sec:MLSetting}

\subsection{Bayesian learning of continuous Lagrangians}\label{sec:MLContinuous}

In the following, we present a framework for learning a continuous Lagrangian from observations of a dynamical system.

Let $\gls{Omega} \subset T\R^d \cong \R^d \times \R^d$ be an open, bounded subset. Our goal is to identify a Lagrangian $L \colon \Omega \to \R$ based on observations $ \hat x = (\overline{x},\ddot x)  = (x,\dot x,\ddot x) \in  \Omega \times \R^d$ for which $\mathrm{EL}(L)(\hat{x})=0$ on all observations $\hat{x}$ such that the dynamics \eqref{eq:ELODE} to $L$ approximate the dynamics of an unknown true Lagrangian $\gls{LRef} \colon \Omega \to \R$.
{We interpret this task as seeking a solution to the Euler--Lagrange equation \eqref{eq:ELODE} that we interpret as a partial differential equation for $L$. We follow a Bayesian approach proposed in \cite{OwhadiLearningPDEGP} and assume a Gaussian field (see \cref{app:GaussianField} for definitions) as a prior for $L$ that we condition on fulfilling the Euler--Lagrange equation \eqref{eq:ELODE} on the data points and on regularisation conditions to obtain a posterior distribution for $L$. Even though in contrast to \cite{OwhadiLearningPDEGP} our partial differential equation is highly ill-posed, we prove in \cref{sec:ConvergenceAnalysis} that the posterior mean converges against a true Lagrangian of the motions in the infinite data limit.}


\subsubsection{\Gls{RKHS} set-up and Gaussian {fields}}\label{sec:RKHSSetup}
We consider the following set-up that makes use of the theory of reproducing kernel Hilbert spaces (RKHS). Refer to \cite{ChristmannSteinwart2008RKHS,OwhadiScovel2019} for background material.

Consider a symmetric function $\gls{K} \colon \Omega \times \Omega \to \R$. Assume that $K$ is positive definite, i.e.\ for all finite subsets $\{\overline{x}^{(j)}\}_{j=1}^{\gls{M}} \subset \Omega$ the matrix $(K(\overline{x}^{(i)},\overline{x}^{(j)}))_{i,j=1}^{\gls{M}}$ is positive definite. $\gls{K}$ is called {\em kernel}.

Consider the reproducing kernel Hilbert space (\gls{RKHS}) $\gls{U}$ to $\gls{K}$, i.e.\ 
consider the inner product space
\[
\mathring{U} =
\left\{ L = \sum_{j=1}^{n} \alpha_j K(\overline{x}^{(j)},\cdot ) \; | \; \alpha_j \in \R, n \in \N_0, \overline{x}^{(j)} \in \Omega \right\}
\]
with inner product defined as the linear extension of
\[
\langle K(\overline{x}, \cdot ) , K(\overline{y}, \cdot ) \rangle = K(\overline{x},\overline{y}).
\]
Then the Hilbert space $U$ is obtained as the topological closure of $\mathring{U}$ with respect to $\langle \cdot , \cdot \rangle$.
%
We denote the \mod{topological} dual space of $U$ by $U^\ast$. We define the map
\begin{equation}\label{eq:MathcalKappa}
	\gls{MathcalKappa}\colon U^\ast \to U, \quad \mod{\gls{phi}} \mapsto \mathcal{K}(\mod{\phi}) \text{ with } \mathcal{K}(\mod{\phi})(x) = \mod{\phi}(K(x,\cdot)).
\end{equation}
The map $\mathcal{K} \colon U^\ast \to U$ is linear, bijective, and symmetric, i.e.\ $\mod{\psi}(\mathcal{K}(\mod{\phi}))=\mod{\phi}(\mathcal{K}(\mod{\psi}))$ for $\mod{\gls{phi}},\mod{\gls{psi}} \in U^\ast$, and positive, i.e.\ $\mod{\phi}(\mathcal{K}(\mod{\phi})) > 0$ for $\mod{\phi} \in U^\ast \setminus \{0\}$.



Consider the {\em canonical Gaussian {field}}\footnote{\label{ft:GPvsGaussianFields}{The notion of a {\em Gaussian field} differs from the notion of a {\em Gaussian process} \cite[Def.3]{DaCosta2024}. See \cite[§3.5-§4, paragraph~1]{pfoertner2024} for further explanation. However, the literature refers to methods that solve pdes using the concept of Gaussian fields as {\em Gaussian processes based} methods (e.g.~\cite{OwhadiLearningPDEGP,Owhadi2023ErrorAnalysisPDEGP}).
		}}
$\gls{xi} \mod{\sim} \gls{Normal}(0,\mathcal K)$ on $U$, which is a weak random variable with the following properties:
\begin{itemize}
	\item {For all $\phi \in U^\ast$, $\phi(\xi) \sim \mathcal{N}(0,\phi(\mathcal{K}(\phi)))$ is a centred Gaussian random variable.}
	\item
	{Moreover,} for any finite collection $\gls{Phi}=(\mod{\phi}_1,\ldots,\mod{\phi}_n)$ with $\mod{\phi}_j \in U^\ast$ for $1\le j \le n$, the random variable $\Phi(\xi){=(\mod{\phi}_1(\xi),\ldots,\mod{\phi}_n(\xi))}$ is multivariate-normally distributed $\Phi(\xi) \in \mathcal{N}(0,\kappa)$
	with covariance matrix given as $\kappa=(\mod{\phi}_i(\mathcal{K}(\mod{\phi}_j)))_{i,j=1}^n$.
\end{itemize}

{See \cref{app:GaussianField} for a formal definition of Gaussian fields and existence statements recalled from \cite{OwhadiScovel2019}.	}

\subsubsection{Data}
Assume we observe distinct data points $\hat x^{(j)} = (\overline{x}^{(j)},\ddot x^{(j)}) = (x^{(j)},\dot x^{(j)}, \ddot x^{(j)}) \in \Omega \times \R^d$, $j=1,\ldots,\gls{M}$ of Lagrangian motions.
Define $\mathrm{EL}_{\hat{x}^{(j)}} \colon U \to \R^d$ as
\begin{equation}\label{eq:ELxhat}
	\gls{ELxhatj}(L)
	=\mathrm{EL}(L)(\hat{x}^{(j)})
	=\frac{\p^2 L(\overline{x}^{(j)})}{\p x \p \dot x} \ddot{x}^{(j)} 
	+ \frac{\p^2 L(\overline{x}^{(j)})}{\p x \p x} \dot{x}^{(j)}
	- \frac{\p L(\overline{x}^{(j)})}{\p x}
\end{equation}
for $1 \le j \le M$.
Furthermore, let $\gls{xb} = (x_b,\dot x_b) \in \Omega$ and consider $\mathrm{Mm}_{\overline{x}_b} \colon U \to \R^d$ defined as
\begin{align}\label{eq:MmEvaluate}
	\gls{Mm}_{\overline{x}_b}(L) 
	= \mathrm{Mm}(L)(\overline{x}_b)
	= \frac{\p L}{\p \dot x}(\overline{x}_b).
\end{align}
Moreover, let $\mathrm{ev}_{\overline{x}_b} \colon U \to \R$ with
\begin{equation}\label{eq:evalFunc}
\gls{evxb}(L) = L(\overline{x}_b)
\end{equation}
denote the evaluation functional. Collect these functionals in a linear map $\Phi_b^{\mod{(}M\mod{)}} \colon U \to (\R^d)^M \times \R^d \times \R$
\begin{equation}\label{eq:PhibM}
	\gls{Phibm} = (\mathrm{EL}_{\hat{x}^{(1)}},\ldots,\mathrm{EL}_{\hat{x}^{(M)}},\mathrm{Mm}_{\overline{x}_b},\mathrm{ev}_{\overline{x}_b}).
\end{equation}
For constants $c_b \in \R$, $p_b \in \R^d$
let
\begin{equation}\label{eq:ybM}
\gls{ybm} = (\underbrace{0,\ldots,0}_{M \text{ times } \mod{0\in \R^d}},p_b,c_b) \in (\R^d)^M   \times \R^d \times \R.
\end{equation}

{\em Interpretation:} When $\Phi_b^{\mod{(}M\mod{)}}(L) = y_b^{\mod{(}M\mod{)}}$ for some $L \in U$, then $L$ is consistent with the dynamical data and fulfils the normalisation conditions $\mathrm{Mm}(L)(\overline{x}_b) = p_b,L(\overline{x}_b)=c_b$. The condition $(\mathrm{EL}(L)(\overline x_b))_k=c_\tau$ of \cref{prop:NormaliseL} is left out due to practical considerations that will be discussed later -- see \cref{rem:MoreNormalisation}.

\subsubsection{Lagrangian as a conditional mean of Gaussian {fields}}\label{sec:LasCondGP}
Let us introduce the formulas required to infer a Lagrangian from data and predict uncertainty in the identified equations of motions and other linear observables such as Hamiltonian or symplectic structure. We postpone to \cref{sec:ConvergenceAnalysis} a more detailed derivation and a justification of applicability of the theory of Gaussian fields, such as the boundedness of certain operators. The following considers the noise-free case.

We will make use of the following assumptions that are fulfilled when the observed system is governed by the Euler--Lagrange equations to a non-degenerate Lagrangian $L \in \mathcal{C}^2(\overline{\Omega})$ and when $K$ is the square exponential kernel $K(\overline{x},\overline{y}) = \exp(-\|\mod{\overline{x}}-\mod{\overline{y}}\|^2/l)$, $l>0$ and $\Omega$ is an \mod{open, bounded and} locally Lipschitz domain (\cref{rem:AssumptionTrueForK}). \mod{Here $\mathcal{C}^2(\overline{\Omega})$ denotes the space of twice continuously differentiable functions on $\Omega$ for which all derivatives extend continuously to the topological closure $\overline{\Omega}$.}
\begin{assumption}\label{assumption:MLSetupAssumption}
Assume that
\[
\{
L \in \mathcal{C}^2(\overline{\Omega}) \, | \, \Phi_b^{\mod{(}M\mod{)}}(L) = y_b^{\mod{(}M\mod{)}}\} \cap U \not =\emptyset
\]
and that the RKHS $U$ to kernel $K$ embeds continuously into $\mathcal{C}^2(\overline{\Omega})$. Let $K$ be four times continuously differentiable.
\end{assumption}

By general theory {recalled in \cref{app:GaussianField}}, the posterior distribution \mod{$\xi_M$} of the canonical Gaussian {field} $\xi$ conditioned on the {bounded}\footnote{\label{foot:KernelAssumptions}{$\Phi_b^{\mod{(}M\mod{)}} \colon C^2(\overline{\Omega}) \to \R^{(M+1)d+1}$ is bounded (\cref{sec:FormalSettingCont}). 
} }
linear constraint $\Phi_b^{\mod{(}M\mod{)}}(\mod{\xi}) = y_b^{\mod{(}M\mod{)}}$ is again a Gaussian field $\gls{xiM}=\mathcal{N}(L_{\mod{(M)}},\mathcal{K}_{\Phi_b^{\mod{(}M\mod{)}}})$. It is characterised by the conditional mean $L_{\mod{(M)}}$ and the conditional covariance operator $\mathcal{K}_{\Phi_b^{\mod{(}M\mod{)}}}$.
To compute $L_{\mod{(M)}}$ and $\mathcal{K}_{\Phi_b^{\mod{(}M\mod{)}}}$, define the symmetric matrix
\[
\Theta \in \R^{((M+1)d+1) \times ((M+1)d+1)}, \quad \Theta_{k,l} = (\Phi_b^{\mod{(}M\mod{)}})_k\mathcal{K}(\Phi_b^{\mod{(}M\mod{)}})_l,
\quad 1 \le k,l \le (M+1)d+1,
\]
where $(\Phi_b^{\mod{(}M\mod{)}})_k$, $(\Phi_b^{\mod{(}M\mod{)}})_l$ refer to the $k$th or $l$th component of $\Phi_b^{\mod{(}M\mod{)}}$, respectively. In block matrix form, $\Theta$ can be written as
\begin{equation}\label{eq:ThetaMat}
	\gls{Theta} = \begin{pmatrix}
		(\gls{ELxhatj1}\gls{ELxhatj2}K)_{ij}
		&(\mathrm{EL}^1_{\hat{x}^{(j)}}\gls{Mm}^2_{\overline{x}_b}K)_{j}
		&(\mathrm{EL}^1_{\hat{x}^{(j)}} \mathrm{ev}^2_{\overline{x}_b}K)_j\\
		(\gls{Mm}^1_{\overline{x}_b}\mathrm{EL}^2_{\hat{x}^{(i)}}K)_{i}
		&\mathrm{Mm}^1_{\overline{x}_b}\mathrm{Mm}^2_{\overline{x}_b}K
		&\mathrm{Mm}^1_{\overline{x}_b} \mathrm{ev}^2_{\overline{x}_b}K\\
		(\mathrm{ev}^1_{\overline{x}_b}\mathrm{EL}^2_{\hat{x}^{(i)}}K)_{i}
		&\mathrm{ev}^1_{\overline{x}_b}\mathrm{Mm}^2_{\overline{x}_b}K
		&K(\overline{x}_b,\overline{x}_b).
	\end{pmatrix}
\end{equation}
The upper indices $1,2$ of the operator indicate their action on the first or second component of the kernel $K \colon \Omega \times \Omega \to \R$, i.e.\
\begin{equation}\label{eq:EL12}
\mathrm{EL}^1_{\hat{x}^{(j)}}\mathrm{EL}^2_{\hat{x}^{(i)}}K
=\mathrm{EL}_{\hat{x}^{(j)}}\big(\overline x\mapsto\mathrm{EL}_{\hat{x}^{(i)}}(\overline y \mapsto  K(\overline{x},\overline{y}))\big) \quad \in \R
\end{equation}
with analogous conventions for $\mathrm{Mm}$ and $\mathrm{ev}$.
Furthermore, we use the convention that when an operator $\mathrm{EL}$, $\mathrm{Mm}$, or $\mathrm{ev}$ is applied to functions with several components their application \mod{is} understood component-wise. With
\[
\mathcal{K}\Phi_b^{\mod{(}M\mod{)}}(\overline{x}) = \begin{pmatrix}
	\mathrm{EL}_{\hat{x}^{(1)}}K(\cdot,\overline{x}),
	& \ldots\;\mod{,} &
	\mathrm{EL}_{\hat{x}^{(M)}}K(\cdot,\overline{x}),
	&
	\mathrm{Mm}_{\overline{x}_b}K(\cdot,\overline{x}),
	& K(\overline{x}_b,\overline{x})
\end{pmatrix}^\top
\]
the conditional mean $\mod{\gls{LM}}$ of the posterior process $\xi_M$ is given as
\begin{equation}\label{eq:LPosterior}
	\mod{\gls{LM}} = {y_b^{\mod{(}M\mod{)}}}^\top {\Theta^{\dagger}} \mathcal{K}\Phi_b^{\mod{(}M\mod{)}},
\end{equation}
{where $\Theta^{\dagger}$ denotes the pseudo-inverse of $\Theta$.}
The conditional covariance operator $\mathcal{K}_{\Phi_b^{\mod{(}M\mod{)}}} \colon {U}^\ast \to U$ is given by
\begin{align}\label{eq:CovPosterior}
	&\psi\mathcal{K}_{\Phi_b^{\mod{(}M\mod{)}}}\phi
	= \psi\mathcal{K}\phi
	- (\psi\mathcal{K}{\Phi_b^{\mod{(}M\mod{)}}}^\top) {\Theta^{\dagger}} (\Phi_b^{\mod{(}M\mod{)}}\mathcal{K}\phi)
\end{align}
for any $\psi,\phi \in U^\ast$. Here
\begin{align*}
\psi\mathcal{K}_{\Phi_b^{\mod{(}M\mod{)}}}\phi &=\psi^1 \phi^2 K\\
\psi\mathcal{K}{\Phi_b^{\mod{(}M\mod{)}}}^\top &=	\begin{pmatrix}
		\psi^1\mathrm{EL}^2_{\hat{x}^{(2)}}K,
		& \ldots &
		\psi^1\mathrm{EL}^2_{\hat{x}^{(n)}}K,
		&
		\psi^1\mathrm{Mm}^2_{\overline{x}_b}K,
		& \psi^1 K(\cdot ,\overline{x}_b)
	\end{pmatrix}\\
\Phi_b^{\mod{(}M\mod{)}}\mathcal{K}\phi &=
	\begin{pmatrix}
		\mathrm{EL}^1_{\hat{x}^{(2)}}\phi^2K, &
		\ldots &
		\mathrm{EL}^1_{\hat{x}^{(n)}}\phi^2K, &
		\mathrm{Mm}^1_{\overline{x}_b}\phi^2K & 
		\phi^2 K(\overline{x}_b,\cdot )
	\end{pmatrix}^\top.
\end{align*}
Again, the upper indices $1,2$ of the linear functionals $\phi, \psi \in U^\ast$ denote actions on the first or second component of $K$, respectively.

The expressions ${y_b^{\mod{(}M\mod{)}}}^\top \Theta^{\dagger}$ and $\Theta^{\dagger} (\Phi_b^{\mod{(}M\mod{)}}\mathcal{K}\phi)$ in \eqref{eq:LPosterior} and \eqref{eq:CovPosterior}, respectively, are least-square solutions to the linear systems
\begin{equation}\label{eq:LinSysTheta}
\Theta z = y_b^{\mod{(}M\mod{)}} \quad \text{and} \quad 
\Theta Z = \Phi_b^{\mod{(}M\mod{)}} \mathcal{K}\phi
\end{equation}
for $z$ and $Z$. It is argued in \cref{app:CondExpandVar} and \cref{sec:ApplyPropCondDist} that these systems are solvable and that \eqref{eq:LPosterior} \mod{is} valid. Moreover, $\Theta^{\dagger} (\Phi_b^{\mod{(}M\mod{)}}\mathcal{K}\phi)$ and $\Theta^{\dagger} (\Phi_b^{\mod{(}M\mod{)}}\mathcal{K}\phi)$ can be substituted by any solution to the linear systems above without changing $L$ in \eqref{eq:LPosterior} or $\psi\mathcal{K}_{\Phi_b^{\mod{(}M\mod{)}}}\phi$ in \eqref{eq:CovPosterior}.

\begin{remark}[Computational {aspects}]\label{rem:CompAspects}
	{The size of the linear systems \eqref{eq:LinSysTheta} scales linearly with the number of data points and the dimension of the state-space. Thus the numerical complexity of solving the linear systems scales approximately cubically, when a direct method is used. The growth in computational complexity is typical for Gaussian process or kernel-based methods \cite{RasmussenWilliams2005}. 
	To tackle this, various approaches exist such as using kernels of finite band-width to promote sparsity of $\Theta$, importance sampling, and sparse Gaussian processes which are based on identifying inducing variables \cite{Titsias2009,QuinRasmussen2005}. An efficient method to approximate Cholesky factors of covariance matrices was presented in \cite{Schaefer2021}. Moreover, a diagonal regularisation technique involving an adaptive nudging term can be found in \cite[Appendix A]{OwhadiLearningPDEGP} in the context of solving pdes with Gaussian processes. A more specialised approach is \cite{Schaefer2021b}.
	In our numerical experiments (\cref{sec:Experiment}) we do not employ any specialised algorithm but use the command {\ttfamily factorize} of the package Julia/LinearAlgebra \cite{bezanson2017julia} on $\Theta$. Depending on the degeneracy of the symmetric matrix $\Theta$, {\ttfamily factorize} computes a Cholesky decomposition or a factorisation based on the Bunch-Kaufman algorithm \cite{BarwellGeorge1976,BunchKaufmann1977}. The factors are then stored and used whenever solving linear systems involving $\Theta$.
	}
\end{remark}

\begin{remark}[Equivalent minimisation problem]\label{rem:GPisMinProblem}
	{T}he conditional mean $L_{\mod{(M)}}$ of \eqref{eq:LPosterior} can alternatively be characterised as the minimiser of the following convex optimisation problem
	\begin{equation}\label{eq:LasConvexOpt}
	L_{\mod{(M)}}=\arg \min_{\mod{L} \in U, \Phi_b^{\mod{(}M\mod{)}}( \mod{L})=y_b^{\mod{(}M\mod{)}}} \|  \mod{L} \|_U,
	\end{equation}
	where $\| \mod{\cdot}  \|_U$ denotes the reproducing kernel Hilbert space norm. ({See \cref{thm:Extremizer} in \cref{app:GaussianField}.)} This will play an important role in the convergence proof in \cref{sec:ConvergenceAnalysis}. Besides the exploit for convergence proofs, formulation \eqref{eq:LasConvexOpt} could be used for the computation of the conditional stochastic processes for non-linear observations and normalisation conditions such as in the alternative regularisation of \cref{app:SymplVolNormalise} using techniques of \cite{OwhadiLearningPDEGP}.
\end{remark}

\begin{remark}[Further normalisation]\label{rem:MoreNormalisation}
	For consistency with \cref{prop:NormaliseL}, one may add
	$c_\tau \in \R \setminus \{0\}$ to $y_b^{\mod{(}M\mod{)}}$ and 
	the normalising condition $(\mathrm{EL}_{(\hat{x}_\tau)})_k$ to $\Phi^{\mod{(}M\mod{)}}_b$ for $\hat{x}_\tau = (x_\tau,\dot x_\tau, \ddot x_\tau)$ that is not a motion and $k \in \{1,\ldots,d\}$.
	While it is realistic to assume knowledge of a data point $\hat{x}_\tau$ that is not a motion (e.g.\ $\hat x = (\overline x^{(1)},\ddot x^{(1)}+1)$ in systems with non-degenerate true Lagrangian), fixing an index $k$ a priori may cause a restriction as to which Lagrangians can be approximated or cause poor scaling of the posterior process.
	Thus, we propose to leave out this condition in the definition of the posterior process.
	One may rather verify $c_\tau \not =0$ a posteriori to check validity of the assumptions of \cref{prop:NormaliseL}.
	Moreover, \cref{app:SymplVolNormalise} discusses an alternative normalisation based on symplectic volume forms. It can be compared to approaches to learn Lagrangians with neural networks \cite{DLNNPDE}.
\end{remark}

\subsubsection{Application}
The conditional mean $L_{\mod{(M)}}$ \eqref{eq:LPosterior} of the posterior Gaussian process $\xi_{\mod{M}}$
serves as an approximation to a true Lagrangian, from which approximations of geometric structures such as symplectic structure and Hamiltonians can be derived. Moreover, uncertainties of linear observables $\psi \in U^\ast$ can be quantified as the variance of $\psi(\xi_{\mod{M}})$, which can be computed as $\psi\mathcal{K}_{\Phi_b^{\mod{(}M\mod{)}}}\psi$ using \eqref{eq:CovPosterior}.
In the numerical experiments, standard deviations will be computed for the random variables $\mathrm{Ham}(\xi_{\mod{M}})(\overline{x})$ for $\overline{x} \in \Omega$
and for $\mathrm{EL}(\xi_{\mod{M}})(\hat{x})$, where $\hat{x} = (x,\dot x,\ddot x)$ is a motion of the system to $L_{\mod{(M)}}$.

\subsection{Gaussian fields for discrete Lagrangians}\label{dec:GPLd}

The data-driven framework for learning of discrete Lagrangians is in close analogy to the presented framework for continuous Lagrangians.
\mod{Recall that the use of discrete Lagrangian models does not cause any discretisation error as true discrete Lagrangian models exist for motions that are governed by a continuous variational principle for a first order Lagrangian, see \cref{Intro:DiscreteLagrangianModel}.}
Instead of repeating the discussion \mod{of \cref{sec:MLContinuous}}, we explain the required modifications and reinterpretations in the following.
A rigorous discussion and justification of the applicability of the theory of Gaussian fields is postponed to \cref{sec:ConvergenceAnalysisLd}.

In the setting of discrete Lagrangians, $\gls{Omega} \subset \R^d \times \R^d$ is an open, bounded subset containing elements denoted by $\gls{xbar} = (x_0,x_1)$. Observed data corresponds to a collection of $M$ triples of snapshots $\gls{xhat}^{(j)} = (x_0^{(j)},x_1^{(j)},x_2^{(j)})$ of motions of a variational dynamical system, where $(x_0^{(j)},x_1^{(j)}) \in \Omega$ and $(x_1^{(j)},x_2^{(j)}) \in \Omega$ for all $j$. The snapshot time (discretisation parameter) $\mod{\Delta t}>0$ is constant (also see \cref{fig:DataOscillatorDiscrete}).
The goal is to identify a discrete Lagrangian $L_{d,\mod{(M)}} \colon \Omega \to \R$ such that discrete motions that fulfil the discrete Euler-Lagrange equations $\mathrm{DEL}(L_{d,\mod{(M)}}) =0$ approximate true motions.

\mod{In analogy to \eqref{eq:PhibM}, let $\overline{x}_b \in \Omega$ and consider the functional $\gls{Phibm}\colon \mathcal{C}^1(\overline{\Omega}) \to (\R^d)^M \times \R^d \times \R$ defined as}
\begin{equation}\label{eq:PhibmDiscrete}
	\gls{Phibm} = (\gls{DELxhatj},\ldots,\mathrm{DEL}_{\hat{x}^{(M)}},\mathrm{Mm^-}_{\overline{x}_b},\mathrm{ev}_{\overline{x}_b}).
\end{equation}
\mod{Moreover, to $p_b \in \R^d$, $c_b \in \R$ let}
\begin{equation}\label{eq:ybmDiscrete}
	\mod{\gls{ybm} = (\underbrace{0,\ldots,0}_{M \text{ times }0\in \R^d},p_b,c_b).}
\end{equation}

Consider a twice continuously differentiable kernel $K \colon \Omega\times \Omega \to \R$ with RKHS $U$.
{We consider the following assumptions that are fulfilled when the observed system is governed by the Euler--Lagrange equations to a non-degenerate Lagrangian $L_{\mod{d}} \in \mathcal{C}^1(\overline{\Omega})$ and when $K$ is the square exponential kernel $K(\overline{x},\overline{y}) = \exp(-\|\mod{\overline x}-\mod{\overline y}\|^2/l)$, $l>0$ and $\Omega$ is an \mod{open, bounded, and} locally Lipschitz domain:
	\begin{assumption}\label{assumption:MLSetupAssumptionLd}
		Assume that
		\[
		\{
		L_d \in \mathcal{C}^1(\overline{\Omega}) \, | \, \Phi_b^{\mod{(}M\mod{)}}(L_{\mod{d}}) = y_b^{\mod{(}M\mod{)}}\} \cap U \not =\emptyset
		\]
		and that the RKHS $U$ to kernel $K$ embeds continuously into $\mathcal{C}^1(\overline{\Omega})$. Let $K$ be twice continuously differentiable.
	\end{assumption}
}

With the reinterpretation of $\Omega$ and of training data points $\hat x^{(j)}$ we can follow the framework for continuous Lagrangians replacing $\mathrm{EL}$ by $\mathrm{DEL}$ and $\mathrm{Mm}$ by $\mathrm{Mm}^-$ (or $\mathrm{Mm}^+$). In particular, this leads to


\begin{equation}\label{eq:ThetaMatLd}
	\gls{Theta} = \begin{pmatrix}
		(\gls{DELxhatj1}\gls{DELxhatj2}K)_{ij}
		&(\mathrm{DEL}^1_{\hat{x}^{(j)}}\gls{MmMinus}^2_{\overline{x}_b}K)_{j}
		&(\mathrm{DEL}^1_{\hat{x}^{(j)}} \mathrm{ev}^2_{\overline{x}_b}K)_j\\
		(\gls{MmMinus}^1_{\overline{x}_b}\mathrm{DEL}^2_{\hat{x}^{(i)}}K)_{i}
		&\mathrm{Mm^-}^1_{\overline{x}_b}\mathrm{Mm^-}^2_{\overline{x}_b}K
		&\mathrm{Mm^-}^1_{\overline{x}_b} \mathrm{ev}^2_{\overline{x}_b}K\\
		(\mathrm{ev}^1_{\overline{x}_b}\mathrm{DEL}^2_{\hat{x}^{(i)}}K)_{i}
		&\mathrm{ev}^1_{\overline{x}_b}\mathrm{Mm^-}^2_{\overline{x}_b}K
		&K(\overline{x}_b,\overline{x}_b).
	\end{pmatrix}
\end{equation}
(cf.~\eqref{eq:ThetaMat})
and a conditioned process that is a Gaussian process $\mathcal{N}(\mod{\gls{LdM}},\mathcal{K}_{\Phi_b^{\mod{(}M\mod{)}}})$ with posterior mean
\begin{equation}\label{eq:LdPosterior}
	\mod{\gls{LdM}} = {y_b^{\mod{(}M\mod{)}}}^\top {\Theta}^{\dagger} \mathcal{K}\Phi_b^{\mod{(}M\mod{)}}
\end{equation}
(cf.~\eqref{eq:LPosterior}).
Again, the upper index $1,2$ of the operators $\mathrm{DEL}$, $\mathrm{Mm}^-$, $\mathrm{ev}$ denote on which input element of $K$ they act. 
The conditional covariance operator $\mathcal{K}_{\Phi_b^{\mod{(}M\mod{)}}} \colon U^\ast \to U$ is defined for any $\psi,\phi \in U^\ast$ by
\begin{align}\label{eq:CovPosteriorLd}
	&\psi\mathcal{K}_{\Phi_b^{\mod{(}M\mod{)}}}\phi
	= \psi\mathcal{K}\phi
	- (\psi\mathcal{K}{\Phi_b^{\mod{(}M\mod{)}}}^\top) {\Theta}^{\dagger} (\Phi_b^{\mod{(}M\mod{)}}\mathcal{K}\phi).
\end{align}
Here
\begin{align*}
\psi\mathcal{K}_{\Phi_b^{\mod{(}M\mod{)}}}\phi	&=\psi^1 \phi^2 K\\
	\psi\mathcal{K}{\Phi_b^{\mod{(}M\mod{)}}}^\top &= \begin{pmatrix}
		\psi^1\mathrm{DEL}^2_{\hat{x}^{(2)}}K,
		& \ldots &
		\psi^1\mathrm{DEL}^2_{\hat{x}^{(n)}}K,
		&
		\psi^1\mathrm{Mm^-}^2_{\overline{x}_b}K,
		& \psi^1 K(\cdot ,\overline{x})
	\end{pmatrix}\\
	\Phi_b^{\mod{(}M\mod{)}}\mathcal{K}\phi &=
	\begin{pmatrix}
		\mathrm{DEL}^1_{\hat{x}^{(2)}}\phi^2K &
		\ldots &
		\mathrm{DEL}^1_{\hat{x}^{(n)}}\phi^2K &
		\mathrm{Mm^-}^1_{\overline{x}_b}\phi^2K &
		\phi^2 K(\overline{x},\cdot )
	\end{pmatrix}^\top.
\end{align*}

{To obtain \eqref{eq:LdPosterior} and \eqref{eq:CovPosteriorLd} we have (as in the continuous case) applied general theory as recalled in \cref{prop:CondDistrGeneral} in \cref{app:CondExpandVar}. Indeed, conditions for the applicability of \cref{prop:CondDistrGeneral} are verified in \cref{prop:CondDistPropLdCheck} (\cref{app:CondExpandVar}). }

%% file: numerical_experiments_continuous.tex

\subsection{Continuous Lagrangians}
Consider dynamical data $\hat x^{(j)}=(x^{(j)},\dot x^{(j)},\ddot x^{(j)})$, $j=1,\ldots,M$ of the coupled harmonic oscillator $L_\rf \colon T\R^2 \to \R$ with
\begin{equation}\label{eq:LOscillator}
L_\rf (x,\dot x) = \frac 12 \|\dot x\|^2 - \frac 12 \|x\|^2 + \alpha x^0x^1, \quad x=(x^0,x^1) \in \R^2, (x,\dot x) \in T\R^2
\end{equation}
with coupling constant $\alpha =0.1$. Here $\overline x^{(j)}=(x^{(j)},\dot x^{(j)})$, $j=1,\ldots,M$ are the first $M$ elements of a Halton sequence in the hypercube $\Omega = \mod{(}-1,1\mod{)}^4 \subset T\R^2$.
We use \mod{the square exponential kernel} $K(\overline{x},\overline{y}) = \exp\left(-\mod{\|}\overline{x}-\overline{y}\mod{\|}^2\mod{/2}\right)$ as a kernel function in all experiments. For $M \in \N$ we obtain a posteriori Gaussian processes denoted by $\xi_M \in \mathcal{N}(L_{\mod{(M)}},\mathcal{K}_M)$ modelling Lagrangians for the dynamical system. We present experiments with $M \in \{80,300\}$. In the following $\gls{var}$ refers to the variance of a random variable (applied component wise when the random variable is $\R^d$-valued). Moreover, $\mod{\gls{g}}_{\overline{x}}(L_{\mod{(M)}})$ refers to the solution of $\mathrm{EL}(L_{\mod{(M)}})(\overline{x},\ddot x)=0$ for $\ddot x \in \R^2$.

\Cref{fig:DataOscillator} displays the location of training data in $\Omega$ projected to the $(x^0,x^1)$-plane.
\Cref{fig:VariancesOscillator} compares the variances of $\mathrm{EL}_{\hat x}(\xi_{M})$ for $M=80,300$ for points of the form $\hat x = (\overline{x},\ddot x)$ with $\overline{x}=(x^0,x^1,0,0) \in \Omega$ and $\overline x = (x^0,0,\dot x^0,0)\in \Omega$ with $\ddot x = \mod{g}_{\overline{x}}(L_{\mod{(M)}})$.
One observes that the variance decreases as more data points are used.
This experiments suggests that the method can be used in combination with an adaptive sampling technique to sample new data points in regions of high model uncertainty.

\begin{figure}
	\centering
	\includegraphics[width=0.4\linewidth]{"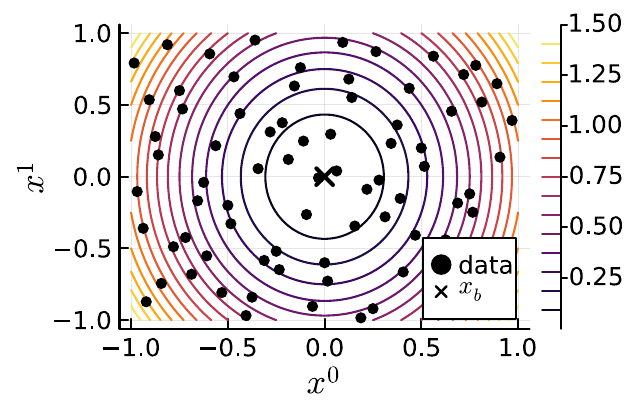"}
	\includegraphics[width=0.4\linewidth]{"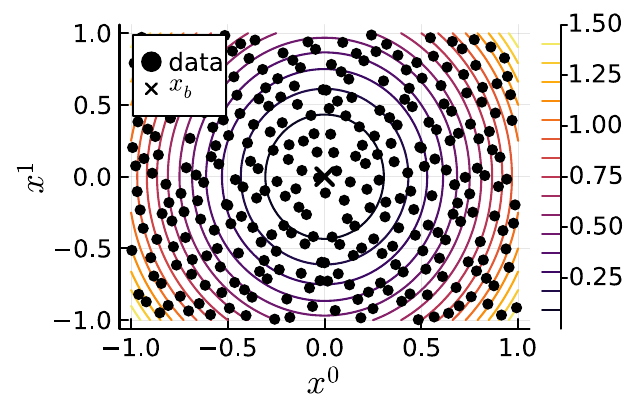"}
	\caption{Training data points projected to the $(x^0,x^1)$-plane of $\xi_{80}$ (left) and $\xi_{300}$ (right).}\label{fig:DataOscillator}
\end{figure}

\begin{figure}
	\centering
	\includegraphics[width=0.24\linewidth]{"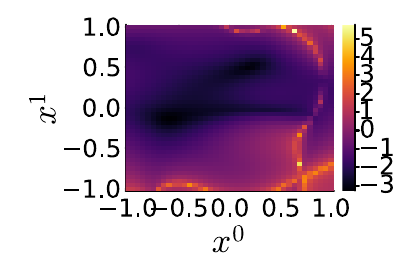"}
	\includegraphics[width=0.24\linewidth]{"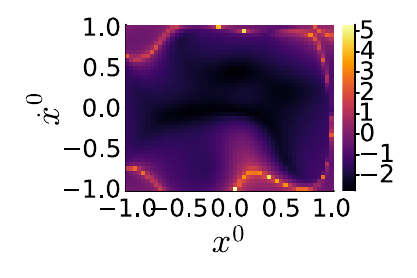"}
	\includegraphics[width=0.24\linewidth]{"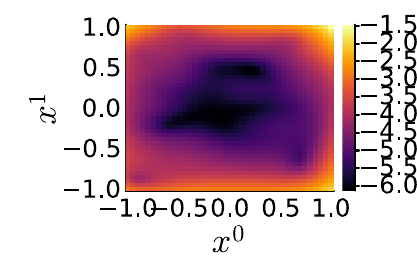"}
	\includegraphics[width=0.24\linewidth]{"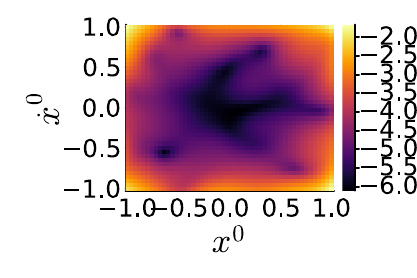"}
	\caption{Plots of variances $\log_{10}(\|\mathrm{var}(\mathrm{EL}(\xi_{M}))\|)$ for $M=80$ (two left plots) and $M=300$ (two right plots) over $(x^0,x^1,0,0)$-plane and $(x^0,0,\dot x^0,0)$-plane. (Ranges of colourbars vary.)}\label{fig:VariancesOscillator}
\end{figure}

\Cref{fig:MotionsOscillator} shows a motion computed by solving\footnote{Computations were performed using DifferentialEquations.jl\cite{rackauckas2017differentialequations}. Comparison with a trajectory computed using the variational midpoint rule \cite{MarsdenWestVariationalIntegrators} (step-size $\mod{\Delta t}=0.01$) shows a maximal difference in the $x$-component smaller than $3.5 \times 10^{-4}$ ($M=300$) along the trajectory.} $\mathrm{EL}(L_{\mod{(M)}})=0$ with initial data $\overline{x} = (0.2, 0.1, 0, 0)$ on the time interval $[0,100]$.
In the plots of the first row, colours indicate the norm of the variance of $\mathrm{EL}(\xi_M)$ along the computed trajectories.
For $M=300$ the trajectory is close to the reference solution while largely different for $M=80$. This is consistent with the lower variance for $M=300$ compared to the experiment with $M=80$.
The plots of the dynamics of $L_{\mod{(}300\mod{)}}$ (bottom row of \cref{fig:MotionsOscillator}) show divergence of the computed motion from the reference solution towards the end of the time interval building up to a difference in $x^0$ component of about $0.1$ at $t=100$. (We will see later that a discrete model model performs better in this experiment.) However, the qualitative features of the motion are captured.

\begin{figure}
	\centering
	\includegraphics[width=0.45\linewidth]{"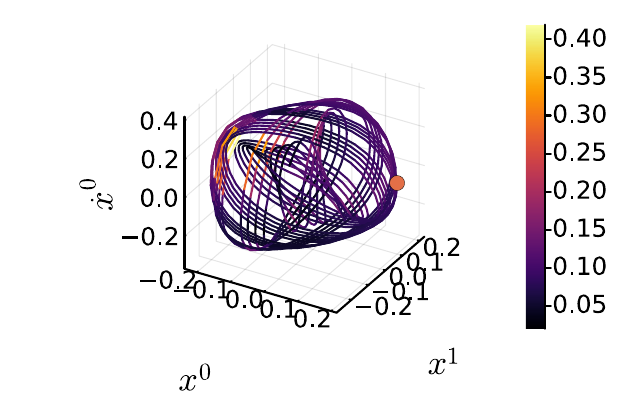"}
	\includegraphics[width=0.45\linewidth]{"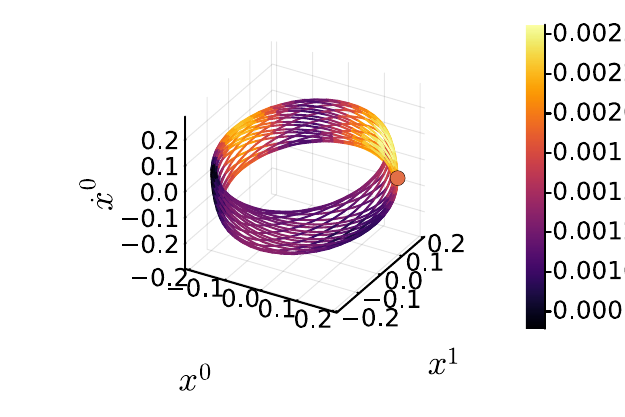"}
	\includegraphics[width=0.45\linewidth]{"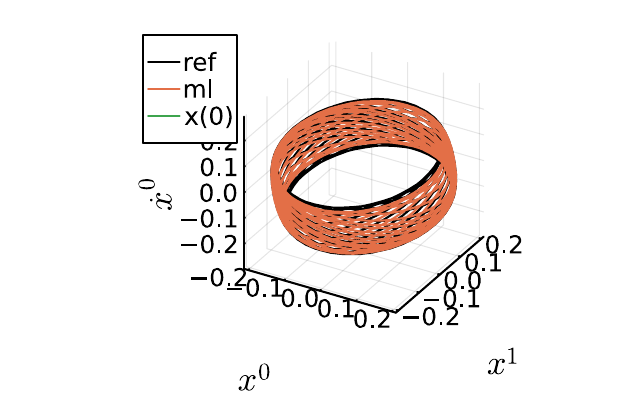"}
	\includegraphics[width=0.45\linewidth]{"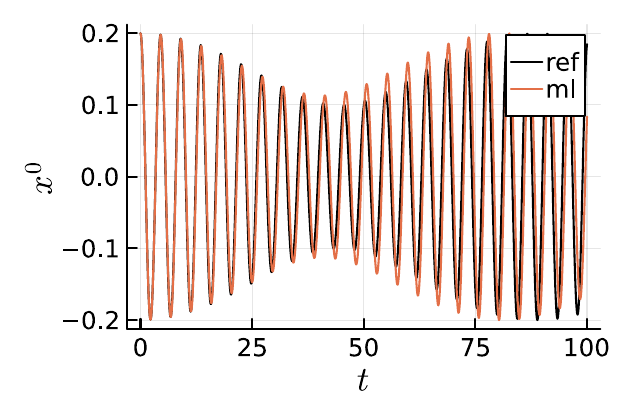"}
	\caption{ Top row: motion of $\xi_{80}$ (left) and $\xi_{300}$ (right) with variance $\|\mathrm{var}(\mathrm{EL}(\xi_{M}))\|$ encoded as colours (ranges of colourbars vary). Bottom row: 
		motions of $\xi_{300}$ compared to reference. 
	}\label{fig:MotionsOscillator}
\end{figure}

\Cref{fig:HamiltonianOscillator} shows the Hamiltonian $H_M=\mathrm{Ham}(L_{\mod{(M)}})$ as well as $H_M \pm 0.2 \sigma_{H_M}$. Here $\sigma_{H_M}$ denotes the standard deviation $\sqrt{\mathrm{var}\mathrm{Ham}(\xi_M)}$. We observe a clear decrease of the standard deviation as $M$ increases from 80 to 300.

\begin{figure}
	\includegraphics[width=0.45\linewidth]{"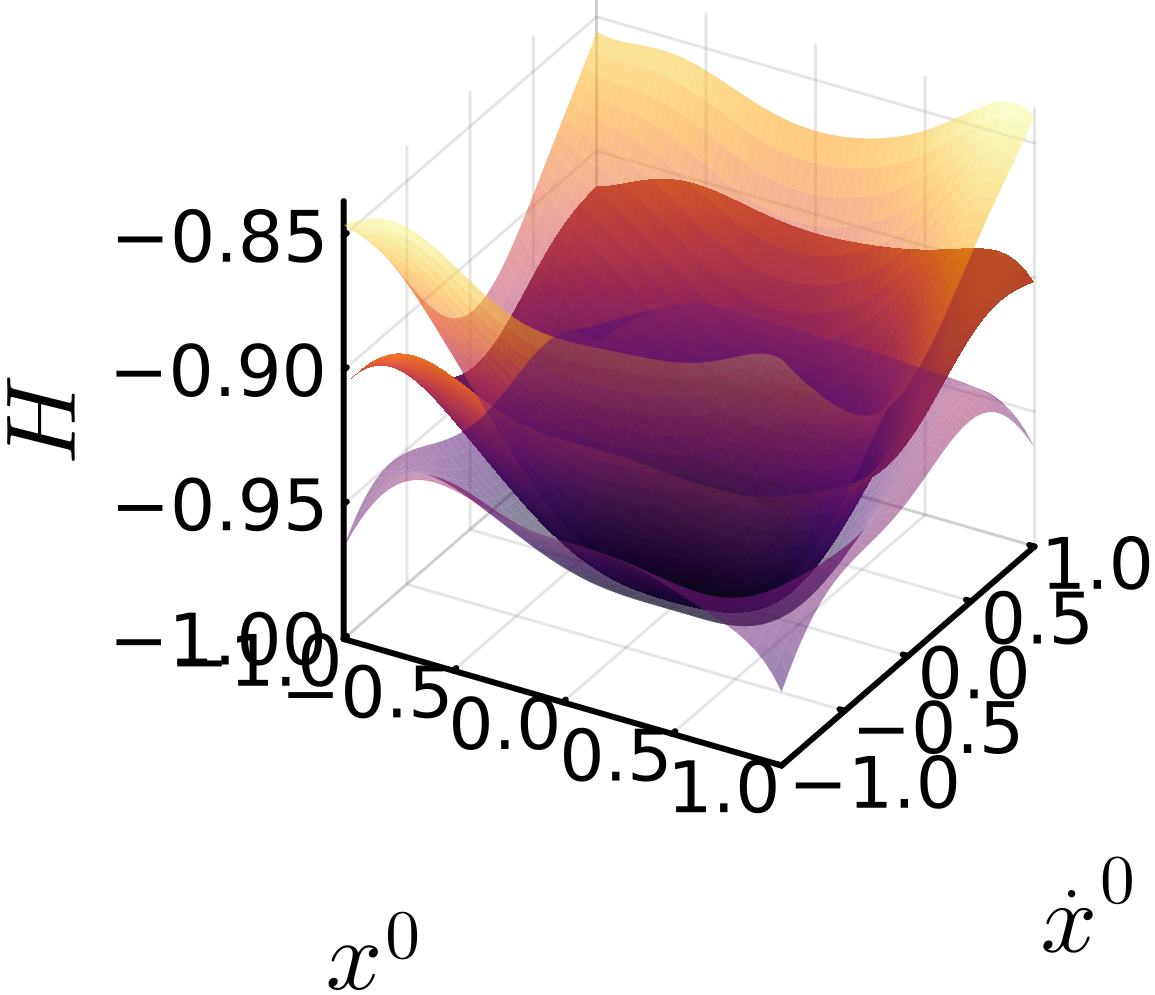"}
	\includegraphics[width=0.45\linewidth]{"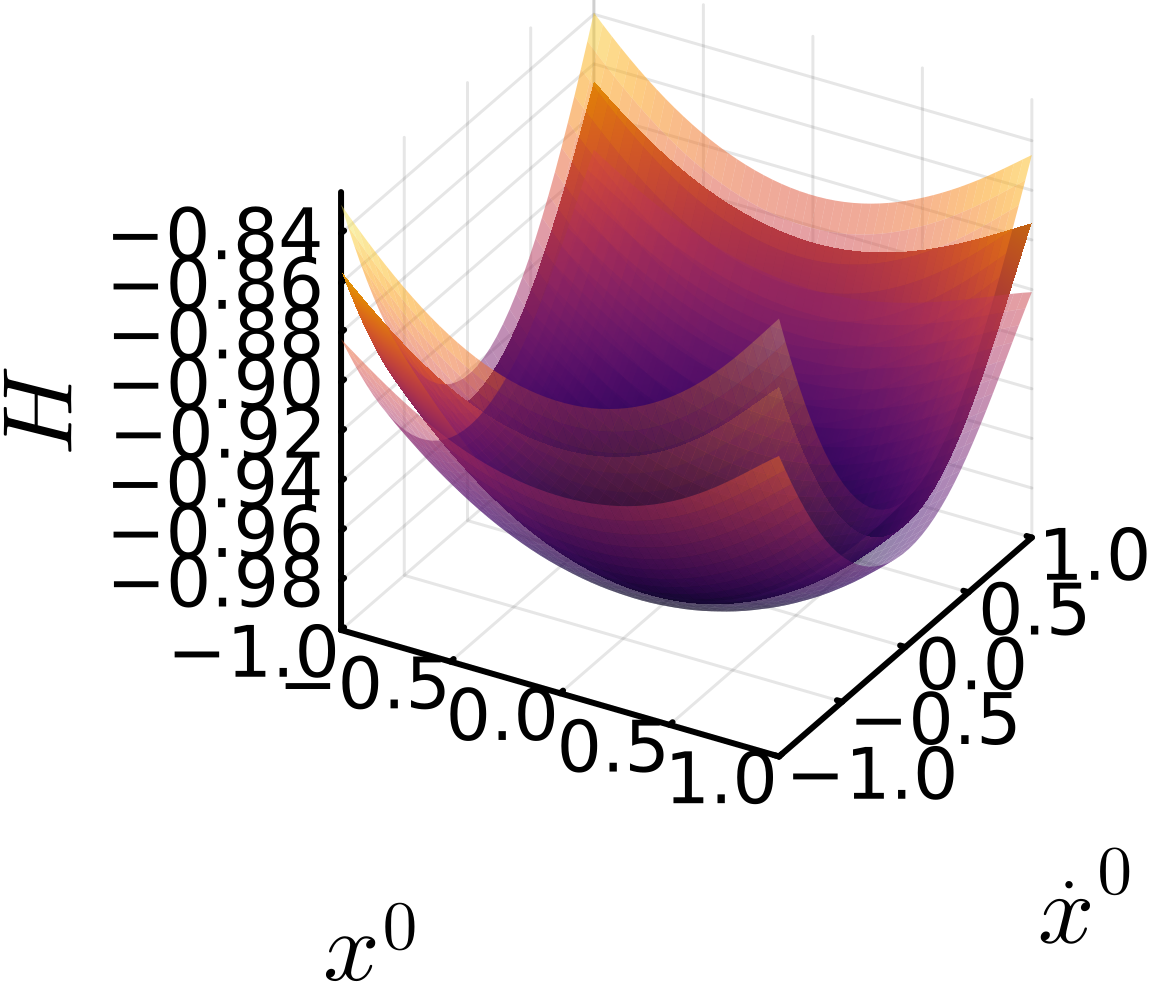"}
	\caption{Mean of Hamiltonian $\mathrm{Ham}(\xi_{80})$, $\mathrm{Ham}(\xi_{300})$ over $(x^0,0,\dot x^0,0)$ plus/minus 20\% standard deviation.}\label{fig:HamiltonianOscillator}
\end{figure}

\Cref{fig:AccOscillator} displays the error in the prediction of $\ddot{x}$ for points $\overline{x}=(x^0,x^1,0,0) \in \Omega$ and $\overline x = (x^0,0,\dot x^0,0)\in \Omega$.
As the magnitudes of errors vary widely, $\log_{10}$ is applied before plotting, i.e.\ we show the quantity
\[
\log_{10}\|\mod{g}_{\overline{x}}(L_{\mod{(M)}})-\mod{g}_{\overline{x}}(L_\rf)\|_{\R^2}.
\]
One sees a clear decrease in error as $M$ is increased from 80 to 300.

\begin{figure}
	\centering
	\includegraphics[width=0.4\linewidth]{"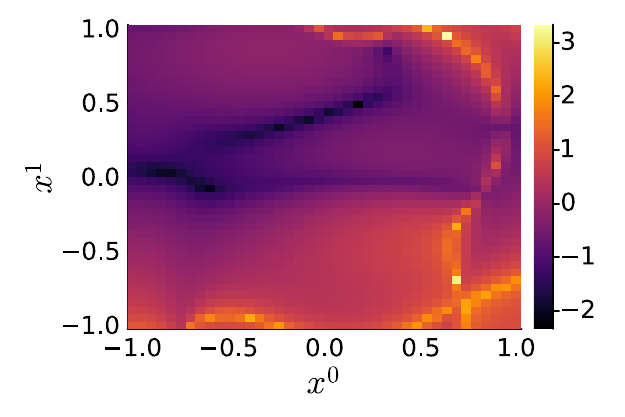"}
	\includegraphics[width=0.4\linewidth]{"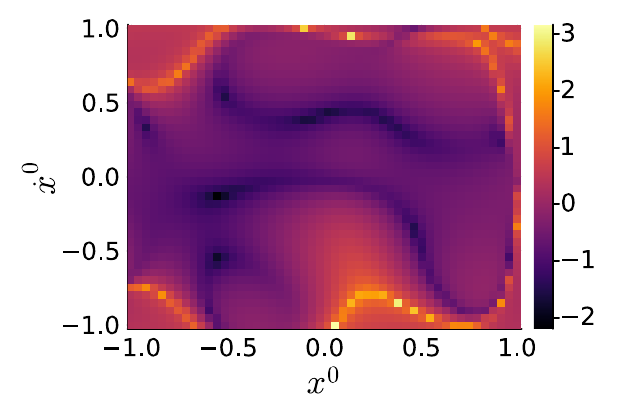"}
	\includegraphics[width=0.4\linewidth]{"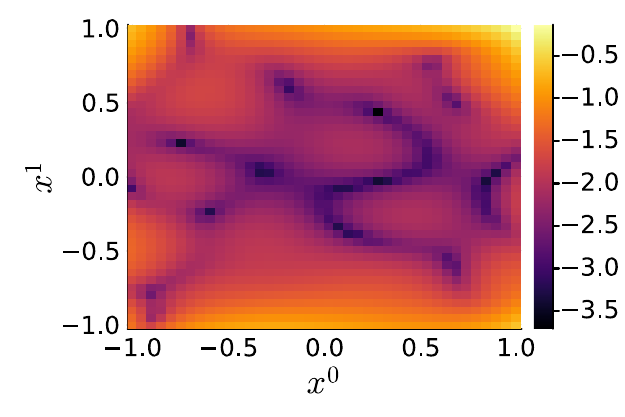"}
	\includegraphics[width=0.4\linewidth]{"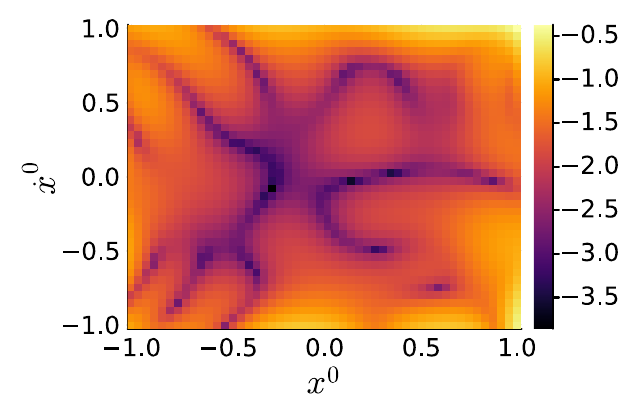"}
	\caption{$\log_{10}$ norm of error of predicted acceleration $\ddot x$ for $\mod{g}(\xi_{M})$ over $x^0,x^1$ plane and $x^0,\dot x^0$ plane for $M=80$ (\mod{top}) and $M=300$ (\mod{bottom}). (The ranges of colourbars vary.)}\label{fig:AccOscillator}
\end{figure}

\mod{Another numerical experiments confirming the convergence of the method and indicating unbounded convergence rates for smooth dynamical systems will be presented in \cref{fig:Convergence} in the context of theoretical convergence results in \cref{sec:ConvRatesContLagrangian}.}

\subsection{Discrete Lagrangian}

Now we consider dynamical data $\hat{x}^{(j)} = (x^{(j)}_0, x^{(j)}_1, x^{(j)}_2)$ where $x^{(j)}_0$, $x^{(j)}_1$, $x^{(j)}_2$ are snapshots of true trajectories at times $t$, $t+\mod{\Delta t}$, $t+2\mod{\Delta t}$, respectively, with $j=1,\ldots,M$. Here $\mod{\Delta t}=0.1$ and, again, $M\in \{80,300\}$.
For data generation, we consider data $(x,p) \in [-1,1]^4 \subset T^\ast \R^2$ from a Halton sequence from where we integrate $L_\rf$ from $[0,3\mod{\Delta t}]$ using the 2nd order accurate variational midpoint rule \cite{MarsdenWestVariationalIntegrators} with step-size $\mod{\Delta t}_{\mathrm{internal}} = \mod{\Delta t}/10$. These dynamics are considered as true for the purpose of this experiment. Training data is visualised in \cref{fig:DataOscillatorDiscrete}.
\begin{figure}
	\centering
	\includegraphics[width=0.49\linewidth]{"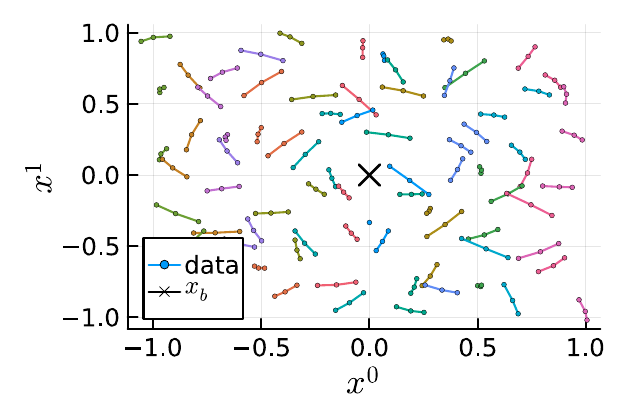"}
	\includegraphics[width=0.49\linewidth]{"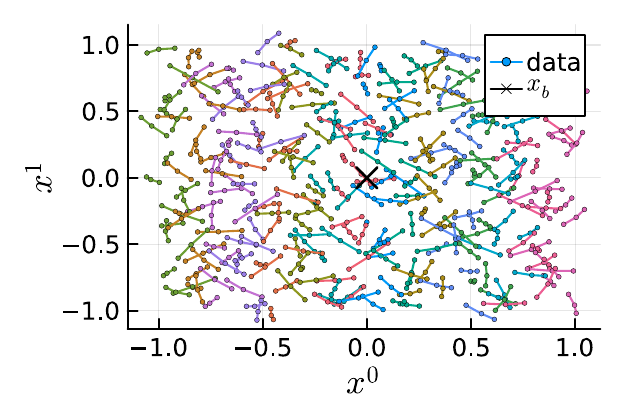"}
	\caption{Training data. Each line connects snapshots points that constitute a training data point $\hat x$. Left: $M=80$, right: $M=300$.}\label{fig:DataOscillatorDiscrete}
\end{figure}

\Cref{fig:VariancesOscillatorDiscrete} (in analogy to \cref{fig:VariancesOscillator}) shows how variance decreases as more data points become available. For the plots, $(x_0,p_0) \in T^\ast\R^2$ are used to compute $\hat x=(x_0,x_1,x_2)$ using $L_\rf$. Here $p$ refers to the conjugate momentum of $L_\rf$. The plots display heatmaps of $\log_{10}(\|\mathrm{var}(\mathrm{DEL}_{\hat x}(\xi_{M}))\|)$.
\begin{figure}
	\centering
	\includegraphics[width=0.4\linewidth]{"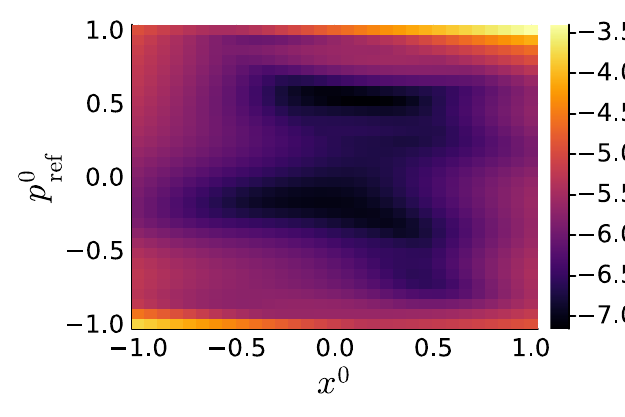"}
	\includegraphics[width=0.4\linewidth]{"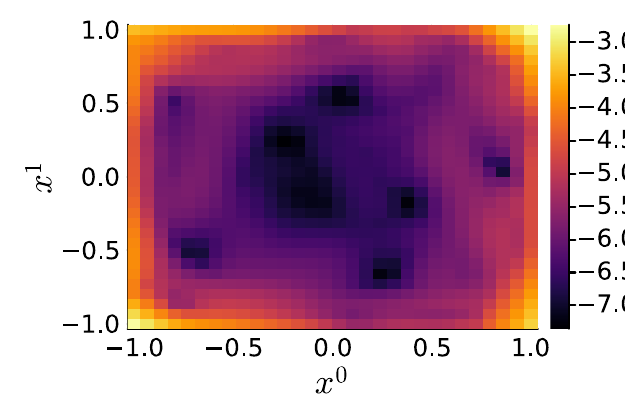"}
	\includegraphics[width=0.4\linewidth]{"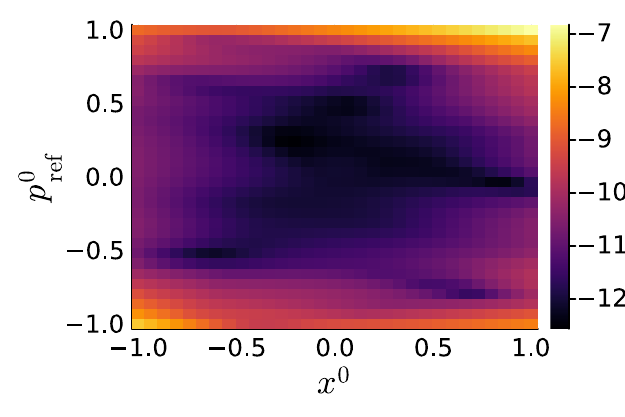"}
	\includegraphics[width=0.4\linewidth]{"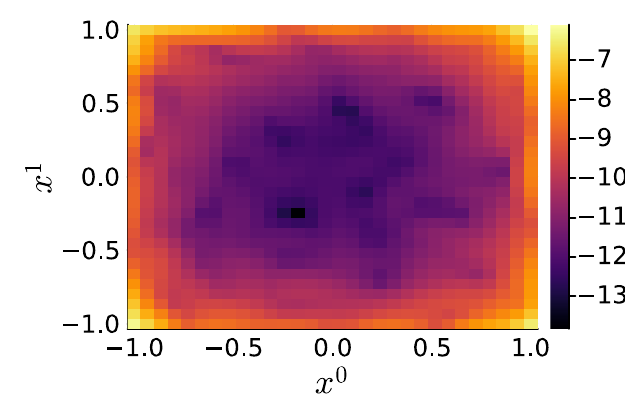"}
	\caption{Plots of variances $\log_{10}(\|\mathrm{var}(\mathrm{EL}(\xi_{M}))\|)$ for $M=80$ (\mod{top}) and $M=300$ (\mod{bottom}) over $(x,p_\rf)=(x^0,x^1,0,0)$-plane and $(x,p_\rf)=(x^0,0,p_\rf^0,0)$-plane. (Ranges of colourbars vary.)}\label{fig:VariancesOscillatorDiscrete}
\end{figure}

\Cref{fig:MotionsOscillatorDiscrete} shows a motion for $t \in [0,100]$ of $\xi_{300}$ with same initial data as in \cref{fig:MotionsOscillator}.
With a maximal error in absolute norm smaller than $0.00043$ it is visually indistinguishable from the true motion.
In the plot to the left, data for $\dot x^0$ was approximated to second order accuracy in $\mod{\Delta t}$ with the central finite differences method.

Comparing \cref{fig:MotionsOscillatorDiscrete} and \cref{fig:MotionsOscillator}, it is interesting to observe that with the same amount of data the discrete model performs better than the continuous model for predicting motions.
\begin{figure}
	\centering
	\includegraphics[width=0.49\linewidth]{"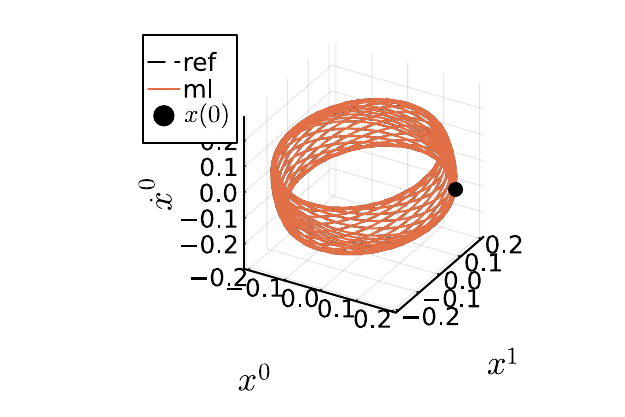"}
	\includegraphics[width=0.49\linewidth]{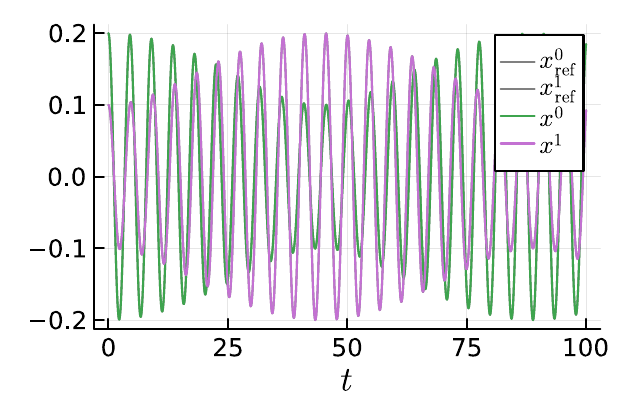}
	\caption{The motion of $\xi_{300}$ and the true motion are indistinguishable.}\label{fig:MotionsOscillatorDiscrete}
\end{figure}

\paragraph{Reproducibility}
Source code of the experiments can be found at
\url{https://github.com/Christian-Offen/Lagrangian_GP}.

%% file: ConvergenceAnalysis_ODE.tex
\section{Convergence Analysis}\label{sec:ConvergenceAnalysis}

This section contains {a theoretical convergence analysis of the considered methods. In \cref{sec:ConvergenceAnalysisL,sec:ConvergenceAnalysisLd}}
convergence theorems for regular continuous Lagrangians (\cref{thm:ConvergenceThm}) and discrete Lagrangians (\cref{thm:ConvergenceThmLdTemporalNonDegenerate}) in the infinite-data limit {are provided} as observations become topologically dense, i.e.\ as the maximal distance between data points converges to zero.
{Moreover, the convergence rates of continuous and discrete Lagrangian models are analysed in \cref{sec:ConvergenceRates}.}
 

\subsection{{Convergence of} continuous Lagrangian {models}}\label{sec:ConvergenceAnalysisL}

\subsubsection{Convergence theorem (continuous, temporal evolution)}

\begin{theorem}\label{thm:ConvergenceThm}
	Let $\Omega \subset T\R^d \cong \R^d \times \R^d$ be an open, bounded non-empty domain.
	Consider a sequence of observations $\{(x^{(j)},\dot{x}^{(j)},\ddot{x}^{(j)})\}_j \subset \Omega \times \R^d$ of a dynamical system governed by the Euler--Lagrange equation of an (unknown) non-degenerate Lagrangian $L_\rf \in \mathcal{C}^2(\gls{OmegaBar})$ (definition of $\mathcal{C}^2(\overline{\Omega})$ below).
	Assume that $\{(x^{(j)},\dot{x}^{(j)})\}_j \subset \Omega $ is topologically dense.
	Let $K$ be a 4-times continuously differentiable kernel on $\Omega$, $\overline{x_b} \in \Omega$, $\mod{c}_b \in \R$, $p_b \in \R$ and assume that $L_\rf$ is contained in the reproducing kernel Hilbert space $(U, \| \cdot \|_U)$ to $K$ and fulfils the normalisation condition
	\begin{equation}\label{eq:NormalisationThm}
		\gls{PhiN}(L_\rf) = (p_b,\mod{c}_b) \quad \text{with} \quad
		\Phi_N(L)= \left( \frac{\p L}{\p \dot x}(\overline{x}_b),L(\overline{x}_b) \right).
	\end{equation}
	Assume that $U$ embeds continuously into $\mathcal{C}^2(\overline{\Omega})$.
	Let $\xi \in \mathcal{N}(0,\mathcal K)$ be a canonical Gaussian \mod{field} on $U$ (see \cref{sec:RKHSSetup} \mod{or \cref{app:GaussianField}}). Then the sequence of conditional means $\gls{Lj}$ of $\xi$ conditioned on the first $j$ observations and on the normalisation conditions
	\begin{equation}
		\mathrm{EL}(\xi)(x^{(i)},\dot{x}^{(i)},\ddot{x}^{(i)}) =0\, (\forall i\le j),\quad
		\Phi_N(\xi) = (p_b,\mod{c}_b)
	\end{equation}
	converges in $\| \cdot \|_U$ and in $\| \cdot \|_{\mathcal{C}^2(\overline{\Omega})}$ to a Lagrangian $\gls{LInf} \in U$ that is
	\begin{itemize}
	 \item consistent with the normalisation $\Phi_N(L_{(\infty)}) = (p_b,\mod{c}_b)$
	 \item consistent with the dynamics, i.e.\ $\mathrm{EL}(L_{(\infty)})(\hat x) \mod{=0}$ for all $\hat x = (x,\dot{x},\ddot{x})$ with $(x,\dot{x}) \in \Omega$ and $\mathrm{EL}(L_\rf)(\hat x)=0$.
	 \item Moreover, $L_{(\infty)}$ is the unique minimiser of $\| \cdot \|_U$ among all Lagrangians \mod{in $U$} with these properties.
	\end{itemize}
\end{theorem}

\begin{remark}\label{rem:RP}
If $\mod{c}_b=0$ and $p_b=0$, then the sequence $\gls{Lj}$ is constantly zero with limit $\gls{LInf} \equiv0$. It is necessary to set $(\mod{c}_b,p_b) \not = (0,0)$ to approximate a non-degenerate Lagrangian.
\end{remark}

\begin{remark}\label{rem:RegularityKernel}
The regularity assumptions of the kernel (four times continuously differentiable) is required for the interpretation of $L_{(j)}$ as a conditional mean of a Gaussian process and for its convenient computation. 
{It can be relaxed to the condition that $\frac{\p^{|\alpha|+|\beta|} K(x,y)}{(\p x^1)^{\alpha_1} \ldots (\p x^d)^{\alpha_d} \p (y^1)^{\beta_1} \ldots (\p y^d)^{\beta_d} }$ for $|\alpha|,|\beta| \le 2$, $x,y \in \Omega$ exists and is continuous on $\overline{\Omega}$. Here $|\alpha| = \alpha_1+\ldots+\alpha_d$, $|\beta| = \beta_1+\ldots+\beta_d$. 
}
\end{remark}

\subsubsection{Formal setting and proof (continuous, temporal evolution)}\label{sec:FormalSettingCont}

Let $\Omega \subset T\R^d$ be an open, bounded, non-empty domain.
We consider the space of $m$-times continuously differentiable functions that extend to the topological closure $\gls{OmegaBar}$ 
\begin{equation}\label{eq:CmOmegabar}
\gls{Cm}(\overline{\Omega},\R^k) = \{ f \in \mathcal{C}^m(\Omega,\R^k) \; | \; \p^\alpha f \textrm{ extends continuously to }\overline\Omega \; \forall |\alpha| \le m\}, \quad m \in \N_0.
\end{equation}
Here
$\p^\alpha f = \frac{\p^{|\alpha|} f}{(\p x^1)^{\alpha_1} \ldots (\p x^d)^{\alpha_d} \p (\dot x^1)^{\dot \alpha_1} \ldots (\p \dot x^d)^{\dot \alpha_d} }$
denotes the partial derivative with respect to coordinates $\overline x = (x,\dot x) = (x^1,\ldots,x^d,\dot x^1,\ldots,\dot x^d)$ for a multi-index $\alpha = (\alpha_1,\ldots,\alpha_d,\dot \alpha_1,\ldots,\dot \alpha_d)$ with $|\alpha| = \alpha_1 + \ldots +\alpha_d+\dot \alpha_1+\ldots+\dot \alpha_d$. 
The space is equipped with the norm
\begin{equation}
\| f\|_{\mathcal{C}^m(\overline{\Omega},\R^k)} = \max_{0 \le |\alpha | \le m} \sup_{\overline{x} \in \Omega} \|\p^\alpha f(\overline{x}) \|.´
\end{equation}
Here $\|\p^\alpha f(\overline{x}) \|$ denotes the Euclidean norm on $\R^k$ for $|\alpha|=1$ or an induced operator norm for $|\alpha|>1$.
The space $\mathcal{C}^m(\overline{\Omega},\R^k)$ is a Banach space \cite[§ 4]{SobolevSpacesAdams2003}.
We will use the shorthand $\mathcal{C}^m(\overline{\Omega})=\mathcal{C}^m(\overline{\Omega},\R^1)$.

Assume that on a dense, countable subset $\gls{Omega0} = \{ \overline{x}^{(j)} = ({x}^{(j)},\dot{x}^{(j)})\}_{j=1}^\infty \subset \Omega$ we have observations of acceleration data $\ddot{x}^{(j)}$ of a dynamical system generated by an (a priori unknown) Lagrangian $L_\rf \in \mathcal{C}^2(\overline \Omega)$, which is non-degenerate, i.e.\ for all $(x,\dot x) \in \overline{\Omega}$ the matrix $\frac{\p^2 L_\rf}{\p \dot x \p \dot x} (x,\dot x)$ is invertible, and the induced function $g_\rf \in \mathcal{C}^0(\overline{\Omega},\R^d)$ with
\begin{equation}\label{eq:AccLref}
	\gls{gref}(x,\dot{x}) = \left( \frac{\p^2 L_\rf}{\p \dot x \p \dot x} (x,\dot x)\right)^{-1}	
	\left(
	\frac{\p L_\rf}{ \p  x} (x,\dot{x}) - \frac{\p^2 L_\rf}{\p x \p \dot x} (x,\dot{x}) \cdot \dot x
	\right)
\end{equation}
recovers $\ddot{x}^{(j)} = g_\rf( \overline{x}^{(j)}) = g_\rf({x}^{(j)},\dot{x}^{(j)})$.

\begin{lemma}\label{lem:PhiInfBounded}
The linear functional $\Phi^{(\infty)} \colon \mathcal{C}^2(\overline{\Omega}) \to \mathcal{C}^0(\overline\Omega,\R^d)$
with
\begin{equation}
\begin{split}
\Phi^{(\infty)} (L) (x,\dot{x})
&= \mathrm{EL}(L)(x,\dot{x},g_\rf(x,\dot{x}))\\
&=\frac{\p^2 L}{\p \dot x \p \dot x} (x,\dot{x}) \cdot g_\rf(x,\dot x)
+ \frac{\p^2 L}{\p x \p \dot x} (x,\dot{x}) \cdot \dot x
-\frac{\p L}{ \p  x} (x,\dot{x}) 
\end{split}
\end{equation}
is bounded. 
\end{lemma}

\begin{proof}
A direct application of the triangle inequality shows
\[\|\Phi^{(\infty)} (L)\|_{\mathcal{C}^0(\overline\Omega,\R^d)} \le \left(\| g_\rf\|_{\mathcal{C}^0(\overline\Omega,\R^d)} + \sup_{(x,\dot x) \in \Omega}\|\dot x\|+1\right)\| L\|_{_{\mathcal{C}^2(\overline\Omega)}}.\]
\end{proof}

Since for each $\overline{x}$ the evaluation functional $\mathrm{ev}_{\overline{x}} \colon f \mapsto f(\overline{x})$ on $\mathcal{C}^0(\overline{\Omega},\R^d)$ is bounded, the following functions constitute bounded linear functionals for $j \in \N$:
\begin{align*}
\Phi_j &\colon \mathcal{C}^2(\overline\Omega) \to \R^d, \quad \quad\Phi_j(L)=\Phi^{(\infty)} (L)(\overline{x}^{(j)})\\
\Phi^{(j)} &\colon \mathcal{C}^2(\overline\Omega) \to (\R^d)^j, \quad \Phi^{(j)} =(\Phi_1,\ldots,\Phi_j).
\end{align*}

For a reference point $\overline{x}_b \in \Omega$ and for $p_b \in \R^d$, $\mod{c}_b \in \R$ we define the bounded linear functional
\begin{align}\label{eq:PhiNContinuous}
\gls{PhiN} &\colon \mathcal{C}^2(\overline\Omega) \to \R^{d+1}, \quad \quad \Phi_N(L)= \left( \frac{\p L}{\p \dot x}(\overline{x}_b),L(\overline{x}_b) \right),
\end{align}
related to our normalisation condition, the shorthands $\Phi_b^{(k)} = (\Phi_1,\ldots,\Phi_k,\Phi_N)$ and $\Phi_b^{(\infty)} = (\Phi^{(\infty)} ,\Phi_N)$, and the data
\begin{align*}
y^{(k)}_{\mod{b}} &= (0,\ldots,0,p_b,\mod{c}_b) \in (\R^d)^k \times \R^d\times \R\\
y^{(\infty)}_{\mod{b}} &= (0,p_b,\mod{c}_b) \in \mathcal{C}^0(\overline\Omega,\R^d) \times \R^d\times \R.
\end{align*}

\begin{assumption}\label{ass:LinB}
Assume that there is a Hilbert space $U$ with continuous embedding $U \hookrightarrow \mathcal{C}^2(\overline\Omega)$ such that
\[\{L \in \mathcal{C}^2(\overline\Omega) \, | \, \Phi_b^{(\infty)}(L) = y^{(\infty)}_{\mod{b}} \} \cap U \not = \emptyset{.}\]
In other words, $U$ is assumed to contain a Lagrangian consistent  with the normalisation and underlying dynamics.
\end{assumption}

The affine linear subspace{s}
\begin{align*}
A^{(j)} &= \{ L \in U \, | \, \Phi_b^{(j)}(L) = y^{(j)}_{\mod{b}} \} \quad (j \in \N)\\
A^{(\infty)} &= \{ L \in U \, | \, \Phi_b^{(\infty)}(L) = y^{(\infty)}_{\mod{b}} \}
\end{align*}
are closed, \mod{convex,} and non empty in $U$ by \cref{ass:LinB} and by the boundedness of $\Phi_b^{(j)}$ and $\Phi_b^{(\infty)}$ on $U \hookrightarrow \mathcal{C}^2(\overline{\Omega})$.
\mod{By the Hilbert projection theorem \cite[§12.3]{Rudin1991}, the following minimisers exist and are uniquely defined:}
\begin{equation}\label{eq:MinProblems}
	\begin{split}
L_{(j)} &\mod{:=} \arg \min_{L \in A^{(j)}} \| L \|_U \\
L_{(\infty)} &\mod{:=} \arg \min_{L \in A^{(\infty)}} \| L \|_U{.}
\end{split}
\end{equation}
Here $\| \cdot \|_U $ denotes the norm in $U$.

\begin{remark}\label{rem:AssumptionTrueForK}
{%
	To an open, non-empty set $\mod{\Omega} \subset \R^{\mod{d'}}$, $m\in \N \cup \{0\}$ denote by $W^{m,2}(\mod{\Omega})=W^m(\mod{\Omega})$ the Sobolev space
	\[
	W^m(\mod{\Omega})=\{ u \in L^2(\mod{\Omega}) \, | \, \forall \alpha \in \N^{\mod{d'}}, |\alpha| \le m, \p^\alpha u \in L^2(\mod{\Omega})\},
	\]
	with Sobolev norm
	\[
	 \| u \|_{W^m} = \sqrt{\sum_{|\alpha|\le m} \int_\Omega (\p^\alpha u (x) )^2 \d x }
	\]
	where $L^2(\mod{\Omega})$ denotes the space of square integrable functions on $\mod{\Omega}$. Here the derivative $\p^\alpha u$ is meant in a distributional sense \cite{SobolevSpacesAdams2003}.} %
In the machine learning setting, $U$ is the reproducing kernel Hilbert space related to a kernel $K \colon \Omega \times \Omega \to \R$. Assume the domain of $\Omega\mod{\subset \R^{d'}=\R^{2d}}$ is locally Lipschitz. When $K$ is the squared exponential kernel, for instance, its reproducing kernel Hilbert space embeds into any Sobolev space $W^m(\Omega)$ ($m>1$) \cite[Thm.4.48]{ChristmannSteinwart2008RKHS}. In particular , it embeds into $W^m(\Omega)$ with $m>2+\mod{d}$, which is embedded into $\mathcal{C}^2(\overline{\Omega})$ by the Sobolev embedding theorem \cite[§4]{SobolevSpacesAdams2003}.
\mod{By \cref{thm:Extremizer} of \cref{app:CondExpandVar} (also see  \cref{rem:GPisMinProblem})} the element $L_{(j)}$ \mod{($j \in \mathbb{N}$)} from \eqref{eq:MinProblems} coincides with the conditional mean of the Gaussian process $\xi$ conditioned on $\Phi_b^{(j)}(\xi) = y^{(j)}$.
\end{remark}

\begin{proposition}\label{prop:ConvergenceInB}
The minima $L_{(j)}$ converge to $L_{(\infty)}$ in the norm $\| \cdot \|_U$ and, thus, in $\| \cdot \|_{\mathcal C^2(\overline{\Omega})}$.
\end{proposition}

\begin{proof}
The sequence of affine spaces $A^{(1)} \supseteq A^{(2)} \supseteq A^{(3)} \supseteq \ldots$ is monotonously decreasing and $A^{(\infty)} \subseteq \bigcap_{j=1}^\infty A^{(j)}$.
Therefore, the sequence $L_{(j)}$ is monotonously increasing and its norm $\|L_{(j)}\|_U$ is bounded from above by $\|L_{(\infty)}\|_U$.
Since $U$ is reflexive, there exists a subsequence $(L_{(j_i)})_{i \in \N}$ that weakly converges to some $L^\dagger_{(\infty)} \in U$.
(This follows from the Banach-Alaoglu theorem and the Eberlein-\v{S}mulian theorem \cite{Diestel1984}.)
By the weak lower semi-continuity of the norm, we obtain
\begin{equation}\label{eq:WeakLimitUpperBound}
\| L^\dagger_{(\infty)} \|_U \le \liminf_{i \to \infty} \| L_{(j_i)} \|_U \le \| L_{(\infty)} \|_U.
\end{equation}

\begin{lemma}\label{lem:LdaggerinA}
The weak limit $L^\dagger_{(\infty)}$ of $(L_{(j_i)})_{i \in \N}$ is contained in $A^{(\infty)}$.
\end{lemma}

Before providing the proof of \cref{lem:LdaggerinA}, we show how this allows us to complete the proof of \cref{prop:ConvergenceInB}.

As $L^\dagger_{(\infty)} \in A^{(\infty)}$, we have $\| L_{(\infty)} \|_U \le \| L^\dagger_{(\infty)} \|_U$ since $L_{(\infty)}$ is the global minimiser of the minimisation problem of \eqref{eq:MinProblems}.
Together with \eqref{eq:WeakLimitUpperBound} we conclude $\| L^\dagger_{(\infty)} \|_U = \| L_{(\infty)} \|_U$ and, by the uniqueness of the minimiser $L_{(\infty)}$, the equality $ L^\dagger_{(\infty)}  =  L_{(\infty)}$. Thus, we have proved weak convergence $L_{(j_i)} \rightharpoonup  L_{(\infty)}$.

Together with the lower semi-continuity of the norm, and since $L_{(j_i)}$ is monotonously increasing and bounded by $\| L_{(\infty)} \|_U$, we have
\[
\| L_{(\infty)} \|_U \le \liminf_{i \to \infty} \| L_{(j_i)} \|_U
\le
\limsup_{i \to \infty} \| L_{(j_i)} \|_U
\le 
\| L_{(\infty)} \|_U
\]
such that $\lim_{i \to \infty} \| L_{(j_i)} \|_U = \| L_{(\infty)} \|_U$. Together with $L_{(j_i)} \rightharpoonup  L_{(\infty)}$ we conclude strong convergence $L_{(j_i)} \to L_{(\infty)}$ in the Hilbert space $U$.

The particular weakly convergent subsequence $(L_{(j_i)})_{i \in \N}$ of $(L_{(j)})_j$ was arbitrary. 
Thus, any weakly convergent subsequence of $(L_{(j)})_j$ converges strongly against $L_{(\infty)}$.
It follows that any subsequence of $(L_{(j)})_j$ has a subsequence that converges to $L_{(\infty)}$. This implies that the whole series $(L_{(j)})_j$ converges to $L_{(\infty)}$.

It remains to prove \cref{lem:LdaggerinA}.

\begin{proof}[Proof of \cref{lem:LdaggerinA}]
Let $\overline{x}\in \Omega$. As the sequence $\Omega_0 = (\overline{x}^{(m)} )_{m=1}^\infty$ is dense in $\Omega$, there exists a subsequence $(\overline{x}^{(m_l)} )_{l=1}^\infty$ converging to $\overline{x}$. 
We have
\begin{align}\label{eq:UseContinuity}
\Phi^{(\infty)}_b(L^\dagger_{(\infty)}) (\overline{x})
& = \lim_{l \to \infty} \Phi^{(\infty)}_b(L^\dagger_{(\infty)}) (\overline{x}^{(m_l)}) \\ \label{eq:UseWeakConvergence}
&=\lim_{l \to \infty} \underbrace{ \lim_{i \to \infty} \Phi^{(\infty)}_b (L_{(j_i)}) (\overline{x}^{(m_l)})}_{\stackrel{(\ast)}=0} =0.
\end{align}
For this, in \eqref{eq:UseContinuity} we use that $\Phi^{(\infty)}_b(L^\dagger_{(\infty)}) \in \mathcal{C}^0(\overline{\Omega})$.
Equality in \eqref{eq:UseWeakConvergence} follows because each projection to a component of $\Phi^{(\infty)}_b(\cdot) (\overline{x}^{(m_l)}) \colon U \to \R^d \times \R^{d+1}$ constitutes a bounded linear functional on $U$ and the sequence $(L_{(j_i)})_{i \in \N}$ converges weakly to $L^\dagger_{(\infty)}$.
Finally, equality $(\ast)$ holds because for each $l$ there exists $N\in \N$ such that $j_N \ge m_l$ and then for all $i \ge N$ we have $\Phi^{(\infty)}_b (L_{(j_i)}) (\overline{x}^{(m_l)})=0$.

From $\Phi^{(\infty)}_b(L^\dagger_{(\infty)})(\overline{x})=0$ for all $\overline{x}\in \Omega$ we conclude $L^\dagger_{(\infty)} \in A^{(\infty)}$.
\end{proof}

This completes the proof of \cref{prop:ConvergenceInB}.
\end{proof}

Now we can prove \cref{thm:ConvergenceThm}:

\begin{proof}[Proof of \cref{thm:ConvergenceThm}]\label{proof:ProofConvergenceThm}

By \cref{thm:Extremizer} of \cref{app:CondExpandVar} (also see  \cref{rem:GPisMinProblem}) the conditional means computed in \eqref{eq:LPosterior} coincide with the unique minimiser{s} \mod{$L_{(j)}$ ($j \in \mathbb{N}$)} of the problems \eqref{eq:MinProblems}.
{Indeed, the assumption of \cref{thm:Extremizer} on $y=y_b^{(M)}$ is verified in \cref{prop:CondDistPropLCheck} of \cref{sec:ApplyPropCondDist}.}
\Cref{thm:ConvergenceThm} is, therefore, a direct consequence of \cref{prop:ConvergenceInB}.
\end{proof}

\subsection{Convergence of discrete Lagrangian models}\label{sec:ConvergenceAnalysisLd}

\mod{In analogy to \cref{sec:ConvergenceAnalysisL}, we can prove convergence of our discrete Lagrangian models to a true discrete Lagrangian model.}
\mod{Recall that the use of discrete Lagrangian models
to model motions governed by a continuous, non-degenerate Lagrangian model does not cause any discretisation error.
This is guaranteed by the existence of exact discrete Lagrangians for continuous non-degenerate Lagrangians \cite[§1.6]{MarsdenWestVariationalIntegrators}
once the discretisation parameter $\Delta t$ is below a threshold, as recalled in \cref{Intro:DiscreteLagrangianModel}.}

\subsubsection{Statement of convergence theorem (discrete, temporal evolution)}

\begin{theorem}\label{thm:ConvergenceThmLdTemporalNonDegenerate}
	Let $\Omega_a, \Omega_b \subset \R^d \times \R^d$ be open, bounded, non-empty domains. 
	Let $\Omega = \Omega_a \cup \Omega_b$.
	Consider a sequence of observations
	\[\hat \Omega_0 = \{\hat x^{(j)} = ({x_0}^{(j)},{x_1}^{(j)},{x_2}^{(j)}) \}_{j=1}^\infty\]
	of a discrete dynamical system with (not explicitly known) globally Lipschitz continuous discrete flow map $\mod{\gls{grefbar}}\colon \Omega_a \to \Omega_b$ related to a discrete Lagrangian $L_d^\rf \in \mathcal{C}^1(\overline{\Omega})$, i.e.\
	\begin{itemize}
		
	\item $\mod{\gls{grefbar}}(x_0^{(j)},x_1^{(j)}) =(x_1^{(j)},x_2^{(j)})$ for all $j \in \N$,
	
	\item $\mathrm{DEL}(L_d^\rf)(x_0,\mod{\overline {g_\rf}}(x_0,x_1))=0$ for all $(x_0,x_1) \in \Omega_a$,
	
	\item $\gls{nabla12}L_d^\rf (x_1,x_2) \in \R^{d \times d}$ is invertible for all $(x_1,x_2) \in \overline \Omega_b$.
	
	\end{itemize}
	 
	Assume that $\{({x_0}^{(j)},{x_1}^{(j)})\}_{j=1}^\infty$ is dense in $\Omega_a$.
	Let $K$ be a twice continuously differentiable kernel on $\Omega$, $\mod{\overline {x}_b} \in \Omega$, $\mod{c}_b \in \R$, $p_b \in \R$ and assume that $L^\rf_d$ is contained in the reproducing kernel Hilbert space $(U,\| \cdot \|_U)$ to $K$ and fulfils the normalisation condition
	\begin{equation}\label{eq:NormalisationThmLdTemporal}
		\gls{PhiN}(L^\rf_d) = (p_b,\mod{c}_b) \quad \text{with} \quad
		\gls{PhiN}(L_d)= \left(-\nabla_2 L_d(\mod{\overline {x}_b}),L_d(\mod{\overline {x}_b}) \right)
	\end{equation}
	and that $U$ embeds continuously into $\mathcal{C}^1(\overline{\Omega})$.
	Let $\xi \in  \mathcal{N}(0,\mathcal{K})$ be a centred Gaussian random variable over $U$.
	Then the sequence of conditional means $L_{d,(j)}$ of $\xi$ conditioned on the first $j$ observations and the normalisation conditions
	\begin{equation}
		\mathrm{DEL}(\xi)(\hat x^{(i)}) =0\, (\forall i\le j),\quad
		\Phi_N(\xi) = (p_b,\mod{c}_b)
	\end{equation}
	converges in $\| \cdot \|_U$ and in $\| \cdot \|_{\mathcal{C}^1(\overline{\Omega})}$ to a Lagrangian $L_{d,(\infty)} \in U$ that is 
	\begin{itemize}
		\item consistent with the normalisation $\Phi_N(L_{d,(\infty)}) = (p_b,\mod{c}_b)$
		\item consistent with the dynamics, i.e.\ $\mathrm{DEL}(L_{d,(\infty)})(\hat x) =0$ for all $\hat x = (x_0,x_1,x_2)$ with $(x_0,x_1)\in \Omega_a$,$(x_1,x_2)\in\Omega_b$ and $\mathrm{DEL}(L_d^\rf)(\hat x) =0$.
		\item Moreover, $L_d$ is the unique minimizer of $\| \cdot \|_U$ among all discrete Lagrangians in $U$ with the properties above.
	\end{itemize}
\end{theorem}

\begin{remark}
The regularity assumption of $K$ (twice continuously differentiable) is needed for the
interpretation of $L_{d,(j)}$ as a conditional mean of a Gaussian process and for a convenient computation of $L_{d,(j)}$. However, the proof will show that a relaxation to continuous differentiability is possible.
\end{remark}

\subsubsection{Formal setting and proof (discrete, temporal evolution)}\label{sec:FormalSettingLd}

Let $\Omega_a, \Omega_b \subset \R^d \times \R^d$ be open, bounded, non-empty domains, let $\Omega = \Omega_a \cup \Omega_b$.
Let $\hat \Omega = \{ (x_0,x_1,x_2) \, | \, (x_0,x_1) \in \Omega_a, (x_1,x_2) \in \Omega_b\}$ and let
\[\gls{Omega0hat} =  \{ (x_0^{(j)},x_1^{(j)},x_2^{(j)}) \}_{j=1}^\infty \subset \hat{\Omega} \quad \text{with } (x_0^{(j)},x_1^{(j)})\in \Omega_a, (x_1^{(j)},x_2^{(j)})\in \Omega_b \; \text{for all } j \in \N.\]
Assume that $\{ (x_0^{(j)},x_1^{(j)}) \}_{j=1}^\infty$ is dense in $\Omega_a$.

\begin{remark}[Interpretation of $\hat \Omega_0$]
The set $\hat \Omega_0$ corresponds to a collection of observation data in the infinite data limit. It can be obtained as a collection of three consecutive snapshots of motions of the dynamical system that we observe and for which we seek to learn a discrete Lagrangian. In a typical scenario where $L_d^\rf \colon \R^d \times \R^d \to \R$ is the exact discrete Lagrangian to some underlying continuous Lagrangian, the motions leave the diagonal of $\R^d \times \R^d$ invariant. It is sensible to consider $\Omega_a$ and $\Omega_b$ that are neighbourhoods of compact sections of the diagonal in $\R^d \times \R^d$.
\end{remark}

We consider the discrete Lagrangian operator
\begin{equation}
	\begin{split}
		&\DEL \colon \mathcal{C}^1(\overline{\Omega}) \to \mathcal{C}^0(\overline{\hat{\Omega}},\R^d)\\
		&\DEL(L_d)(x_0,x_1,x_2) = \nabla_2 L_d(x_0,x_1) + \nabla_1 L_d(x_1,x_2).
	\end{split}
\end{equation}
Here $\nabla_j L_d$ denotes the partial derivatives with respect to the $j$th input argument of $L_d$.

Assume that the observations $\hat \Omega_0 = \{ (x_0^{(j)},x_1^{(j)},x_2^{(j)})\}_{j=1}^\infty$ correspond to a discrete Lagrangian dynamical system governed by $L_d^\rf \in \mathcal{C}^1(\overline\Omega)$ with globally Lipschitz continuous flow map $\mod{\overline {g_\rf}} \colon \Omega_a \to \Omega_b$, i.e.\ 
$\DEL(L_d^\rf)(x_0,\mod{\overline {g_\rf}}(x_0,x_1))=0$ for all $(x_0,x_1) \in \Omega_a$ and
$\mod{\overline {g_\rf}}(x_0^{(j)},x_1^{(j)})=(x_1^{(j)},x_2^{(j)})$ for all $j \in \N$.

\begin{lemma}\label{lem:BoundedPhiDiscrete}
	The linear functional $\Phi^{(\infty)} \colon \mathcal{C}^1(\overline{\Omega}) \to \mathcal{C}^0(\overline\Omega_a,\R^d)$ with
	\begin{equation}\label{eq:DefPhiInfLdTemporal}
		\Phi^{(\infty)}(L_d)(x_0,x_1) = \DEL(L_d)(x_0,\mod{\overline{g_\rf}}(x_0,x_1))
	\end{equation}
	is bounded.
\end{lemma}

\begin{proof}
Indeed, $\mod{\overline {g_\rf}}$ extends to a globally Lipschitz continuous map $\mod{\overline {g_\rf}} \colon \overline\Omega_a \to \overline\Omega_b$ such that $\Phi^{(\infty)} \colon \mathcal{C}^1(\overline{\Omega}) \to \mathcal{C}^0(\overline\Omega_a,\R^d)$ is a well-defined map between Banach spaces defined via \eqref{eq:DefPhiInfLdTemporal}.
Let $\| L_d\|_{\mathcal{C}^1(\overline\Omega)} \le 1$. In particular,
	\begin{equation}
		\sup_{(x_0,x_1) \in \Omega_a} \| \nabla_2 L_d(x_0,x_1)\| \le 1 
		\quad \text{and} \quad 
		\sup_{(x_1,x_2) \in \Omega_b} \| \nabla_2 L_d(x_1,x_2)\| \le 1.
	\end{equation}
	Therefore, by the triangle inequality
	\begin{equation}
		\begin{split}
			\sup_{(x_0,x_1) \in \Omega_a} \DEL(L_d)(x_0,\mod{\overline {g_\rf}}(x_0,x_1))
			&\le 1 +
			\sup_{(x_0,x_1) \in \Omega_a} \| \nabla_2 L_d(\mod{\overline {g_\rf}}(x_0,x_1))\|\\
			&\le 1 +
			\sup_{(x_1,x_2) \in \Omega_b} \| \nabla_2 L_d(x_1,x_2)\| \le 2.
		\end{split}
	\end{equation}
\end{proof}

We can now proceed in direct analogy to the continuous setting (\cref{sec:FormalSettingCont}) with $L$ replaced by $L_d$ and the functional $\Phi_N$ of \eqref{eq:PhiNContinuous} (normalisation conditions) replaced by the corresponding functional for discrete Lagrangians. The details are provided in the following.

Since for each $\overline{x}$ the evaluation functional $\mathrm{ev}_{\overline{x}} \colon f \mapsto f(\overline{x})$ on $\mathcal{C}^0(\overline\Omega_a,\R^d)$ is bounded, the following functions constitute bounded linear functionals for $j \in \N$:
\begin{align*}
	\Phi_j &\colon \mathcal{C}^1(\overline\Omega) \to \R^d, \quad \quad\Phi_j(L_d)=\Phi^{(\infty)} (L_d)(\overline{x}^{(j)})\\
	\Phi^{(j)} &\colon \mathcal{C}^1(\overline\Omega) \to (\R^d)^j, \quad \Phi^{(j)} =(\Phi_1,\ldots,\Phi_j).
\end{align*}

For a reference point $\overline{x}_b \in \Omega$ and for $p_b \in \R^d$, $\mod{c}_b \in \R$ we define the bounded linear functional
\begin{align}\label{eq:PhiNDiscrete}
	\gls{PhiN} &\colon \mathcal{C}^1(\overline\Omega) \to \R^{d+1}, \quad \quad \Phi_N(L)= \left( -\nabla_1 L_d(\overline{x_b}), L_d(\overline{x}_b) \right),
\end{align}
related to our normalisation condition for discrete Lagrangians. We will further use the shorthands $\Phi_b^{(k)} = (\Phi_1,\ldots,\Phi_k,\Phi_N)$ and $\Phi_b^{(\infty)} = (\Phi^{(\infty)} ,\Phi_N)$, and define
\begin{align*}
	y^{(k)}_{\mod{b}} &= (0,\ldots,0,p_b,\mod{c}_b) \in (\R^d)^k \times \R^d\times \R\\
	y^{(\infty)}_{\mod{b}} &= (0,p_b,\mod{c}_b) \in \mathcal{C}^0(\overline\Omega,\R^d) \times \R^d\times \R.
\end{align*}

In analogy to \cref{ass:LinB} we consider the following assumption.
\begin{assumption}\label{ass:LinBDiscrete}
	Assume that there is a Hilbert space $U$ with continuous embedding $U \hookrightarrow \mathcal{C}^1(\overline\Omega)$ such that
	\[\{L_d \in \mathcal{C}^1(\overline \Omega) \, | \, \Phi_b^{(\infty)}(L_d) = y^{(\infty)}_{\mod{b}} \} \cap U \not = \emptyset{.}\]
	In other words, $U$ is assumed to contain a Lagrangian consistent  with the normalisation and underlying dynamics.
\end{assumption}

The affine linear subspace{s}
\begin{align*}
	A^{(j)} &= \{ L_d\in U \, | \, \Phi_b^{(j)}(L_d) = y^{(j)}_{\mod{b}} \} \quad (j \in \N)\\
	A^{(\infty)} &= \{ L_d \in U \, | \, \Phi_b^{(\infty)}(L_d) = y^{(\infty)}_{\mod{b}} \}
\end{align*}
are closed in $U$, \mod{convex,} and not empty by \cref{ass:LinBDiscrete}.
\mod{By the Hilbert projection theorem \cite[§12.3]{Rudin1991}, the following minimisers exist and are uniquely defined:}
\begin{equation}\label{eq:MinProblemsDiscreteTemporal}
	\begin{split}
		{L_d}_{\mod{,}(j)} &\mod{:=} \arg \min_{{L_d} \in A^{(j)}} \| {L_d} \|_U \\
		{L_d}_{\mod{,}(\infty)} &\mod{:=} \arg \min_{{L_d} \in A^{(\infty)}} \| {L_d} \|_U{.}
	\end{split}
\end{equation}
Here $\| \cdot \|_U $ denotes the norm in $U$.

\begin{proposition}\label{prop:ConvergenceInBDiscrete}
	The minima ${L_d}_{\mod{,}(j)}$ converge to ${L_d}_{\mod{,}(\infty)}$ in the norm $\| \cdot \|_U$ and, thus, in $\| \cdot \|_{\mathcal C^1(\overline{\Omega})}$.
\end{proposition}

\begin{proof}
	The proof is in complete analogy to \cref{prop:ConvergenceInB}.
\end{proof}

\begin{proof}[Proof of \cref{thm:ConvergenceThmLdTemporalNonDegenerate}]\label{proof:ProofofConvThmLd}
{An application of \cref{thm:Extremizer} (\cref{app:CondExpandVar}) to the components of $\Phi_b^{(M)}$ considered as elements of the dual to the RKHS $U$ shows that} the unique minimisers $L_{d,(j)}$ in \eqref{eq:MinProblemsDiscreteTemporal} \mod{for $j \in \N$} are the conditional means {\eqref{eq:LdPosterior}} considered in \cref{thm:ConvergenceThmLdTemporalNonDegenerate}.
{Notice that the assumption of \cref{thm:Extremizer} on $y=y_b^{(M)}$ is fulfilled, see \cref{prop:CondDistPropLdCheck} (\cref{sec:ApplyPropCondDist}).}
Thus, \cref{thm:ConvergenceThmLdTemporalNonDegenerate} follows from \cref{prop:ConvergenceInBDiscrete}.
\end{proof}

%% file: convergence_rates.tex
\section{\mod{Lower bounds for} convergence rates of continuous and discrete Lagrangian models}\label{sec:ConvergenceRates}

Let $L_{(M)}$ denote the Lagrangian inferred from $M$ observations as in \cref{thm:ConvergenceThm} and let $L_{(\infty)}$ denote the limit as the observations densely fill a compact set.
We analyse how fast the learned equations of motions $\EL(L_{(M)})=0$ converge to the true equations of motions $\EL(L_{(\infty)})=0$ as the distance between observation data points converges to zero.
We will show that the extrapolation error $\|\EL(L_{(M)})(x,\dot x,\ddot x)\|$ for $(x,\dot x,\ddot x)$ an observation of the true dynamical system can be bounded. The bound tends to zero as $h^{r}$, where $h$ relates to the maximal distance between data points and $r$ is related to the smoothness of the true dynamics, \mod{provided that the kernel is sufficiently regular}. 
\mod{If} that the observation data fill the space at least as efficiently as uniform meshes, \mod{then} the bound tends to zero as $M^{-\frac{r}{2d}}$, where $M$ is the number of observation points. 

Away from degenerate points, the Euler--Lagrange equations implicitly define an acceleration field 
that expresses $\ddot x$ in terms of $(x,\dot x)$ such that $\EL(L_{(M)})(x,\dot x,\ddot x)=0$.
Roughly speaking, we will show that away from critical points, the convergence rate of the learned acceleration field to the true acceleration field is $h^{r}$ (or $M^{-\frac{r}{2d}}$ for uniform meshes) as well. Moreover, analogous statements will be shown for discrete Lagrangian models.

\subsection{Preliminaries: \mod{I}nterpolation and \mod{S}moothening theory}

Our proofs make use of statements from \mod{I}nterpolation and \mod{S}moothening theory \cite{Arcangeli2007,Narcowich2005,Wendland2005}. Let us recall notions and results that are relevant in our context.

\begin{definition}[Fill distance]\label{def:hFill}
To $\Omega \subset \R^{d'}$ and a finite subset $\Omega_0 \subset \overline{\Omega}$ we define the {\em fill distance} $h$ of $\Omega_0$ in $\Omega$ as
\begin{equation}\label{eq:hfillDef}
\gls{hFill} = \dist(\Omega_0,\overline{\Omega}) = \sup_{\mod{\overline{x} \in \overline{\Omega}}} \min_{\mod{\overline{x}_0 \in \Omega_0}} \| \overline{x}_0 - \overline{x} \|.
\end{equation}
\end{definition}

The fill distance of $\Omega_0$ in $\Omega$ coincides with the Hausdorff distance between the sets $\Omega_0$ and $\Omega$.

\begin{example}[Fill distance of uniform mesh and of Halton sequence]\label{rem:FillDistance}
	When $\overline{\Omega} \subset \R^{d'}$ is a ${d'}$-dimensional cube and $\Omega_0$ is a uniform mesh with mesh width $\Delta \overline{x}$ then $h_{\Omega_0} = \sqrt{{d'}} \Delta \overline{x}/2$. If $\Omega_0$ contains $M$ points,
	\[h_{\Omega_0} = \frac {\sqrt{{d'}}}{2(\sqrt[d']{M}-1)}.\]
	\Cref{fig:HausdorffDist} shows the fill distance $h_{\Omega_0}$ when $\Omega_0$ is an equidistant uniform mesh on a ${d'}$-dimensional cube $\overline \Omega$ and when $\Omega_0$ is a Halton sequence with the same number of elements. Here $2h_{\Omega_0}$ corresponds to the maximal distance between any two points in $\Omega_0 \cup \p \Omega$. It illustrates that in low dimensions Halton sequences reduce the fill distance roughly at a similar rate as uniform meshes.
\end{example}

\begin{figure}
	\centering
	\includegraphics[width=0.45\linewidth]{"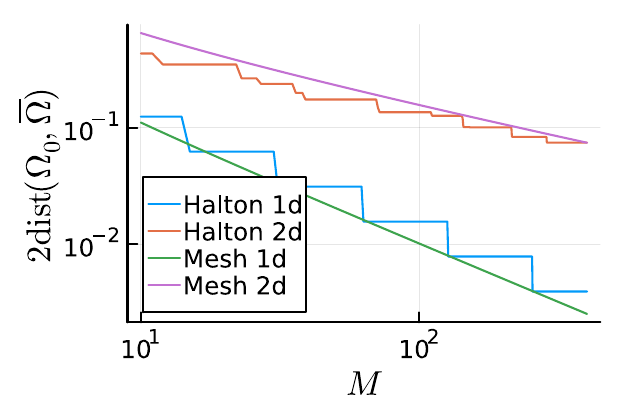"}
	\caption{Max.~distance between any two points in $\Omega_0 \cup \p \Omega$ for $\overline \Omega = [0,1]^d$, $d=1,2$, and $\Omega_0$ a uniform mesh or a Halton sequence with $M$ elements. (See \cref{rem:FillDistance}.)}\label{fig:HausdorffDist}
\end{figure}

In our analysis we will make use of the following theorem from interpolation and smoothening theory.

\begin{theorem}[Sobolev bounds]\label{thm:SobolevBoundsInterpolation}
Let $\Omega \subset \R^{d'}$ be a bounded domain with a Lipschitz continuous boundary. Let $r> \frac 12 {d'}$.
Then there exist constants $\delta_r,C_r>0$ (depending on $\Omega$ and $r$) such that for any finite $\Omega_0 \subset \overline{\Omega}$ with $h_{\Omega_0} \le \delta_r$, for any $u \in \gls{WSobolev}$ with $u|_{\Omega_0} \equiv 0$, and for any $l=0,\ldots,r$
\[
\| u\|_{W^l(\Omega)}
\le C_r (h_{\Omega_0})^{r-l} \| u\|_{W^r(\Omega)}.
\]
\end{theorem}

In the theorem's statement, $W^r(\Omega) = W^{r,2}(\Omega)$ denotes the Sobolev space (defined in \cref{rem:AssumptionTrueForK}). $u$ is continuous by the Sobolev embedding theorem \cite[§4]{SobolevSpacesAdams2003}. $u|_{\Omega_0}$ denotes the restriction of the function $u$ to the set $\Omega_0$.

\begin{proof}
This is a special case of \cite[Cor.~4.1]{Arcangeli2007}.
\end{proof}

\subsection{Convergence rates for continuous Lagrangian models}\label{sec:ConvRatesContLagrangian}

We will \mod{prove lower bounds for} the convergence rates of the inferred equations of motions and the acceleration field to true equations of motions and the true acceleration field as the fill-distance of observations converges to zero.

\begin{assumption}[Underlying system and RKHS]\label{ass:trueSystem}
Assume $\Omega \subset \R^{2d}$ is open, bounded and has locally Lipschitz boundary. Consider a kernel $K \colon \Omega \times \Omega \to \R$ such that the RKHS $U$ embeds continuously into the Sobolev space\footnote{See \cref{rem:AssumptionTrueForK} for a definition.} $W^{r+2}(\Omega)$ for $r>2+d$. Assume that the true acceleration $\ddot{x}$ can be described by a function $g_{\mathrm{ref}} \colon \Omega \to \R^d$ with $g_{\mathrm{ref}} \in (W^r(\Omega))^d$.
\end{assumption}

When \cref{ass:trueSystem} holds, then by the Sobolev embedding theorem \cite[§4]{SobolevSpacesAdams2003}, $W^{r+2}(\Omega)$ embeds continuously into $C^4(\overline{\Omega})$.
Therefore, the kernel $K$ necessarily fulfils sufficient smoothness properties such that for $p_b \in \R^d$, $c_b \in \R$, $\overline{x}_b \in \Omega$ we can define to any finite subset $\Omega_0 = \{(x,\dot x)\}_{j=1}^M \subset \Omega$ a Lagrangian $\gls{LOmega0} \in U$ by \eqref{eq:LPosterior}.

Consider data-driven equations of motions 
$\EL(L_{\Omega_0})(x,\dot x,\ddot x)=0$ inferred from finitely many observations $(x^{(j)},\dot x^{(j)},\ddot x^{(j)})$ with $\Omega_0 = \{(x^{(j)},\dot x^{(j)})\}_j$. 
The following \namecref{thm:ConvergenceRateL} provides a bound on the extrapolation error $\EL(L_{\Omega_0})(x,\dot x,\ddot x)$ on observations $(x,\dot x,\ddot x)$ of the true system.

\begin{theorem}[Convergence rates for equations of motion]\label{thm:ConvergenceRateL}
Let \cref{ass:trueSystem} hold. For $p_b \in \R^d$, $c_b \in \R$, $\overline{x}_b \in \Omega$ assume there exists a Lagrangian $L_{\mathrm{ref}} \in U$ consistent with the normalisation \eqref{eq:NormalisationThm} and the dynamics, i.e.\ $\EL(L_\mathrm{ref})(\overline{x},g_{\mathrm{ref}}(\overline{x})) =0$, for all $\overline x \in \overline{\Omega}$.
Denote by ${}_{k}\Phi^\infty(L)(\overline{x}) = {}_{k}\EL(L)(\overline{x},g_{\mathrm{ref}}(\overline{x}))$ for $L\in U$, $\overline x \in \Omega$
the $k$th component of $\EL(L)(\overline{x},g_{\mathrm{ref}}(\overline{x}))$ ($k=1,\ldots,d$).

Then there exist constants $\delta_r,C_r>0$ such that for all finite $\Omega_0 \subset \overline{\Omega}$ with
$h_{\Omega_0} = \dist(\Omega_0,\overline{\Omega}) < \delta_r$
and for all $l=0,1,\ldots,r$, $k=1,\ldots,d$
\[
\| {}_{k}\Phi^\infty(L_{\Omega_0})\|_{W^l(\Omega)}
\le C_r h^{r-l}_{\Omega_0} \|L_{\mathrm{ref}}\|_U.
\]
\end{theorem}

\begin{proof}
All components ${}_{k}\Phi^\infty$ of the map $\Phi^\infty \colon U \to (W^r(\Omega))^d$ have bounded operator norm:
\mod{f}or any $L \in U$ and any $k=1,\ldots,d$

\[
{}_{k}\Phi^{(\infty)}(L)
= \sum_{i=1}^{d}\frac{\p^2 L}{\p \dot x^k \p \dot x^i} \cdot g_\rf^i
+  \frac{\p^2 L}{\p x \p \dot x^k} \cdot \dot x^i
-\frac{\p L}{ \p  x^k}.
\]
In the above formula, $\dot x^i$ needs to be interpreted as the projection map sending a point $(x,\dot x) \in \Omega$ to the component $\dot x^i$.
Using the triangle inequality and the Cauchy-Schwarz inequality on the Hilbert space $W^r(\Omega)$ we have
\begin{align*}
\|{}_{k}\Phi^{(\infty)}(L)\|_{W^r(\Omega)}
&\le \sum_{i=1}^{d} \left( \left\| \frac{\p^2 L}{\p \dot x^k \p \dot x^i} \cdot g_\rf^i \right\|_{W^r(\Omega)}
+ \left\| \frac{\p^2 L}{\p x^i \p \dot x^k} \cdot \dot x^i \right\|_{W^r(\Omega)}\right)
+\left\|\frac{\p L}{ \p  x^k} \right\|_{W^r(\Omega)}\\
&\le \sum_{i=1}^{d} \left( \left\| \frac{\p^2 L}{\p \dot x^k \p \dot x^i}  \right\|_{W^r(\Omega)} \left\| g_\rf^i \right\|_{W^r(\Omega)}
+ \left\| \frac{\p^2 L}{\p x^i \p \dot x^k}  \right\|_{W^r(\Omega)} \left\| \dot x^i \right\|_{W^r(\Omega)}\right)\\
&\quad +\left\|\frac{\p L}{ \p  x^k} \right\|_{W^r(\Omega)}\\
&\le \| L\|_{W^{r+2}(\Omega)} \left(1+\sum_{i=1}^{d}(\left\| g_\rf^i \right\|_{W^r(\Omega)}
+  \left\| \dot x^i \right\|_{W^r(\Omega)})\right)
\end{align*}

As the embedding $U \hookrightarrow W^{r+2}(\Omega)$ is continuous, there exists $c_r >0$ such that $\| L\|_{W^{r+2}(\Omega)}  \le c_r  \|L\|_U$.
Thus, ${}_{k}\Phi^\infty \colon U \to W^r(\Omega)$ has bounded operator norm
$\|{}_{k}\Phi^\infty\|_{U,W^r(\Omega)}$.

By \cref{thm:SobolevBoundsInterpolation} there exist $\delta_r>0$, $\tilde C_r >0$ such that for all finite $\Omega_0 \subset \overline\Omega$ (defining $L_{\Omega_0}$) with $h_{\Omega_0}<\delta_r$ and all $l=0,\ldots,r$
\[
\|{}_{k}\Phi^\infty(L_{\Omega_0})\|_{W^l(\Omega)}
\le \tilde C_r h^{r-l} \|{}_{k}\Phi^\infty(L_{\Omega_0})\|_{W^r(\Omega)}.
\]
As by \cref{rem:GPisMinProblem}, $L_{\Omega_0} \in U$ minimizes the RKHS-norm while fulfilling the normalisation condition \eqref{eq:NormalisationThm} and $\Phi^\infty(L_{\Omega_0})(\overline{x}) =0$ for all $\overline{x}\in \Omega_0$. As $L_{\mathrm{ref}} \in U$ fulfils \eqref{eq:NormalisationThm} and the stricter condition $\Phi^\infty(L_{\mathrm{ref}})(\overline{x}) =0$ for all $\overline{x}\in \Omega$, we have $\|L_{\Omega_0}\|_U \le \|L_{\mathrm{ref}}\|_U$. Therefore, combining all estimates, we arrive at
\begin{align*}
\|{}_{k}\Phi^\infty(L_{\Omega_0})\|_{W^l(\Omega)}
&\le \tilde C_r h^{r-l} \|{}_{k}\Phi^\infty(L_{\Omega_0})\|_{W^r(\Omega)}\\
&\le \tilde C_r h^{r-l} \|{}_{k}\Phi^\infty\|_{U,W^r(\Omega)}
\|L_{\Omega_0}\|_U\\
&\le \tilde C_r h^{r-l} \|{}_{k}\Phi^\infty\|_{U,W^r(\Omega)}
 \|L_{\mathrm{ref}}\|_U.
\end{align*}
This proves the claim.
\end{proof}

As by \cref{rem:FillDistance}, when observations are obtained over a sequence of uniform meshes in $\Omega \subset \R^{2d}$ then the \mod{lower bound for the} convergence rate \mod{guaranteed by} \cref{thm:ConvergenceRateL} is $M^{-\frac r {2d}}$, where $M$ is the number of observations.

When the dynamics and the kernel are smooth, then \cref{thm:ConvergenceRateL} can be applied for any $r$. However, we expect that the constants $\delta_r$, $C_r$ grow with $r$. Thus, higher and higher convergence rates become dominant as the fill distance $h_{\Omega_0}$ decreases.
This is discussed in the following \nameCref{cor:ConvRateEQMSmooth}.

\begin{corollary}[Convergence rates equations of motions, Gaussian kernel]\label{cor:ConvRateEQMSmooth}
Let $\Omega \subset \R^d$ open, bounded with locally Lipschitz boundary and $K \colon\Omega \times \Omega \to \R$ the squared exponential kernel. Assume the observed acceleration field $g_{\mathrm{ref}}\colon \Omega \to \R^d$ is smooth and all derivatives are bounded on $\overline{\Omega}$.
For $p_b \in \R^d$, $c_b \in \R$, $\overline{x}_b \in \Omega$ assume there exists a Lagrangian $L_{\mathrm{ref}} \in U$ consistent with the normalisation \eqref{eq:NormalisationThm} and the dynamics.
Then for all $r \in \N$ there exist $C_r,\delta_r>0$ such that for all finite subsets $\Omega_0 \subset \Omega$ (defining $L_{\Omega_0}$) with $h_{\Omega_0}<\delta_r$ and for all $l =0,\ldots,r$
\[
\|\overline{x} \mapsto {}_{k}\EL(L_{\Omega_0})(\overline{x},g_{\mathrm{ref}}(\overline{x})) \|_{W^l(\Omega)} \le C_r h_{\Omega_0}^{r-l}
\|L_{\mathrm{ref}}\|_U
\]
for any component $k=1,\ldots,d$.
\end{corollary}

\begin{proof}
As $K$ is the squared exponential kernel, its reproducing kernel Hilbert space $U$ embeds continuously into any Sobolev space $W^m(\Omega)$ ($m>1$) \cite[Thm.4.48]{ChristmannSteinwart2008RKHS}. Thus, \cref{ass:trueSystem} is fulfilled for any $r>2+d$. Therefore, for $r>2+d$ \mod{the} statement follows by \cref{thm:ConvergenceRateL}.
For $r \le 2+d$ the statement can be deduced from the statement with $r = 3+d$ for a sufficiently small $0<\delta_r<\delta_{3+d}$ and sufficiently large $C_r>C_{3+d}$.
\end{proof}

For a Lagrangian $L \in \mathcal{C}^2(\Omega)$ at non-degenerate points, i.e.\ where the matrix $\frac{\p^2 L}{\p \dot x \p \dot x}$ is invertible, we can define the acceleration field
\[
g(L)(\overline{x}) = \left(\frac{\p^2 L}{\p \dot x \p \dot x}(\overline{x}) \right)^{-1} \left(\frac{\p L}{ \p  x} (x,\dot{x}) - \frac{\p^2 L}{\p x \p \dot x} (x,\dot{x}) \cdot \dot x \right).
\]
It fulfils $\EL(L)(\overline{x},g(L)(\overline{x}))=0$.
We have the following \mod{lower bound for the rate of the} pointwise convergence of the acceleration field.

\begin{corollary}[Convergence rates of acceleration field]\label{cor:RateAcc}
Under the assumptions of \cref{thm:ConvergenceRateL}, consider a sequence $(\overline{x}^{(j)})_{j=1}^\infty \subset \Omega$ defining a dense subset of $\Omega$. Consider the Lagrangians $L_{(j)}$, $L_{(\infty)}$ characterised in \cref{thm:ConvergenceThm}. Assume $L_{(\infty)}$ is non-degenerate at $\overline{x} \in \Omega$.
Let $\Omega_0^{k} := \{\overline{x}^{(j)}\}_{j=1}^k$.
Then there exist $J\in \N$, $C_r>0$ such that for all $k>J$
\[
\| g(L_{(k)})(\overline{x}) - g_{\mathrm{ref}}(\overline{x}) \| 
\le C_r (h_{\Omega_0^k})^r .
\]
\end{corollary}

Again, as by \cref{rem:FillDistance}, when the observations are obtained over uniform meshes in $\Omega \subset \R^{2d}$, then \cref{cor:RateAcc} \mod{guarantees that the pointwise acceleration error goes to zero at least as fast as} $M^{-\frac r {2d}}$, where $M$ is the number of samples.

\begin{proof}
The Lagrangian $L_{(\infty)}$ is non-degenerate at $\overline{x}$, i.e.\ all eigenvalues of the symmetric matrix $\frac{\p^2 L_{(\infty)}}{\p \dot x \p \dot x}(\overline{x})$
are non-zero. Let $\lambda$ be the eigenvalue closest to 0.
Since the Lagrangians $L_{(j)}$ converge to $L_{(\infty)}$ in $\| \cdot \|_{\mathcal{C}^2(\overline{\Omega})}$, there exists $J_1$ such that for all $k>J_1$
the eigenvalue $\lambda_j$ of $\frac{\p^2 L_{(j)}}{\p \dot x \p \dot x}(\overline{x})$ closest to zero fulfils $| \lambda_j -\lambda| < \frac  {|\lambda| }2$.
As
\[\Omega_0^1 \subset \Omega_0^2 \subset \ldots \subset \bigcup_{j=1}^\infty \Omega_0^j \subset  \overline{\Omega}\]
and $\bigcup_{j=1}^\infty \Omega_0^j $ is dense in the compact set $\overline{\Omega}$, we have $h_{\Omega_0^j} \to 0$. By \cref{thm:ConvergenceRateL} and $U \subset C(\overline{\Omega})$, there exists $J_2 \in \N$, $C>0$ such that for all $k>J_2$
\[
\|\EL(L_{(k)})(\overline{x},g_\rf(\overline{x}))\| \le C h_{\Omega_0^k}^r,
\]
where the norm $\|L_{(\infty)}\|_U$ has been absorbed in the constant $C$.
For $k>J:=\max(J_1,J_2)$ we have
\begin{align*}
	C (h_{\Omega_0^k})^r 
	&\ge \|\EL(L_{(k)})(\overline{x},g_\rf(\overline{x}))  - \underbrace{\EL(L_{(k)})(\overline{x},g(L_{(k)})(\overline{x}))}_{=0}\|\\
	&=\left\| \frac{\p^2 L_{(k)}}{\p \dot x \p \dot x}(\overline{x}) (g_\rf (\overline{x}) - g(L_{(k)})(\overline{x})) \right\|\\
	&\ge \frac {|\lambda|} 2 \left\| g_\rf (\overline{x}) - g(L_{(k)})(\overline{x}) \right\|
\end{align*}%
This proves the claim.
\end{proof}

\subsection{\mod{Numerical convergence test for smooth dynamical systems}}

\Cref{fig:Convergence} shows a convergence plot for the relative error in predicted acceleration $\mathrm{err}_{\mod{g}}$, i.e.\ of
\[
\mathrm{err}_{\mod{g}}(\overline{x}) = \frac{\|\mod{g}(L_{\mod{(M)}})\mod{(\overline{x})}-\mod{g}(L_\rf)\mod{(\overline{x})}\|_{\R^d}}{\|\mod{g}(L_\rf)\mod{(\overline{x})}\|_{\R^d}},
\]
\mod{where $M$ denotes the number of observations.}
The data for the plot in \cref{fig:Convergence} was computed for the 1d harmonic oscillator $L_\rf(x) = \frac 12 \dot x^2 - \frac 12 x^2$ with $(x,\dot x) \in [-1,1]^2$ in quadruple precision. For each $M \in \{2^1,2^2,\ldots,2^6\}$ the error $\mathrm{err}_{\mod{g}}(\overline{x})$ was evaluated on a uniform mesh with $10\times 11$ mesh points in $[-1,1]\times [-1,1] \mod{\subset} T\R$. The plot shows \mod{the maximum value of $\mathrm{err}_g$ on the evaluation mesh.}
We can see convergence with errors levelling out due to round-off errors at approximately $10^{-11}$.
Moreover, as $M$ increases, higher and higher convergence rates become dominant before round-off errors dominate.

\begin{figure}
	\centering
	\includegraphics[width=0.49\linewidth]{"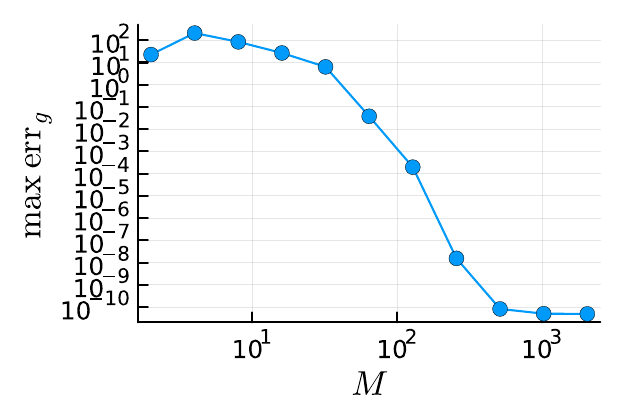"}
	\caption{Convergence of $\mod{g}(L_{\mod{(M)}})$ to true acceleration data.}\label{fig:Convergence}
\end{figure}

\mod{The result is consistent with \cref{cor:RateAcc}}: as the reference Lagrangian and the kernel in the experiment are smooth, \cref{cor:RateAcc} applies for any $r \in \N$.
This confirms our observation that as the number of observation points $M$ increases (and $h_{\Omega_0}$ shrinks as visualised in \cref{fig:HausdorffDist}), higher and higher convergence rates become dominant until round-off errors become dominant.

\subsection{Convergence rates of discrete Lagrangian models}

We now turn to discrete Lagrangian models. For preparation, we prove the following Cauchy-Schwarz-type inequality.

\begin{lemma}\label{lem:CSSobolev}
Let $\Omega \subset \R^{2d}$ an open, non-empty, bounded domain with Lipschitz boundary. Let $r>d$ and $\mod{\overline g} \colon \Omega \to \Omega \subset \R^{2d}$ with $\mod{\overline g} \in (W^r(\Omega))^{2d}$. Then there exists $C_g >0$ such that for all $f \in W^r(\Omega)$ 
\[
\| f \circ \mod{\overline g} \|_{W^{r-d-1}(\Omega)} \le C_{\mod{\overline g}} \|f\|_{W^r(\Omega)}.
\]
\end{lemma}

\begin{proof}
Denote coordinates of $\Omega$ by $z^1,\ldots,z^{2d}$.
Let $f \in W^r(\Omega)$, \mod{$\overline g \in (W^r(\Omega))^{2d}$} and let $s \le r-d-1$. For $m>d$ the Sobolev embedding $W^m(\Omega) \subset \mathcal{C}(\overline{\Omega})$ holds. Therefore, the derivatives $\p^\alpha f = \frac{\p^{|\alpha|}f}{(z^1)^{\alpha_1}\ldots (z^{2d})^{\alpha_{2d}}}$ of $f$ fulfil $\p^\alpha f  \in \mathcal{C}(\overline{\Omega})$ for all multi-indices $\alpha$ with $|\alpha|\le s$. Moreover, each component of $\p^\alpha \mod{\overline g}$ with $|\alpha|\le s$ lies in $L^2(\overline{\Omega})$. 

A multivariate version of the Faá di Bruno formula \cite{Hardy2006} shows
\begin{gather*}
\p^\alpha (f \circ \mod{\overline g}) = \sum_\pi (\p^{\alpha(\pi)} f) \circ \mod{\overline g} \cdot \mod{\overline g}_\pi ,\end{gather*}%
where $\pi$ runs through the set of partitions of the unordered $|\alpha|$-tuple (multi-set)
\[\{\underbrace{1,\ldots,1}_{\alpha_1 \text{ times}},\ldots, \underbrace{2d,\ldots,2d}_{\alpha_{2d} \text{ times}} \}\] and defines multi-indices $\alpha(\pi)$ for derivatives with $|\alpha(\pi)|\le s$.

The expression $\mod{\overline g}_\pi$ consists of products of derivatives of $\mod{\overline g}$ of order less than or equal to $s$. For each $\pi$ the norm $\| \mod{\overline g}_\pi \|_{L^2(\Omega)}$ can, therefore, be bounded by a repeated application of the Cauchy\mod{-Schwarz} inequality in $L^2(\Omega)$. 
Moreover, $\p^{\alpha(\pi)} f  \in \mathcal{C}(\overline{\Omega})$. As $W^{r-i}(\Omega) \subset C(\overline{\Omega})$ for all $i \le s$,
$\|\p^\alpha (f \circ \mod{\overline g})\|_{L^2(\Omega)}$ is bounded in terms of $\|f\|_{W^r(\Omega)}$ and a $\mod{\overline g}$ dependent constant $C_{\mod{\overline g}}>0$.
\end{proof}

\begin{proposition}\label{prop:OpPhiLdBound}
Let $\Omega \subset \R^{2d}$ an open, non-empty, bounded domain with Lipschitz boundary. Let $r>d$ and $g_{\mathrm{ref}} \in (W^r(\Omega))^d$. Consider the map $\Phi^{(\infty)}$ defined by
\[
\Phi^{(\infty)}(L_d)(\overline{x}) = \mathrm{DEL}(L_d)(\overline{x},g_{\mathrm{ref}}(\overline x)).
\]
The map $\Phi^{(\infty)}$ considered as a linear operator $\Phi^{(\infty)} \colon W^{r+1}(\Omega) \to (W^{r-d-1}(\Omega))^d$ is bounded.
\end{proposition}

\begin{proof}
Let $k \in \{1,\ldots,d\}$ and let ${}_k(\cdot)$ denote the $k$th component of a function. For $(x_0,x_1) \in \Omega$ define $\mod{\overline {g_\rf}}(x_0,x_1) = (x_1,g_\rf (x_0,x_1))$. We have $\mod{\overline {g_\rf}} \in (W^r(\Omega))^{2d}$. Let $L_d \in W^{r+1}(\Omega)$. Let $f = {}_k(\nabla_1 L_d)$. We have $f \in W^r(\Omega)$.

Now $\|{}_k(\Phi^{(\infty)})(L_d)\|_{W^{r-d-1}}$ can be bounded in terms of $\|L_d\|_{W^{r+1}}$ using \cref{lem:CSSobolev}:
\begin{align*}
	\| {}_k (\Phi^{(\infty)})(L_d) \|_{W^{r-d-1}}
	&\le \| {}_k (\nabla_2 L_d) \|_{W^{r-d-1}} + \|f \circ \mod{\overline {g_\rf}} \|_{W^{r-d-1}}\\
	&\le \| L_d\|_{W^{r}} + C_{\mod{\overline {g_\rf}}} \| f\|_{W^r}\\
	&\le \| L_d\|_{W^{r+1}} + C_{\mod{\overline {g_\rf}}} \| L_d\|_{W^{r+1}}
	=(1+C_{\mod{\overline {g_\rf}}})\| L_d\|_{W^{r+1}}
\end{align*} for a $\mod{\overline {g_\rf}}$ dependent constant $C_{\mod{\overline {g_\rf}}} >0$.
\end{proof}

\begin{assumption}\label{ass:ConvergenceRatesLd}
Let $\Omega \subset \R^{2d}$ an open, non-empty, bounded domain with locally Lipschitz boundary. Consider a kernel $K \colon \Omega \times \Omega \to \R$ such that the RKHS $U$ embeds continuously into the Sobolev space $W^{r+1}(\Omega)$ for $r>2d+1$. Assume that the true discrete dynamical system $(x_0,x_1) \mapsto (x_1,x_2)$ on $\Omega$ can be described by a map $g_{\mathrm {ref}} \in (W^r(\Omega))^d$, where $x_2 = g_{\mathrm {ref}}(x_0,x_1)$.
\end{assumption}

\begin{remark}[Stricter smoothness assumption]
	Comparing \cref{ass:ConvergenceRatesLd} and \cref{ass:trueSystem}, the smoothness assumptions on the dynamics and on the RKHS $\mathcal{U}$ appear to be stricter for discrete Lagrangian models than for continuous models. This is related to the requirement that the target space of $\Phi^{(\infty)}$ (\cref{prop:OpPhiLdBound}) embeds into $\mathcal{C}(\overline{\Omega})$ to apply smoothening theory (\cref{thm:SobolevBoundsInterpolation}).
\end{remark}

When \cref{ass:ConvergenceRatesLd} holds, then by the Sobolev embedding theorem \cite[§4]{SobolevSpacesAdams2003}, $W^{r+1}(\Omega)$ embeds continuously into $C^{2+d}(\overline{\Omega}) \subset C^2(\overline{\Omega})$.
Therefore, the kernel $K$ necessarily fulfils sufficient smoothness properties such that for $p_b \in \R^d$, $c_b \in \R$, $\overline{x}_b \in \Omega$ we can define to any finite subset $\Omega_0 = \{(x_0,x_1)\}_{j=1}^M \subset \Omega$ a Lagrangian $L_{d,\Omega_0} \in U$ by \eqref{eq:LdPosterior}.

The following \namecref{thm:ConvergenceRateLd} provides a bound on the extrapolation error $\DEL({L_{d,\Omega_0}})(x_0,x_1,x_2)$ on observations $(x_0,x_1,x_2)$ of the true system, when ${L_{d,\Omega_0}}$ is inferred from finitely many observations. \Cref{thm:ConvergenceRateLd} corresponds to \cref{thm:ConvergenceRateL}, which relates to continuous Lagrangian models.

\begin{theorem}\label{thm:ConvergenceRateLd}
Under \cref{ass:ConvergenceRatesLd}, assume that for $p_b \in \R^d$, $c_b \in \R$, $\overline{x}_b \in \Omega$ there exists a discrete Lagrangian $L_d^\rf$ consistent with the normalisation \eqref{eq:NormalisationThmLdTemporal} and the dynamics, i.e.\ $\mathrm{DEL}(L_d^\rf)(\overline{x},g_\rf (\overline{x}))=0$ for all $\overline{x}\in \Omega$.
Denote by ${}_{k}\Phi^\infty(L_d)(\overline{x}) = {}_{k}\DEL(L_d)(\overline{x},g_{\mathrm{ref}}(\overline{x}))$ for $L_d\in U$, $\overline x \in \Omega$
the $k$th component of $\DEL(L_d)(\overline{x},g_{\mathrm{ref}}(\overline{x}))$ ($k=1,\ldots,d$).

Then there exist constants $\delta_r,C_r>0$ such that for all finite $\Omega_0 \subset \overline{\Omega}$ (defining $L_{d,\Omega_0}$) with
$h_{\Omega_0} = \dist(\Omega_0,\overline{\Omega}) < \delta_r$
and for all $l=0,1,\ldots,r-d-1$, $k=1,\ldots,d$
\[
\| {}_{k}\Phi^\infty(L_{d,\Omega_0})\|_{W^l(\Omega)}
\le C_r h^{r-d-1-l}_{\Omega_0} \|L_d^{\mathrm{ref}}\|_U.
\]

\end{theorem}

\begin{proof}
Let $C_{r,\Phi}>0$ be a bound for the operator norm of $\Phi^{(\infty)} \colon W^{r+1}(\Omega) \to (W^{r-d-1}(\Omega))^d$ (see \cref{prop:OpPhiLdBound}).
As $r>2d+1$, by \cref{thm:SobolevBoundsInterpolation} there exists $\delta_r, C_r>0$ such that for all finite subsets $\Omega_0 \subset \overline{\Omega}$ (defining $L_{d,\Omega_0}$) with $h_{\Omega_0} \le \delta_r$ and for all $l=0,\ldots,r-d-1$ we have
\begin{align*}
\| \Phi^{(\infty)}(L_{d,\Omega_0}) \|_{W^l(\Omega)}
&\le \tilde C_r (h_{\Omega_0})^{r-d-1-l} \|  \Phi^{(\infty)}(L_{d,\Omega_0}) \|_{W^{r-d-1}(\Omega)}\\
&\le \tilde C_r (h_{\Omega_0})^{r-d-1-l} C_{r,\Phi} \|  L_{d,\Omega_0}\|_{W^{r+1}(\Omega)}\\
&\le \tilde C_r   (h_{\Omega_0})^{r-d-1-l} C_{r,\Phi} \tilde c\|  L_{d,\Omega_0}\|_U,
\end{align*}
where $\tilde c$ is related to the embedding $U\subset W^{r+1}$.
Discrete Lagrangians obtained via \eqref{eq:LdPosterior} fulfil a minimisation principle as explained in the proof of \cref{thm:ConvergenceThmLdTemporalNonDegenerate} (in direct analogy to \cref{rem:GPisMinProblem}, which is formulated for continuous Lagrangians). 
Thus $ \|  L_{d,\Omega_0}\|_U \le \|  L_d^\rf\|_U$. This completes the proof.
\end{proof}

For $L_d \in \mathcal{C}^1(\Omega)$, $(x_0^\ast,x_1^\ast),(x_1^\ast,x_2^\ast) \in \Omega$ with $\DEL(L_d)(x_0^\ast, x_1^\ast,x_2^\ast)=0$ and $\nabla_{1,2}L_d(x_1^\ast,x_2^\ast) = \frac{\p^2 L_d}{\p x_1 \p x_2}(x_1^\ast,x_2^\ast)$ invertible, the triple $(x_0^\ast,x_1^\ast,x_2^\ast)$ is called {\em non-degenerate motion segment} of $L_d$.
By the implicit function theorem we can define a unique continuous map $g$ on a connected open neighbourhood $\mathfrak O$ of $(L_d,(x_0^\ast,x_1^\ast)) \in \mathcal{C}^1(\Omega) \times \Omega$ 
with $g(L_d)(x_0^\ast,x_1^\ast)=x_2^\ast$ and
\[
\mathrm{DEL}(L_d)(\overline{x},g(L_d)(\overline{x}))=0 \quad \forall (L_d,\overline x) \in \mathfrak{O}.
\]
The map $g(L_d)$ is the {\em discrete evolution rule} of the discrete dynamical system defined by the Lagrangian $L_d$.


We have the following \mod{lower bound for the rate of the} pointwise convergence \mod{of the discrete evolution map}.

\begin{corollary}[Convergence rates discrete evolution rule]\label{cor:RateAccLd}
In the setting of \cref{thm:ConvergenceThmLdTemporalNonDegenerate} assume $\Omega=\Omega_a=\Omega_b$ and that \cref{ass:ConvergenceRatesLd} is fulfilled in addition.
Let $\overline{x}^\ast=(x_0^\ast,x_1^\ast),(x_1^\ast,x_2^\ast) \in \Omega$ with $x_2^\ast=g_\rf(\overline{x}^{\mod{\ast}})$ be a nondegenerate motion sequence of the limit Lagrangian $L_{d,(\infty)}$ defined in \cref{thm:ConvergenceThmLdTemporalNonDegenerate}. 
Denote $\Omega_0^{k} := \{\overline{x}^{(j)}\}_{j=1}^k$.
Then there exist $K\in \N$, $C_r>0$ such that for all $k>K$ the discrete evolution $g(L_{d,(k)})(\overline{x}^\ast)$ can be defined
and
\begin{equation*}\label{eq:cor_RateAccLd}
\| g(L_{d,(k)})(\overline{x}^\ast) - g_{\mathrm{ref}}(\overline{x}^\ast) \| 
\le C_r (h_{\Omega_0^k})^{r-d-1}.
\end{equation*}
\end{corollary}

\begin{proof}
By \cref{thm:ConvergenceThmLdTemporalNonDegenerate}, $L_{d,(k)}$ converges to $L_{d,(\infty)}$ in the RKHS $U$, which is continuously embedded into $\mathcal C^2(\overline \Omega)$ by \cref{ass:ConvergenceRatesLd}.
Therefore and by the non-degeneracy properties of $L_{d,(\infty)}$, there exists a neighbourhood $O$ of $g_\rf(\overline{x}^\ast)$, an index $K \in \N$, and $\delta>0$ such that for all $k>K$ and all $\overline{x} \in O$
each row and each column vector of $\nabla_{1,2}L_{d,(k)}(\overline x) = \frac{\p^2 L_{d,(k)}}{\p x^1 \p x^2}(\overline{x})$ have norm at least $\delta >0$.
%
We can assume $O$ to be convex and $K$ so large that for all $k>K$ the line segment between $g(L_{d,(k)})(\overline{x}^\ast)$ and $g(L_{d,(\infty)})(\overline{x}^\ast)$ 
is contained in $O$.

Let $j \in \{1,\ldots,d\}$ denote an index. Again, we denote the component of a function by a lower-left-aligned index.
By \cref{thm:ConvergenceRateLd} (with $l=0$) there exists $\tilde C_r>0$ such that for all $k>K$
\begin{align*}
\tilde C_r (h_{\Omega_0^k})^{r-d-1}
&\ge \| {}_j\DEL(L_{d,{(k)}})(\overline{x}^\ast,g_\rf (\overline{x}^\ast))\|\\
&=  \| {}_j\DEL(L_{d,{(k)}})(\overline{x}^\ast,g_\rf (\overline{x}^\ast)) 
- \underbrace{{}_j\DEL(L_{d,{(k)}})(\overline{x}^\ast,g(L_{d,{(k)}})(\overline{x}^\ast))}_{=0}  \|\\
&= \| {}_j\nabla_1 L_{d,(k)}(x_1^{\ast},g_\rf (\overline{x}^\ast)) - {}_j\nabla_1 L_{d,(k)}(x_1^{\ast},g(L_{d,(k)}) (\overline{x}^\ast)) \|\\
&= \| \nabla_2({}_j\nabla_{1})L_{d,{(k)}}(x_1^{\ast},x')^\top ( g_\rf (\overline{x}^\ast) - g(L_{d,{(k)}})(\overline{x}^\ast))\|\\
& \ge \delta \|g_\rf (\overline{x}^\ast) - g(L_{d,{(k)}})(\overline{x}^\ast))\|.
\end{align*}
Above, $x'$ lies on the line segment between $g(L_{d,(k)})(\overline{x}^\ast)$ and $g(L_{d,(\infty)})(\overline{x}^\ast)$. Its existence is guaranteed by the intermediate value theorem.
The expression $\nabla_2({}_j\nabla_{1})L_{d,{(k)}}$ denotes the gradient of ${}_j\nabla_{1}L_{d,{(k)}}$ with respect to the second input slot of $L_{d,{(k)}}$.
The last inequality holds true since the norm of each row and each column of $ \nabla_{1,2}L_{d,{(k)}}(\overline{x}^\ast,x')$ is bounded from below by $\delta>0$. Thus the theorem follows with $C_r = \frac {\tilde C_r} \delta$.
\end{proof}

%% file: Summary.tex
\section{Summary}\label{sec:Summary}

We have introduced a method to learn general continuous Lagrangians and discrete Lagrangians from observational data of dynamical systems that are governed by variational ordinary differential equations. The method is based on kernel-based, meshless collocation methods for solving partial differential equations \cite{SchabackWendland2006}. In our context, collocation methods are used to solve the Euler--Lagrange equations that we interpret as a partial differential equations for a Lagrangian function $L$, or discrete Lagrangian $L_d$, respectively.
Additionally, the use of Gaussian processes gives access to a statistical framework that allows for a quantification of the model uncertainty of the identified dynamical system. This could be used for adaptive sampling of data points. Uncertainty quantification can be efficiently computed for any quantity that is linear in the Lagrangian, such as the Hamiltonian or symplectic structure of the system, which is of relevance in the context of system identification.
We prove the convergence of the methods to a true Lagrangian and prove \mod{lower bounds for} convergence rates for the inferred equations of motion, acceleration fields, and evolution rules as the maximal distance \mod{$h$} of observation data points converges to zero.
\mod{Indeed, provided the model's kernel and the underlying dynamic is sufficiently regular, we can guarantee convergence of the acceleration field with rate at least $h^r$ when the true underlying acceleration field is $r$-times continuously differentiable. Similar results are shown for the discrete evolution map, when discrete Lagrangians are learnt.}

The article overcomes the major difficulty that Lagrangians are not uniquely determined by a system's motions and the presence of degenerate solutions to the Euler--Lagrange equations. 
This is tackled by a careful consideration of {regularisation} conditions that reduce the gauge freedom of Lagrangians but do not restrict the generality of the ansatz. Our method profits from implicit regularisation that can be understood as an extremization of a reproducing kernel Hilbert space norm, based on techniques of game theory \cite{OwhadiScovel2019}. This interpretation as convex optimisation problems is the key ingredient that allows us to provide a rigorous proof of convergence of the method as the maximal distance of observation data points converges to zero.

{In \cite{LagrangianGPpde} we have extended} the method to dynamical systems governed by variational partial differential equations.
Another direction {of research} is {to adapt the method to dynamical systems with low regularity such as systems with collisions and to incorporate noise models into our statistical framework. Furthermore, a}
combination with detection methods for Lie group variational symmetries \cite{SymHNN,SymLNN} or with detection methods of travelling waves \cite{DLNNDensity,DLNNPDE} {is of interest}. This may allow for a quantitative analysis of the interplay of symmetry assumptions and model uncertainty.

%% file: statements_ack_data.tex
\section*{Acknowledgments}
\mod{The author would like to thank Konstantin Sonntag for a helpful discussion that has advanced the convergence proof. Moreover, the author acknowledges the Ministry for Culture and Science of the State of North Rhine-Westphalia (MKW), Germany, and computing time provided by the Paderborn Center for Parallel Computing (PC2), Paderborn, Germany.
}

\section*{Data availability}
The data that support the findings of this study are openly available in the GitHub repository Christian-Offen/Lagrangian\_GP at\\
\url{https://github.com/Christian-Offen/Lagrangian_GP}.
An archived version \cite{LagrangianGPSoftware} of release v1.0 of the GitHub repository is openly available at \url{https://doi.org/10.5281/zenodo.11093645}.

%% file: Appendix_Owhadi_Recap.tex
\section{{Gaussian fields}}\label{app:GaussianField}

\subsection{Definitions}

We recall from \cite{OwhadiScovel2019} definitions and properties of Gaussian fields and their interpretation as weak random variables.

\begin{definition}\label{def:SymOperator}
	Let $V$ be a topological vector space and $V^\ast$ its topological dual.
	A linear operator $T \colon V^\ast \to V$ is {\em positive symmetric} if $\psi(T\phi) = \phi (T \psi)$ for all $\phi,\psi \in V^\ast$ and $\phi (T \phi)\ge 0$ for all $\phi \in V^\ast$.
\end{definition}

Let $(\gls{B}, \| \cdot \|_B)$ be a separable Banach space with quadratic norm $\| \cdot \|_B$, i.e.\ there exists a linear, positive symmetric, bijection $\gls{Q} \colon B^\ast \to B$ such that $\| u \|_B = (Q^{-1}u)(u)$.
Even though this implies that $B$ is a Hilbert space, the Banach space terminology is used as the dual pairing of $B^\ast$ and $B$ does not coincide with the inner product pairing via the Riesz representation theorem. Moreover, as any positive symmetric linear operator $B^\ast \to B$ is automatically continuous \cite[Prop.~11.2]{OwhadiScovel2019}, $Q \colon B^\ast \to B$ is continuous.

\begin{definition}[{{\cite[Def.~17.3]{OwhadiScovel2019}}}]\label{def:GPGeneralTheory}
Let $\gls{T} \colon B^\ast \to B$ be a positive symmetric linear operator, $u \in B$, $(\mathcal A, \Sigma, \mathbb{P})$ a probability space with $\mathbb{P}$ a Borel measure, and $H \subset L^2(\mathcal A, \Sigma, \mathbb{P})$ a linear subspace such that each $X \in H$ is a Gaussian random variable. A linear map
\[
\gls{xi} \colon B^\ast \to H \subset L^2(\mathcal A, \Sigma, \mathbb{P})
\]
is a {\em Gaussian field with mean $u$ and covariance operator $T$} if for each $\phi \in B^\ast$ the random variable $\xi(\phi)$ is normally distributed with mean $\phi(u)$ and covariance $\phi(T\phi)$, i.e.\ $\xi(\phi) \sim \gls{Normal}(\phi(u),\phi(T\phi))$. We denote such a field by $\xi \sim \gls{Normal}(u,T)$. When $u=0$, then we say $\xi$ is a {\em centred Gaussian field}.
\end{definition}

\begin{remark}[Notation]
	
	Consider a Gaussian field $\xi \sim \mathcal{N}(u,T)$, $\xi \colon B^\ast \to L^2(\mathcal{A},\Sigma,\mathbb{P})$ as in \cref{def:GPGeneralTheory}. The Gaussian field $\xi$ post-composed with evaluation at $\omega\in \mathcal{A}$ is a linear map $\xi(\cdot)(\omega) \colon B^\ast \to \R$, which is an element in the algebraic dual to $B^\ast$. Strictly speaking, the map $\omega \mapsto \xi(\cdot)(\omega)$ cannot be interpreted as a $B$-valued random variable
	because it takes values in the {\em algebraic dual} to $B^\ast$ but not necessarily in the topological dual $B^{\ast \ast} \cong B$ because $\xi \colon B^\ast \to L^2(\mathcal A,\Sigma, \mathbb{P})$ might not be bounded. However, $\omega \mapsto \xi(\cdot)(\omega)$ admits the interpretation as a {\em weak} $B$-valued random variable \cite[§17.4]{OwhadiScovel2019} and we say that $\xi$ is a {\em Gaussian field on $B$}.
	
	For $\phi \in B^\ast$ we define $\phi(\xi) := \xi (\phi)$, which is the notation used in 
	\cref{sec:MLSetting,sec:Experiment,sec:ConvergenceAnalysis}.
\end{remark}

\begin{theorem}[{{\cite[Thm.~17.4]{OwhadiScovel2019}}}]\label{thm:existenceGaussianField}
To any $u \in B$ and symmetric positive covariance operator $T$ a Gaussian field $\xi \sim \mathcal{N}(u,T)$ exists.
\end{theorem}

\begin{lemma}\label{lem:CovLemma}
Let $\xi \sim \mathcal{N}(u,T)$ for $u\in B$ and a positive symmetric operator $T\colon B^\ast \to B$. Then for $\phi, \psi \in B^\ast$ the covariance of $\xi(\phi)$ and $\xi(\psi)$ is given as
\[
\mathrm{cov}(\xi(\psi),\xi(\phi)) = \mathbb{E}[(\xi(\psi)-\psi(u))(\xi(\phi)-\phi(u))]
= \psi T \phi.\]
\end{lemma}

\begin{proof}
As covariances are invariant under shifts, without loss of generality we may assume $u=0$. We have
\begin{align*}
(\psi+\phi)T(\psi+\phi) &= \mathrm{cov}(\xi(\psi+\phi),\xi(\psi+\phi))
= \mathbb{E}[\xi(\psi+\phi)\xi(\psi+\phi)]\\
&=\mathbb{E}[\xi(\psi)\xi(\psi)+2\xi(\psi)\xi(\phi)+\xi(\phi)\xi(\phi)]\\
&=\psi T \psi +2 \mathrm{cov}(\xi(\psi),\xi(\phi)) + \phi T \phi
\end{align*}
It follows that $\mathrm{cov}(\xi(\psi),\xi(\phi)) = \psi T \phi$.
\end{proof}

\subsection{Conditional expectation and variance}\label{app:CondExpandVar}

Let $\xi \sim \mathcal{N}(u,T)$ be a Gaussian field with covariance operator $T$ and let $\gls{phi},\phi_1,\ldots,\phi_m \in B^\ast$. Let $\gls{Phi} = (\phi_1,\ldots,\phi_m)$ and denote $\xi(\Phi) = (\xi(\phi_1),\ldots,\xi(\phi_m))$, $\Phi(u) = (\phi_1(u),\ldots,\phi_m(u))$, $\gls{Theta} = (\phi_i T \phi_j)_{i,j=1}^m \in \R^{m \times m}$, $\Theta_0 = (\phi T \phi_j)_{j=1}^m \in \R^{m}$, $\Theta_0^0 = \phi T \phi$.
Using \cref{lem:CovLemma}, the joint distribution of $(\xi(\phi),\xi(\Phi)) \colon \mathcal{A} \to \R^{m+1}$ is given as
\[
\begin{pmatrix}
\xi(\phi)\\
\xi(\Phi)
\end{pmatrix}
\sim \mathcal{N} \left( \begin{pmatrix}
\phi (u)\\
\Phi (u)
\end{pmatrix} , \begin{pmatrix}
\Theta_0^0 & \Theta_0^\top\\
\Theta_0 & \Theta
\end{pmatrix} \right).
\]
We have $\xi(\Phi) -  \Phi(u) \in  \mathrm{range}(\Theta)$ almost surely \cite[Prop.~2.7]{Eaton2007}. Here $\mathrm{range}(\Theta)$ denotes the span of the columns of $\Theta$. Let $y \in \Phi(u) + \mathrm{range}(\Theta)$. Let the expression $\Theta^\dagger$ denote the Penrose pseudo-inverse of $\Theta$.
Using $\Theta_0^\top = \Theta_0^\top \Theta^\dagger \Theta$ \cite[Prop.2.16]{Eaton2007}, the two linear systems of equations
\begin{equation}\label{eq:ConditionalMeanLinSys}
	\Theta z = y-\Phi(u) \quad  \text{and} \quad \Theta Z = \Theta_0,
\end{equation}
are solvable.

\begin{proposition}[{{{\cite[Prop.~3.13]{Eaton2007}}}}]\label{prop:CondDistrGeneral}
The conditional distribution of $\xi(\phi)$ given $\xi(\Phi)=y$ is given as
\[
\xi(\phi) | \xi(\Phi)=y \sim \mathcal{N}(
\phi(u) + \Theta_0^\top \Theta^\dagger (y-\Phi(u)),
\Theta_0^0 - \Theta_0^\top \Theta^\dagger \Theta_0
).
\]
\end{proposition}

\begin{remark}\label{rem:AnySolution}
The expressions $\Theta^\dagger (y-\Phi(u))$ and $\Theta^\dagger \Theta_0$ denote the (column-wise) least square solutions to \eqref{eq:ConditionalMeanLinSys}. However, since $\mathrm{null}(\Theta) \subseteq \mathrm{null}(\Theta_0^\top)$ \cite[Prop.~2.16]{Eaton2007} any solution to \eqref{eq:ConditionalMeanLinSys} will yield the same conditional distribution.	
\end{remark}

\begin{remark}
As by the existence result (\cref{thm:existenceGaussianField}), the function $\phi \mapsto \xi(\phi) | \xi(\Phi)=y$ can be interpreted as a Gaussian field with $\overline{\xi} \sim \mathcal{N}(\overline{u},\overline{T})$ with mean
\[
\overline{u} = u+ (T \Phi)^\top \Theta^\dagger(y-\Phi(u)) \in B, \text{ with } T\Phi = (T\phi_1,\ldots,T\phi_m)
\]
and covariance given by the positive symmetric operator $\overline{T} \colon B^\ast \to B$ (in the sense of \cref{def:SymOperator})
\[
\overline{T} = T - (T \Phi)^\top \Theta^\dagger(T\Phi).
\]
For an interpretation of $\overline{\xi}$ as an orthogonal projection of $\xi$ and a measure theoretic discussion, we refer to \cite{OwhadiScovel2019}.
\end{remark}

The following statements are helpful to characterize conditional means of Gaussian fields by an extremization principle.

\begin{theorem}[{{{\cite[Thm.~12.5]{OwhadiScovel2019}}}}]\label{thm:extrLinInd}
Let $\phi_1,\ldots,\phi_m \in B^\ast$ be linearly independent. Define $\Theta \in \R^{m \times m}$ by it elements $\Theta_{i,j} = \phi_i (Q \phi_j)$. Denote $\Phi=\begin{pmatrix}
	 \phi_1, \ldots, \phi_m \end{pmatrix}^\top \in (B^\ast)^m$ and $Q \Phi = \begin{pmatrix}
Q \phi_1, \ldots, Q\phi_m \end{pmatrix}^\top \in B^m$.
Then $\Theta$ is invertible and for any $y \in \R^m$
\[
\Psi = y^\top \Theta^{-1} Q\Phi = \sum_{i,j=1}^{m} y_i (\Theta^{-1})_{i,j} Q \phi_j
\]
is the minimizer of the convex optimization problem
\[
\Psi = \arg \min_{\{\Psi \in B \, | \, \Phi(\Psi) = y\}} \| \Psi \|_B.
\]
\end{theorem}

We weaken the assumptions of \cref{thm:extrLinInd} slightly.

\begin{theorem}\label{thm:Extremizer}
Let $\phi_1,\ldots,\phi_m \in B^\ast$. Define $\Theta \in \R^{m \times m}$ by $\Theta_{i,j} = \phi_i (Q \phi_j)$. Denote $\Phi=\begin{pmatrix}
	\phi_1, \ldots, \phi_m \end{pmatrix}^\top \in (B^\ast)^m$ and $Q \Phi = \begin{pmatrix}
	Q \phi_1, \ldots, Q\phi_m \end{pmatrix}^\top \in B^m$.
For any $y \in \mathrm{range}(\Phi \colon B \to \R^m)$
\[
\Psi = y^\top \Theta^\dagger Q\Phi = \sum_{i,j=1}^{m} y_i (\Theta^\dagger)_{i,j} Q \phi_j
\]
is the minimizer of the convex optimization problem
\begin{equation}\label{eq:PsiAsMin}
\Psi = \arg \min_{\{\Psi \in B \, | \, \Phi(\Psi) = y\}} \| \Psi \|_B.
\end{equation}
\end{theorem}

\begin{proof}
	\mod{As $\Phi \colon B \to \R^m$ is a continuous linear operator, the set $\{\Psi \in B \, | \, \Phi(\Psi) = y\}$ is closed in $B$. Further, it is convex and non-empty. By the Hilbert projection theorem \cite[§12.3]{Rudin1991} the minimum in \eqref{eq:PsiAsMin} exists and is unique.}
	In preparation of the argument, we first prove the following \nameCref{lem:KernelProperty}\mod{.}
	
	\begin{lemma}\label{lem:KernelProperty}
	We have
	\[
	\ker (\Theta\colon \R^m \to \R^m) \subseteq \ker(\R^m \to B, x \mapsto x^\top Q\Phi).
	\]
	\end{lemma}

	\begin{proof}
		The proof is inspired by \cite[Prop.~2.16]{Eaton2007}. 
		Let $\phi \in B^\ast$. As the bijection $Q$ is positive symmetric, the following matrix is symmetric, positive semi-definite
		\[
		\Sigma = \begin{pmatrix}
			\phi Q \phi & \phi (Q \Phi)^\top\\
			\Phi^\top Q \phi & \Phi^\top Q \Phi
		\end{pmatrix} \in \R^{(m+1) \times (m +1)}
		\]
	Therefore, for any $x \in \ker \Theta = \ker \Phi^\top Q \Phi$, $\alpha,\beta \in \R$ we have
	\begin{align*}
		0 &\le \begin{pmatrix}
			\beta & \alpha x^\top
		\end{pmatrix} \Sigma \begin{pmatrix}
		\beta \\ \alpha x
		\end{pmatrix}
		= \beta^2 \phi Q \phi + 2 \alpha \beta \phi (Q\Phi)^\top x
	\end{align*}
	As this holds for all $\alpha, \beta \in \R$ we conclude $\phi ((Q \Phi)^\top x)=0$. 
	Since $B$ is a Hilbert space and $\phi ((Q \Phi)^\top x)=0$ holds for all $\phi \in B^\ast$ we conclude $(Q \Phi)^\top x =0$.
	\end{proof}

	Let $\{\tilde \phi_j\}_{j=1}^{\tilde m} \subset B^\ast$ be a basis for the linear span $\mathrm{span}\{\phi_j\}_{j=1}^{m}$ ($\tilde m \le m$).
	The basis elements define the vector $\tilde{\Phi} = \begin{pmatrix}
		\tilde \phi_1, & \ldots, & \tilde \phi_{\tilde{m}}
	\end{pmatrix}^\top \in (B^\ast)^{\tilde m}$.
	By linear independence of $\tilde{\phi}_j$, for $k=1,\ldots, m$ there exist unique $\alpha_{kj} \in \R$ ($j=1,\ldots, \tilde m$) with
	\[
	\phi_k = \sum_{j=1}^{\tilde m} \alpha_{kj} \tilde \phi_j.
	\]
	This defines a unique matrix $A = (\alpha_{ij}) \in \R^{m \times \tilde m}$ with $\Phi = A \tilde \Phi$ with linearly independent columns.
	Moreover, the matrix $\tilde \Theta \in \R^{\tilde m \times \tilde m}$ defined by $\tilde \Theta_{i,j} = \tilde \phi_i (Q(\tilde \phi_j))$ is invertible.
	
%
	
Since $\tilde{\Phi} \colon B \to \R^{\tilde m}$ is surjective,
\[
\mathrm{range}(\Phi \colon B \to \R^{m})
=\mathrm{range}(A \circ \tilde \Phi \colon B \to \R^{m})
= \mathrm{range}(A \colon \R^{\tilde m} \to \R^m).
\]
Let $y \in 	\mathrm{range}(\Phi) = \mathrm{range}(A)$.
Define $\tilde y = A^\dagger y$ and $\tilde z = \tilde{\Theta}^{-1} \tilde y$.
As $A$ has linearly independent columns, it is an isomorphism onto $\mathrm{range}(A)$ such that for any $\Psi \in B$ we have
\begin{align*}
(A \circ \tilde \Phi)(\Psi) = \Phi(\Psi) = y = A \tilde y \\
\iff \tilde \Phi(\Psi) = \tilde y.
\end{align*}
Thus, the following minima coincide
\[
 \arg \min_{\{\Psi \in B \, | \, \Phi(\Psi) = y\}} \| \Psi \|_B
 =\arg \min_{\{ \Psi \in B \, | \, \tilde \Phi(\Psi) = \tilde y\}} \| \Psi \|_B.
\]
\mod{Notice that the existence and uniqueness of the minima are guaranteed by the Hilbert projection theorem \cite[§12.3]{Rudin1991}.}
By \cref{thm:extrLinInd} this minimum coincides with $\tilde \Psi = \tilde z^\top Q \tilde \Phi$.
To complete the proof of the theorem, it remains to prove the following \nameCref{lem:PsiCoincide}.

\begin{lemma}\label{lem:PsiCoincide}
Consider the linear system $\Theta z = y$ for $z \in \R^m$. Then $\Theta z = y$ is solvable and for any solution $z$
\[
\Psi = z^\top Q \Phi \quad \text{and}\quad  \tilde \Psi = \tilde z^\top Q \tilde \Phi
\]
coincide.
\end{lemma}
	
	\begin{proof}
		 Let $\tilde z$ be the solution to $\tilde \Theta \tilde z = \tilde y$ and set $\tilde \Psi = \tilde z^\top Q \tilde \Phi$. Using linearity of $Q$, we have
	\begin{align*}
	\tilde \Psi = \tilde z^\top Q \tilde \Phi
	= \tilde z^\top Q A^\dagger \Phi
	= ((A^\dagger)^\top \tilde z)^\top Q \Phi
	= \overline z^\top Q \Phi
	\end{align*}
	with $\overline z := (A^\dagger)^\top \tilde z$. We have
	\[
	A^\dagger \Theta \overline z = A^\dagger \Theta (A^\dagger)^\top \tilde z
	= \tilde \Theta \tilde z
	= \tilde y 
	=A^\dagger y.
	\]
	As $A^\dagger A=\mathrm{Id}_{\tilde m}$, the restriction $A^\dagger|_{\mathrm{range}(A)} \colon \mathrm{range}(A) \to \R^{\tilde m}$ is an isomorphism. Therefore, as $y \in \mathrm{range}(A)$ it follows that $\overline z = (A^\dagger)^\top \tilde z$ solves the linear system
	\begin{equation}\label{eq:ThetaSolinLem}
	\Theta z = y.
	\end{equation}
	For any other solution $z \in \R^m$ of \eqref{eq:ThetaSolinLem} it holds that
	\[
	z - \overline{z} \in \ker \Theta \subseteq \ker( (\Theta\Phi)^\top \colon \R^m \to \R),
	\]
	where the inclusion holds by \cref{lem:KernelProperty}.
	Therefore, $\tilde{\Psi}  = \tilde z^\top Q \Phi = \overline{z}^\top Q \Phi = z^\top Q \Phi = \Psi$.
	\end{proof}
	
This completes the proof of \cref{thm:Extremizer}.

\end{proof}

\subsection{Applicability of \cref{prop:CondDistrGeneral} in \mod{the setting of} \cref{sec:MLSetting}}\label{sec:ApplyPropCondDist}

\mod{
\Cref{prop:CondDistrGeneral} provides an expressions for the mean and covariance of conditional distributions of Gaussian fields. In case the prior has zero mean, \cref{prop:CondDistrGeneral} requires $y \in \mathrm{range}(\Theta)$. 
To apply these results in the setting of}
\cref{sec:LasCondGP}, we need to verify \mod{that for $y_b^{\mod{(}M\mod{)}}$ and $\Theta$ of \cref{sec:LasCondGP} it holds that} $y_b^{\mod{(}M\mod{)}} \in \mathrm{range}(\Theta)$.

\begin{proposition}\label{prop:CondDistPropLCheck}
Employing notation of \cref{sec:MLSetting}, \cref{assumption:MLSetupAssumption} implies $y_b^{\mod{(}M\mod{)}} \in \mathrm{range}(\Theta)$.
\end{proposition}

\begin{proof}
Denote the components of $\Phi_b^{\mod{(}M\mod{)}}$ by $\phi_1,\ldots,\phi_{\overline{M}} \in U^\ast$, where $\overline{M} = Md+d+1$. 
Let $\{\tilde \phi_j\}_{j=1}^{\tilde M} \subset U^\ast$ be a basis for the linear span $\mathrm{span}\{\phi_j\}_{j=1}^{\overline{M}}$ ($\tilde M \le \overline{M}$).
The basis elements define the vector $\tilde{\Phi} = \begin{pmatrix}
\tilde \phi_1, & \ldots, & \tilde \phi_{\tilde{M}}
\end{pmatrix}^\top \in (U^\ast)^{\tilde M}$.
By the linear independence of $\tilde{\phi}_j$, for $k=1,\ldots, \overline M$ there exist unique $\alpha_{kj} \in \R$ ($j=1,\ldots, \tilde M$) with
\[
\phi_k = \sum_{j=1}^{\tilde M} \alpha_{kj} \tilde \phi_j.
\]
This defines a unique matrix $A = (\alpha_{ij}) \in \R^{\overline{M} \times \tilde M}$ with $\Phi_b^{\mod{(}M\mod{)}} = A \tilde \Phi$ with linearly independent columns.
Define $\tilde \Theta \in \R^{\tilde M \times \tilde M}$ by $\tilde \Theta_{i,j} = \tilde \phi_i (\mathcal{K}(\tilde \phi_j))$, \mod{where $\gls{MathcalKappa}$ is the quadratic identification of $U^\ast$ and $U$ induced by the kernel $K$, see \eqref{eq:MathcalKappa}.}

Recall that $\Theta_{i,j} = \phi_i (\mathcal{K}(\phi_j))$ defines $\Theta \in \R^{\overline{M}\times \overline{M}}$.
\mod{To complete the proof, we need the following \namecref{lem:ThetaThetaTilde}.}

\begin{lemma}\label{lem:ThetaThetaTilde}
We have
\[
\Theta = A \tilde \Theta A^\top.
\]
\end{lemma}

\begin{proof}
	Using linearity of $\mathcal{K}\colon U^\ast \to U$,
	\begin{align*}
\Theta_{i,j} &= \phi_i (\mathcal{K}(\phi_j))
= \sum_{k=1}^{\tilde M} \alpha_{ik} \tilde \phi_k \left(\mathcal K \left( \sum_{s=1}^{\tilde M} \alpha_{js} \tilde \phi_s \right) \right)\\
&= \sum_{k,s=1}^{\tilde M} \alpha_{ik} \alpha_{js} \tilde \phi_k ( \mathcal K ( \tilde \phi_s))
=\sum_{k,s=1}^{\tilde M} \alpha_{ik} \tilde{\Theta}_{k,s} \alpha_{js}.
	\end{align*}
\end{proof}

The matrix $\tilde \Theta$ is invertible by construction. (Indeed, we could have chosen $\tilde \phi_j$ such that $\tilde \Theta$ is the identity matrix.) Moreover, $A$ is injective such that $A^\top \colon \R^{\overline{M}} \to \R^{\tilde M}$ is surjective.
Therefore, by \cref{lem:ThetaThetaTilde}, $\mathrm{range}(\Theta) = \mathrm{range}(A)$.
Viewing $\Phi_b^{\mod{(}M\mod{)}}\colon U \to \R^{\overline{M}}$, $\tilde \Phi \colon U \to \R^{\tilde M}$, $A \colon \R^{\tilde M} \to \R^{\overline M}$ as linear maps,
\[
\mathrm{range}(\Phi_b^{\mod{(}M\mod{)}} \colon U \to \R^{\overline{M}})
=\mathrm{range}(A \circ \tilde \Phi \colon U \to \R^{\overline{M}})
\subset \mathrm{range}(A) = \mathrm{range}(\Theta).
\]

\Cref{assumption:MLSetupAssumption} implies $y_{b}^{\mod{(}M\mod{)}} \in \mathrm{range}(\Phi_b^{\mod{(}M\mod{)}})$. Therefore, $y_{b}^{\mod{(}M\mod{)}} \in \mathrm{range}(\Theta)$. \mod{This completes the proof of \cref{prop:CondDistPropLCheck}.}

\end{proof}

In the setting of discrete Lagrangians, an application of \cref{prop:CondDistrGeneral} is justified as by the following \nameCref{prop:CondDistPropLdCheck}.

\begin{proposition}\label{prop:CondDistPropLdCheck}
Employing notation of \cref{dec:GPLd}, \cref{assumption:MLSetupAssumptionLd} implies $y_b^{\mod{(}M\mod{)}} \in \mathrm{range}(\Theta)$.
\end{proposition}

\begin{proof}
The proof follows in complete analogy to \mod{the proof of} \cref{prop:CondDistPropLCheck}.
\end{proof}

%% file: Appendix_ODE.tex
\section{Alternative regularisation}\label{app:SymplVolNormalise}

The following proposition justifies an alternative regularisation strategy. As it involves non-linear conditions, we prefer the regularisation strategy presented in the main body of the document. However, it is presented here for comparison with regularisation strategies for learning of Lagrangian densities using neural networks \cite{DLNNPDE}.

\begin{proposition}\label{prop:NormaliseLNonlinear}
	Let $\overline{x}_b = (x_b,\dot x_b) \in T\R^d \cong \R^d \times \R^d$ and $\mathring L$ a Lagrangian with $\frac{\p \mathring L}{\p \dot x \p \dot x}(\overline{x}_b)$ non-degenerate.
	Let $c_b \in \R$, $p_b \in \R^d$, $c_\omega >0$. There exists a Lagrangian $L$ such that $L$ is equivalent to $\mathring{L}$ and
	\begin{equation}\label{eq:normLCondNonlinear}
		L(\overline{x}_b)=c_b,
		\quad 
		\mathrm{Mm}(L)(\overline{x}_b) = \frac{\p L}{\p \dot x}(\overline{x}_b) = p_b,
		\quad
		N_\omega(L)(\overline{x}_b) = \left|\det\left( \frac{\p^2 L}{\p \dot x \p \dot x}(\overline{x}_b)\right)\right| = c_\omega.
	\end{equation}
	
\end{proposition}

\begin{proof}
	Let $\mathring{c_b} = \mathring L(\overline{x}_b)$, $\mathring{p_b}=\mathrm{Mm}(\mathring L)(\overline{x}_b)$, $\mathring{c_\omega} = N_\omega(\mathring L)(\overline{x}_b)$.
	The quantity $\mathring c_\omega$ is not zero since $\frac{\p \mathring L}{\p \dot x \p \dot x}(\overline{x}_b)$ is non-degenerate. We set
	\[
	\rho = \sqrt[d]{\left|\frac{c_\omega}{\mathring{c_\omega}}\right|},
	\quad
	F(x) = x^\top (p_b- \rho \mathring{p_b}),
	\quad 
	c = c_b - \dot{x}_b^\top (p_b- \rho \mathring{p_b})-\rho \mathring{c_b}.
	\]
	Now the Lagrangian $L = \rho \mathring{L} + \d_t F + c$ is equivalent to $\mathring{L}$ and fulfils \eqref{eq:normLCond}.
\end{proof}

The condition $	N_\omega(L)(\overline{x}_b) = c_\omega > 0$ may be compared to the regularisation strategies for training Lagrangians modelled as neural networks in \cite{DLNNPDE}:
denoting observation data by $\hat x^{(j)} = (x^{(j)},\dot x^{(j)} , \ddot x^{(j)})$, in \cite{DLNNPDE} (transferred to our continuous ode setting) parametrises $L$ as a neural network and considers the minimisation of a loss function function $\ell = \ell_{\mathrm{data}} + \ell_{\mathrm{reg}}$ with data consistency term
\[
\ell_{\mathrm{data}} = \sum_j \| \mathrm{EL}(L)(\hat{x}^{(j)})\|^2
\]
and with regularisation term $\ell_{\mathrm{reg}}$ that maximises the regularity of the Lagrangian at data points $\hat x^{(j)} = (x^{(j)},\dot x^{(j)} , \ddot x^{(j)})$
\[
\ell_{\mathrm{reg}} = \sum  \left\| \left(\frac{\p^2 L}{\p \dot x \p \dot x}(x^{(j)},\dot x^{(j)}) \right)^{-1} \right\|.
\]

The corresponding statement for discrete Lagrangians is as follows.

\begin{proposition}\label{prop:NormaliseLdNonlinear}
	Let $\overline{x}_b = (x_{0b}, x_{1b}) \in  \R^d \times \R^d$ and $\mathring L_d$ a discrete Lagrangian with $\mathrm{Mm}^-(\overline{x}_b)$ non-degenerate.
	Let $c_b \in \R$, $p_b \in \R^d$, $c_\omega >0$. There exists a discrete Lagrangian $L_d$ such that $L_d$ is equivalent to $\mathring{L}_d$ and
	\begin{equation}\label{eq:normLdCondNonlinear}
		L_d(\overline{x}_b)=c_b,
		\quad 
		\mathrm{Mm}^-(L_d)(\overline{x}_b) = p_b,
		\quad
		N_\omega^-(L_d)(\overline{x}_b) = \left|\det\left( \frac{\p^2 L_d}{\p x_0 \p x_1}(\overline{x}_b)\right)\right| = c_\omega.
	\end{equation}
	
\end{proposition}

\begin{proof}
	Let $\mathring{c_b} = \mathring L_d(\overline{x}_b)$, $\mathring{p_b}=\mathrm{Mm}^-(\mathring L_d)(\overline{x}_b)$, $\mathring{c_\omega} = N_\omega^-(\mathring L_d)(\overline{x}_b)$.
	The quantity $\mathring c_\omega$ is not zero since $\frac{\p \mathring L_d}{\p x_0 \p x_1}(\overline{x}_b)$ is non-degenerate. We set
	\[
	\rho = \sqrt[d]{\left|\frac{c_\omega}{\mathring{c_\omega}}\right|},
	\quad
	F(x) = x^\top (p_b- \rho \mathring{p_b}),
	\quad 
	c = c_b -\rho \mathring{c_b}  - (x_{1b}-x_{0b})^\top (p_b- \rho \mathring{p_b}).
	\]
	Now the Lagrangian $L_d = \rho \mathring{L}_d + \Delta_t F + c$ is equivalent to $L_d$ and fulfils \eqref{eq:normLdCondNonlinear}.
\end{proof}

Again, the condition $	N_\omega^-(L)(\overline{x}_b) = c_\omega > 0$ may be compared to the regularisation strategies for training discrete Lagrangians modelled as neural networks in \cite{DLNNPDE}:
denoting observation data by $\hat x^{(j)} = (x^{(j)}_0,x_1^{(j)} , x_2^{(j)})$, in \cite{DLNNPDE} (when transferred to our discrete ode setting) parametrises $L_d$ as a neural network and considers the minimisation of a loss function function $\ell = \ell_{\mathrm{data}} + \ell_{\mathrm{reg}}$ with data consistency term
\[
\ell_{\mathrm{data}} = \sum_j \| \mathrm{DEL}(L_d)(\hat{x}^{(j)})\|^2
\]
and with regularisation term $\ell_{\mathrm{reg}}$ that maximises the regularity of the Lagrangian at data points $\hat x^{(j)} = (x^{(j)}_0,x_1^{(j)} , x_2^{(j)})$:
\[
\ell_{\mathrm{reg}} = \sum  \left\| \left(\frac{\p^2 L}{\p x_0 \p x_1}(x_0^{(j)}, x_1^{(j)}) \right)^{-1} \right\|.
\]

\section{Derivation of symplectic structure induced by discrete Lagrangians}\label{app:SymplVolLd}
Denote the coordinate of the domain of definition $\R^d \times \R^d$ of a discrete Lagrangian $L_d$ by $(x_0,x_1)$.
Consider the two discrete Legendre transforms $\mod{\mathrm{Leg}}^\pm \colon \R^d \times \R^d \to T^\ast \R^d$ \cite{MarsdenWestVariationalIntegrators} with
\begin{align*}
	\mod{\mathrm{Leg}}^-(x_0,x_1) = (x_0,-\frac{\p L}{\p x_0}(x_0,x_1))
\qquad	
	\mod{\mathrm{Leg}}^+(x_0,x_1) = (x_1,\frac{\p L}{\p x_1}(x_0,x_1)).
\end{align*}

When we pullback the canonical symplectic structure $\sum_{k=1}^d \d q^k \wedge \d p_k$ on $T^\ast \R^d$ to the discrete phase space $\R^d \times \R^d$ with $\mod{\mathrm{Leg}}^\pm$ we obtain

\begin{align*}
	\mathrm{Sympl}^-(L_d)
	&= \sum_{s=1}^{d}
	\d x_0^s \wedge \d \left( - \frac{\p L_d}{\p x^s_0} \right)
	= \sum_{r,s=1}^{d}
	- \frac{\p^2 L_d}{\p x_0^s \p x_0^r} \d x_0^s \wedge \d x_0^r
	- \frac{\p^2 L_d}{\p x_0^s \p x_1^r} \d x_0^s \wedge \d x_1^r\\
	&= \sum_{r,s=1}^{d}
	- \frac{\p^2 L_d}{\p x_0^s \p x_1^r} \d x_0^s \wedge \d x_1^r\\
	\mathrm{Sympl}^+(L_d)
	&= \sum_{s=1}^{d}
	\d x_1^s \wedge \d \left( \frac{\p L_d}{\p x^s_1} \right)
	= \sum_{r,s=1}^{d}
	\frac{\p^2 L_d}{\p x_1^s \p x_0^r} \d x_1^s \wedge \d x_0^r
	+\frac{\p^2 L_d}{\p x_1^s \p x_1^r} \d x_1^s \wedge \d x_1^r\\
	&= \sum_{r,s=1}^{d}
	\frac{\p^2 L_d}{\p x_1^s \p x_0^r} \d x_1^s \wedge \d x_0^r
\end{align*}

We see $\mathrm{Sympl}^-(L_d)=\mathrm{Sympl}^+(L_d)$.

The 2-form corresponds to the notion of a {\em discrete Lagrangian symplectic form} in \cite[§1.3.2]{MarsdenWestVariationalIntegrators}.